\documentclass[10pt,reqno]{amsart}
\usepackage[dvipsnames]{xcolor}
\usepackage{mathabx}
\usepackage[margin = 1.2in]{geometry}
\usepackage{graphicx, listings} %
\usepackage{subcaption}
\newcommand{\mc}[1]{\mathcal{#1}}
\newcommand{\Ka}[2]{\mathcal{K}_D^{#1,#2,*}}
\def \vp{\varphi}

\usepackage{caption, braket, float, pdfpages}
\usepackage[inkscapearea=page]{svg}
\usepackage{adjustbox}
\usepackage[makeroom]{cancel}
\usepackage{pgffor, mathdots}
\usepackage{bbm}
\usepackage{bm}
\usepackage{amsmath, amssymb, amsthm}
\usepackage{tikz, xparse}
\usetikzlibrary{decorations.pathreplacing, positioning}
\usepackage[textsize=tiny]{todonotes}
\usepackage{algorithm, algpseudocode}
\usetikzlibrary{arrows.meta,calc}

\numberwithin{equation}{section}

\usepackage{aliascnt}
\usepackage[
  colorlinks=true,
  linkcolor=BrickRed,
  citecolor=RoyalBlue,
  urlcolor=RoyalBlue
]{hyperref}
\usepackage{cleveref}

\newtheorem{theorem}{Theorem}[section]

\newaliascnt{lemma}{theorem}
\newtheorem{lemma}[lemma]{Lemma}
\aliascntresetthe{lemma}

\newaliascnt{proposition}{theorem}
\newtheorem{proposition}[proposition]{Proposition}
\aliascntresetthe{proposition}

\newaliascnt{corollary}{theorem}
\newtheorem{corollary}[corollary]{Corollary}
\aliascntresetthe{corollary}

\newaliascnt{definition}{theorem}
\newtheorem{definition}[definition]{Definition}
\aliascntresetthe{definition}

\newaliascnt{example}{theorem}

\aliascntresetthe{example}

\newaliascnt{remark}{theorem}
\newtheorem{remark}[remark]{Remark}
\aliascntresetthe{remark}

\newaliascnt{assumption}{theorem}

\aliascntresetthe{assumption}

\crefname{theorem}{Theorem}{Theorems}
\crefname{lemma}{Lemma}{Lemmas}
\crefname{proposition}{Proposition}{Propositions}
\crefname{corollary}{Corollary}{Corollaries}
\crefname{definition}{Definition}{Definitions}
\crefname{example}{Example}{Examples}
\crefname{remark}{Remark}{Remarks}
\crefname{assumption}{Assumption}{Assumptions}

\newcommand{\vect}[1]{\boldsymbol{#1}}

\newcommand{\outercdots}{\textcolor{Cerulean!90}{\cdots}}
\newcommand{\outervdots}{\textcolor{Cerulean!90}{\vdots}}
\newcommand{\outerddots}{\textcolor{Cerulean!90}{\ddots}}
\newcommand{\outeriddots}{\textcolor{Cerulean!90}{\iddots}}

\newcommand{\innercdots}{\textcolor{red!55}{\cdots}}
\newcommand{\innervdots}{\textcolor{red!55}{\vdots}}
\newcommand{\innerddots}{\textcolor{red!55}{\ddots}}
\newcommand{\inneriddots}{\textcolor{red!55}{\iddots}}

\makeatletter
\newcommand\ackname{Acknowledgements}
\if@titlepage
  \newenvironment{acknowledgements}{%
      \titlepage
      \null\vfil
      \@beginparpenalty\@lowpenalty
      \begin{center}%
        \bfseries \ackname
        \@endparpenalty\@M
      \end{center}}%
     {\par\vfil\null\endtitlepage}
\else
  \newenvironment{acknowledgements}{%
      \if@twocolumn
        \section*{\abstractname}%
      \else
        \small
        \begin{center}%
          {\bfseries \ackname\vspace{-.5em}\vspace{\z@}}%
        \end{center}%
        \quotation
      \fi}
      {\if@twocolumn\else\endquotation\fi}
\fi

\makeatletter
\newcommand\codename{Code Availability}
\if@titlepage
  \newenvironment{code}{%
      \titlepage
      \null\vfil
      \@beginparpenalty\@lowpenalty
      \begin{center}%
        \bfseries \codename
        \@endparpenalty\@M
      \end{center}}%
     {\par\vfil\null\endtitlepage}
\else
  \newenvironment{code}{%
      \if@twocolumn
        \section*{\abstractname}%
      \else
        \small
        \begin{center}%
          {\bfseries \codename\vspace{-.5em}\vspace{\z@}}%
        \end{center}%
        \quotation
      \fi}
      {\if@twocolumn\else\endquotation\fi}
\fi

\def \i{\mathrm{i}}

\makeatother

\title[Waveguiding in systems of high-contrast resonators]{Waveguiding in systems of high-contrast resonators: Theory and fast computations}
\author[H. Ammari]{Habib Ammari}
	\address{ETH Z\"urich, Department of Mathematics, Rämistrasse 101, 8092 Z\"urich, Switzerland}
	\email{habib.ammari@math.ethz.ch}

 \author[B. Li]{Bowen Li}
 \address{Department of Mathematics,  City University of Hong Kong, Kowloon Tong, Hong Kong SAR}
 \email{bowen.li@cityu.edu.hk}

	\author[B. Miao]{Borui Miao}
	\address{Yau Mathematical Sciences Center, Tsinghua University, 100084 Beijing, China}
	\email{mbr@mail.tsinghua.edu.cn}
	\author[J. Qiu]{Jiayu Qiu}
	\address{ETH Z\"urich, Department of Mathematics, Rämistrasse 101, 8092 Z\"urich, Switzerland}
	\email{jiayu.qiu@sam.math.ethz.ch}
\author[L. Vrabac]{Lara Vrabac}
\address{ETH Z\"urich, Department of Mathematics, Rämistrasse 101, 8092 Z\"urich, Switzerland}
	\email{lara.vrabac@sam.math.ethz.ch}
\date{}



\begin{document}

\begin{abstract} In this work, we study guided modes in systems of high-contrast resonators near nonzero interior Neumann frequencies, beyond the subwavelength regime. In the regular exterior regime, where the exterior Dirichlet problem is well-posed at the reference wavenumber, we introduce an infinite-dimensional frequency-dependent capacitance operator obtained by compressing the exterior Helmholtz Dirichlet-to-Neumann map to the traces of the interior resonant Neumann eigenspaces. 
We prove the norm-resolvent convergence of the continuous problem to this discrete effective operator as the contrast $\delta\to0$, and derive first-order asymptotic formulas for compact-defect frequencies and line-defect band functions. We then establish exponential off-diagonal decay of the capacitance coefficients by a Combes--Thomas argument, yielding an exponentially accurate truncation of the discrete operator, and show that its retained coefficients can be computed from local Helmholtz problems. At the physical frequency, this local approximation converges exponentially under a uniform stability assumption for the growing finite-cluster problems. The stability assumption can be removed by introducing a vanishing complex absorption together with a Hermitian symmetrization. In particular, an absorption parameter of order $\sqrt\delta$, together with interaction truncation and patch radii of order $|\log\delta|$, suffices to preserve the $\mathcal O(\delta^2)$ accuracy of the first-order high-contrast expansion of the defect eigenfrequencies, yielding a fast computational method. Numerical experiments for dipole and quadrupole resonances illustrate the accuracy, exponential locality, and applicability of the discrete model to straight and bent waveguides generated by material or geometric detuning. 
\end{abstract}

\maketitle

\section{Introduction}

Periodic systems of high-contrast resonators provide a framework for analyzing and controlling wave propagation through the hybridization of localized resonances supported by individual resonators; see~\cite{cbms,ammariSubwavelengthGuidedModes2021,erik2019,anderson} for the resulting localized and guided waves in the subwavelength regime, or equivalently, the low-frequency range. In this case, the relevant resonance of each connected resonator originates from the constant Neumann eigenmode, and the leading-order spectral behavior of the full system is governed by a discrete capacitance matrix for finite configurations and by a capacitance operator for infinite systems. This reduction is particularly useful for analyzing defect modes, line-defect guided modes, and topologically protected interface and edge modes \cite{ammari2020robust,ammari2024functional,ammari2019topologically,jiayu2026a,jiayu2026b}.

Waveguiding beyond the subwavelength regime requires a different effective description. In fact, it was recently shown in~\cite{li2025high,ammariFrequencydependentCapacitanceMatrix2026} that the relevant scattering resonant frequencies are close to nonzero interior Neumann eigenfrequencies of the individual resonators. The associated interior eigenspaces may then have dimensions greater than one, and the exterior fields satisfy a Helmholtz equation rather than a harmonic equation. Therefore, the effective discrete model could have several degrees of freedom per resonator and depends explicitly on the interior Neumann eigenfrequencies. In addition, radiation effects can contribute at leading order.

More precisely, in~\cite{ammariFrequencydependentCapacitanceMatrix2026}, we introduced the frequency-dependent capacitance formalism to approximate the non-subwavelength resonances, called Fabry–P\'erot resonances, in finite and periodic systems of high-contrast resonators. In the regular case, where the leading-order resonant frequency corresponds to an interior Neumann frequency, but the associated exterior wave number squared belongs to the resolvent of the 
exterior Dirichlet problem, \emph{i.e.}, under assumptions \eqref{gapassump1}, \eqref{gapassump1uniform}, and \eqref{gapassump2equiv}, the leading shifts of the resonant frequencies are governed by a frequency-dependent capacitance matrix, obtained by compressing the exterior Dirichlet-to-Neumann (DtN) map to the traces of the interior resonant Neumann eigenspace. For periodic structures, the eigenvalues of the quasi-periodic capacitance matrices then give the first-order asymptotics of the Bloch band functions; see~\cite{ammariFrequencydependentCapacitanceMatrix2026}.

On the other hand, a resolvent-convergence and patch-approximation framework for general, possibly nonperiodic, systems in the subwavelength regime is developed in~\cite{ammari2026resolvent_con}. This framework rigorously justifies the reduction of the continuous high-contrast scattering resonance problem to an infinite-dimensional capacitance operator and introduces a local patch method for its efficient computation. This patch method replaces each global exterior problem that defines a column of the capacitance operator by a bounded local problem centered on the corresponding resonator, yielding an exponentially small error in the operator norm.

In this paper, we study waveguiding in systems of high-contrast resonators near nonzero interior Neumann frequencies. Our main objective is not only to derive an effective discrete model beyond the subwavelength regime, but also to identify the mathematical structure that makes this effective problem locally computable. The central mechanism is the exterior spectral gap. It first permits the reduction of the continuous high-contrast problem to a frequency-dependent discrete operator acting on the interior resonant modes, and then forces the interaction coefficients of this operator to decay exponentially in space. As a consequence, the effective operator is exponentially compressible and can be approximated by local computations with a quantitatively controlled error. 

Our first main result is an operator-norm resolvent convergence of the continuous high-contrast defect eigenvalue problem to the frequency-dependent capacitance operator $\mathcal{C}(\omega_0)$; see \cref{def:freq-cap-operator} and \cref{thm_cont_to_disc_resolvent_converge}.
The active degrees of freedom are the interior Neumann eigenspaces at the reference frequency $\omega_0$, which may have dimension larger than one, and hence the effective model is in general a block operator.  This reduction yields the corresponding spectral asymptotics. Specifically, for a compact defect, if $\lambda$ is an eigenvalue of $\mathcal C(\omega_0)$, then the associated defect eigenfrequencies satisfy \[ \omega_\delta = \omega_0+\delta\lambda+ \mathcal{O}(\delta^2). \] For a straight line defect, a partial Floquet--Bloch transform gives the quasi-periodic capacitance matrices $\mathcal C^\alpha(\omega_0)$, whose eigenvalues $\lambda_j^\alpha$ determine the guided band functions by \[ \omega_j^\alpha(\delta) = \omega_0+\delta\lambda_j^\alpha + \mathcal{O}(\delta^2), \] uniformly in $\alpha$. See \cref{cor:compact-asymptotics,cor:line-asymptotics} respectively. We also show that these defect eigenvalues and guided bands lie in a spectral gap of the unperturbed bulk operator for sufficiently small $\delta$.

Our second main result is the exponential localization of this effective discrete operator $\mathcal{C}(\omega_0)$ for the defect eigenvalue problem (\cref{thm:decay}). Using a Combes--Thomas estimate for the exterior Dirichlet resolvent, we prove that the matrix coefficients of $\mathcal{C}(\omega_0)$ satisfy, for some $\beta > 0$,  \[ \left| \mathcal{C}_{(\vect{j},\ell),(\vect{j}',\ell')}(\omega_0) \right| \lesssim  e^{-\beta |\vect{j} - \vect{j}'|}. \] Thus, although the effective interaction is generally infinite-range, it is exponentially compressible in the sense that $\mathcal{C}(\omega_0)=\mathcal{C}_N^{\mathrm{tr}}(\omega_0)+\mathcal{O}(e^{-\mu N})$, where $\mathcal{C}_N^{\mathrm{tr}}(\omega_0)$ is obtained by retaining only interactions of range at most $N$ (\cref{prop:exact-truncation}). Equivalently, an interaction radius $N=\mathcal{O}(\log\varepsilon^{-1})$ suffices to achieve operator accuracy $\mathcal{O}(\varepsilon)$. More generally, if the defect set has intrinsic dimension $d_{\mathrm{act}}$, then only $\mathcal{O}((\log\varepsilon^{-1})^{d_{\mathrm{act}}})$ interaction blocks need to be retained per resonant site. In particular, for a finite line-like defect containing $N_{\mathrm{act}}$ resonators, the compressed effective matrix has only $\mathcal{O}(N_{\mathrm{act}}\log\varepsilon^{-1})$ nonzero blocks, compared with the $\mathcal{O}(N_{\mathrm{act}}^2)$ blocks of a dense matrix. 

Our third main result concerns the fast computation of this compressed operator. Recall that each coefficient of $\mathcal C(\omega_0)$ is defined through a global exterior Helmholtz problem \eqref{eq:def-Vjl-updated}. We show that these coefficients can instead be approximated by solving finite local problems centered at the corresponding source resonators. More precisely, at the physical real frequency $\omega_0$, the resulting patch approximation converges exponentially, provided that the growing local Helmholtz problems satisfy a uniform resolvent bound; see \cref{lem:dipatch}. This is the natural analogue of the subwavelength patch approximation developed in \cite{ammari2026resolvent_con}, but with an important difference: at non-subwavelength frequencies, the maximum-principle argument available for harmonic patch problems is no longer applicable, and growing real-frequency clusters may approach scattering resonances.
To remove this additional real-frequency stability assumption, we introduce a vanishing complex absorption
$$
z_\gamma=\omega_0^2+\mathrm i\gamma,\qquad \gamma>0.
$$
The complex shift moves the finite-cluster problem uniformly away from the real spectrum and yields an $\mathcal O(\gamma^{-1})$ resolvent bound independent of the cluster size. Using Schwarz reflection symmetry together with Hermitian symmetrization, we cancel the first-order regularization bias and obtain the unconditional estimate in \cref{thm:stabilized-finite-cluster}:
$$
\mathcal C(\omega_0) = \widehat{\mathcal C}_{N,R,\gamma} + \mathcal O\left(\gamma^2+e^{-\beta N}+\gamma^{-1}e^{-\mu R}\right),
$$
where $N$ is the interaction truncation radius and $R$ is the local finite-cluster radius. In particular, choosing $\gamma_R=e^{-\mu R/3}$ and $N\asymp R$ immediately yields an exponentially convergent stabilized local approximation without any uniform finite-cluster stability assumption. Moreover, by taking $\gamma=\sqrt{\delta}$ and $N,R=\mathcal O(|\log\delta|)$, we show that the finite-cluster approximation preserves the $\mathcal O(\delta^2)$ accuracy of the first-order effective model (\cref{cor:delta-log-patch}):
$$
\omega_{\delta,j} = \omega_0 + \delta\widehat\lambda_{j,N,R,\sqrt{\delta}} +
\mathcal O(\delta^2),
$$
with the analogous estimate holding uniformly in the quasi-momentum $\alpha$ for line defects.

The resulting framework does not require global periodicity and is therefore particularly suited to nonperiodic configurations such as bent waveguides. We illustrate the theory numerically for dipole and quadrupole interior Neumann resonances generated by both material and geometric detuning. For straight line defects, the effective band functions agree closely with those obtained from the full characteristic-value formulation, while the real-space interaction blocks exhibit the predicted exponential decay and finite-patch convergence. We then apply the same local construction to sharply bent waveguides and reconstruct approximate dipole and quadrupole modes. Compared to existing approaches such as the supercell method or the method based on the DtN map \cite{sofiane,supercell2,supercell1,sonia}, our approach significantly reduces computational costs by using the locality of the effective operator and local patch computations. Band structure calculations by the  aforementioned methods involve high computational costs due to the fine mesh required by the high-contrast regime and the size of the computational domain. 

The paper is organized as follows. \Cref{sec:problem-formulation} introduces material and shape defects in a unified continuous formulation and derives the exterior DtN operator pencil. \Cref{sec:capacitance-spectral} introduces the full frequency-dependent capacitance operator, its partial Floquet representation for a straight line defect, the norm-resolvent and spectral asymptotics, and the exponential decay of its off-diagonal coefficients. \Cref{sec:pathch} develops the local computational theory, including exponential operator compression, real-frequency patch approximations, complex-frequency stabilization, and the resulting spectral-error and complexity estimates. \Cref{sec:numerical} presents numerical experiments and \Cref{sec:conclusion} makes concluding remarks. The appendix collects the boundary-integral formulations used in the computations.

\section{High contrast crystal with defects}
\label{sec:problem-formulation}

\subsection{The defect resonance problem}
\label{subsec:continuous-problem}

In this section, we formulate a unified whole-space spectral problem for high-contrast resonator systems with compactly supported or extended defects. In particular, the formulation established in this section does not require periodicity of the underlying structure.

Let $d=2,3$, and let $\Lambda$ be the lattice generated by linearly independent vectors $e_1,\dots,e_d\in\mathbb R^d$:
\begin{equation*}
\Lambda := \left\{ j_1e_1+\cdots+j_de_d:\ j_1,\dots,j_d\in\mathbb Z \right\}.
\end{equation*}
When convenient, we identify the lattice vector $\vect{j}=j_1e_1+\cdots+j_de_d$ with its coordinate tuple $(j_1,\dots,j_d)$. We use the fundamental cell
\begin{equation*}
Y := \left\{ c_1e_1+\cdots+c_de_d:\ 0\leq c_1,\dots,c_d<1 \right\},
\end{equation*}
and let $D,\widetilde D\Subset Y$ be bounded connected $C^{1,1}$-domains. In particular, the domain $D$ represents the unperturbed reference resonator in the unit cell, while $\widetilde D$ denotes its perturbed counterpart with a deformed boundary. On the other hand, we denote the unperturbed and perturbed material parameters by $v_b,\widetilde v_b>0$. Let $\mathcal J_{\mathrm{sh}},\mathcal J_{\mathrm{mat}}\subset\Lambda$ be the sets on which the shape and the material parameter, respectively, are perturbed, and put
\begin{equation*}
\mathcal J_{\mathrm{def}} := \mathcal J_{\mathrm{sh}}\cup\mathcal J_{\mathrm{mat}},
\end{equation*}
which labels the defected resonators. For $\vect{j}\in\Lambda$, set
\begin{equation} \label{def:resonantor}
B_{\vect{j}} :=
\begin{cases}
\widetilde D,&\vect{j}\in\mathcal J_{\mathrm{sh}},\\
D,&\vect{j}\notin\mathcal J_{\mathrm{sh}},
\end{cases}
\qquad v_{\vect{j}} :=
\begin{cases}
\widetilde v_b,&\vect{j}\in\mathcal J_{\mathrm{mat}},\\
v_b,&\vect{j}\notin\mathcal J_{\mathrm{mat}},
\end{cases}
\qquad D_{\vect{j}}:=B_{\vect{j}}+\vect{j}.
\end{equation}
Thus, a material parameter defect corresponds to $\mathcal J_{\mathrm{sh}}=\varnothing$, a shape defect to $\mathcal J_{\mathrm{mat}}=\varnothing$, and our framework permits both types of perturbation. Finally, the collection of resonators is denoted by
\begin{equation*}
    \mathcal D:=\bigcup_{\vect{j}\in\Lambda}D_{\vect{j}}.
\end{equation*}

It is said that $\mathcal{D}$ is a compactly defected structure if $|\mathcal J_{\mathrm{def}}|<\infty$, and noncompact otherwise. A noncompact defected structure may retain a lower-dimensional periodicity, as in a straight line defect, or be completely nonperiodic, as in a bent waveguide. The precise straight-line setting used below is specified in \eqref{eq:defectindex}; see \Cref{fig:illustration2}.

Let $0<\delta\ll1$ be the contrast parameter and let $v>0$ be the exterior wave speed. Define
\begin{equation*}
a_\delta := \delta\mathbbm{1}_{\mathbb R^d\setminus\overline{\mathcal D}}+\mathbbm{1}_{\mathcal D}, \qquad m_\delta := \frac{\delta}{v^2}\mathbbm{1}_{\mathbb R^d\setminus\overline{\mathcal D}} +\sum_{\vect{j}\in\Lambda}\frac{1}{v_{\vect{j}}^2}\mathbbm{1}_{D_{\vect{j}}}.
\end{equation*}
Here, $\mathbbm{1}_{\mathbb{R}^d \setminus \overline{\mathcal{D}}}$ (respectively, $\mathbbm{1}_{\mathcal{D}}$ and $\mathbbm{1}_{D_{\vect{j}}}$) denotes the characteristic function of $\mathbb{R}^d \setminus \overline{\mathcal{D}}$ (respectively, $\mathcal{D}$ and $D_{\vect{j}}$). 
The closed form
\begin{equation} \label{def:formLdelta}
(u,w)\longmapsto \int_{\mathbb R^d}a_\delta\nabla u\cdot\nabla\overline w\,\mathrm dx, \qquad u,w\in H^1(\mathbb R^d),
\end{equation}
defines a self-adjoint elliptic operator $\mathcal L_\delta$ in $L^2(\mathbb R^d,m_\delta\,\mathrm dx)$. Equivalently,
\begin{equation} \label{def:perturbedop}
\mathcal L_\delta:\operatorname{dom}(\mathcal L_\delta)\subset L^2(\mathbb R^d,m_\delta\,\mathrm dx) \longrightarrow L^2(\mathbb R^d,m_\delta\,\mathrm dx),
\qquad \mathcal L_\delta u = -m_\delta^{-1}\nabla\cdot(a_\delta\nabla u),
\end{equation}
with its natural operator domain
\begin{equation*}
\operatorname{dom}(\mathcal L_\delta):=\Big\{u\in H^1(\mathbb R^d):
m_\delta^{-1}\nabla\cdot(a_\delta\nabla u)\in L^2(\mathbb R^d,m_\delta\,\mathrm dx)\Big\},
\end{equation*}
where the divergence is understood in the sense of distributions.

With $z=\omega^2$, the whole-space spectral problem $(\mathcal L_\delta-z)u=0$ is equivalent to
\begin{equation}
\label{eq:scattering-problem}
\begin{cases}
\Delta u+\dfrac{\omega^2}{v^2}u=0 &\text{in }\mathbb R^d\setminus\overline{\mathcal D},\\
\Delta u+\dfrac{\omega^2}{v_{\vect{j}}^2}u=0 &\text{in }D_{\vect{j}},\\
u|_+=u|_- &\text{on }\partial D_{\vect{j}},\\
\delta\partial_\nu u|_+-\partial_\nu u|_-=0 &\text{on }\partial D_{\vect{j}},
\end{cases}
\qquad \vect{j}\in\Lambda.
\end{equation}
Here, $\nu$ is the unit normal pointing out of $D_{\vect{j}}$, and the signs $+$ and $-$ denote exterior and interior traces, respectively.

\subsection{Dirichlet-to-Neumann formulation and operator pencil}
\label{subsec:NtD}

In this work, we are interested in the spectral properties of $\mathcal{L}_{\delta}$ as $\delta\to 0$, while $\delta$-dependence of the domain $\operatorname{dom}(\mathcal{L}_{\delta})$ hinders the asymptotic analysis. In this section, we introduce an alternative characterization of the eigenvalue problem \eqref{eq:scattering-problem}, which has a $\delta$-independent domain as in \cite{ammari2026resolvent_con}. The idea is to reformulate the eigenvalue problem \eqref{eq:scattering-problem} in the interior domain $\mathcal{D}$ (\emph{i.e.}, inside the resonators) by eliminating the exterior field with the help of the exterior DtN map \cite{Feppon2024}. 

Here and in what follows, we use the interior space $H^1(\mathcal D) := \bigoplus_{\vect{j}\in\Lambda}H^1(D_{\vect{j}})$ with the standard $H^1$ norm. Let $\gamma_+$ denote the exterior trace map on $H^1(\mathbb R^d\setminus\overline{\mathcal D})$. We define the exterior trace space by
\begin{equation*}
H^{1/2}(\partial\mathcal D):=\gamma_+H^1(\mathbb R^d\setminus\overline{\mathcal D}),
\qquad \|\phi\|_{H^{1/2}(\partial\mathcal D)} :=\inf_{\substack{U\in H^1(\mathbb R^d\setminus\overline{\mathcal D})\\ \gamma_+U=\phi}}\|U\|_{H^1(\mathbb R^d\setminus\overline{\mathcal D})},
\end{equation*}
and denote its anti-dual by $H^{-1/2}(\partial\mathcal D)$. We denote by $\mathcal{B}(X,Y)$ the space of bounded linear operators from $X$ to $Y$ and write $\mathcal{B}(X):=\mathcal{B}(X,X)$. 
  
Let $\sigma(-\Delta_{\mathrm{Dir}}; \mathbb R^d\setminus\overline{\mathcal D})$
be the Dirichlet spectrum of $-\Delta$ in $\mathbb R^d\setminus\overline{\mathcal D}$. For any $z\in\mathbb C$ such that $z/v^2 \notin \sigma(-\Delta_{\mathrm{Dir}}; \mathbb R^d\setminus\overline{\mathcal D})$,
the exterior DtN map $\mathcal{T}^{z}$ is defined as follows.     
\begin{definition}
For $\phi\in H^{1/2}(\partial \mathcal{D})$, the exterior DtN map $\mathcal{T}^{z}$ is given by
\begin{equation*}
	\begin{aligned}
	    \mathcal{T}^{z}: H^{1/2}(\partial \mathcal{D}) &\to H^{-1/2}(\partial \mathcal{D}), \\
		\phi & \mapsto \partial_\nu u \bigg|_{+},
		\end{aligned}
\end{equation*}
where $u\in H^1(\mathbb R^d\setminus\overline{\mathcal D})$ is the unique finite-energy solution of
\begin{equation} \label{eq_d2n_def}
	\left\{ \begin{aligned}
	&-\Delta u- \frac{z}{v^2} u=0 \quad \text{in }\mathbb R^d\setminus\overline{\mathcal D}, \\
	&u\big|_{+}=\phi \quad \text{on } \partial \mathcal{D}.
	\end{aligned}\right.
\end{equation}
\end{definition} 

The following property was established in the subwavelength regime, \emph{i.e.}, for $z$ in a small neighborhood of $0$: in \cite[Proposition 3.1]{Feppon2024} when $\mathcal{D}$ consists of finitely many resonators, and in \cite[Proposition 2.2]{qiu2025nonlinear} in the periodic case. It was subsequently extended to arbitrary configurations $\mathcal{D}$ in \cite{ammari2026resolvent_con}. We note that the restriction to a neighborhood of $0$ enters the argument of \cite{ammari2026resolvent_con} only through the existence of the exterior resolvent; the construction itself applies to every $z$ such that $z/v^2$ belongs to the resolvent set $\rho(-\Delta_{\mathrm{Dir}}; \mathbb R^d\setminus\overline{\mathcal D})$ of the exterior Dirichlet Laplacian. In the subwavelength regime, such a neighborhood is available because the geometric assumptions of \cite{ammari2026resolvent_con} yield a uniform Poincar\'e inequality on $\mathbb{R}^d\setminus\overline{\mathcal{D}}$, and hence $\inf\sigma(-\Delta_{\mathrm{Dir}}; \mathbb R^d\setminus\overline{\mathcal D}) >0$; see \cite[Appendix A]{ammari2026resolvent_con}. 

\begin{proposition}
\label{prop_d2n_map}
The operator-valued map
\begin{equation*}
z\longmapsto\mathcal T^z\in\mathcal B\bigl(H^{1/2}(\partial\mathcal D),H^{-1/2}(\partial\mathcal D)\bigr)
\end{equation*}
is well-defined and analytic on the set
\begin{equation*}
\left\{z\in\mathbb C:\frac{z}{v^2}\in\rho(-\Delta_{\mathrm{Dir}}; \mathbb R^d\setminus\overline{\mathcal D})\right\}.
\end{equation*}
For every real $z$ in this set, $\mathcal T^z$ is self-adjoint with respect to the boundary duality pairing.
\end{proposition}

For any such $z$, eliminating the exterior field in \eqref{eq:scattering-problem} gives the interior problem
\begin{equation}
\label{eq_interior_pde_subwavelength_resonance}
\begin{cases}
-\Delta u-\dfrac{z}{v_{\vect{j}}^2}u=0&\text{in }D_{\vect{j}},\\
\partial_\nu u|_-= \delta\mathcal T^z[u|_{\partial\mathcal D}] &\text{on }\partial D_{\vect{j}},
\end{cases}
\qquad \vect{j}\in\Lambda.
\end{equation}
For $u,w\in H^1(\mathcal D)$, define
\begin{equation}
\label{eq_sesquilinear_form_subwavelength_resonance}
\begin{aligned}
\mathfrak a(u,w;z,\delta) :={}& \sum_{\vect{j}\in\Lambda} \int_{D_{\vect{j}}} \nabla u\cdot\nabla\overline w -\frac{z}{v_{\vect{j}}^2}u\overline w \,\mathrm dx - \delta \left\langle \mathcal T^z[u|_{\partial\mathcal D}], w|_{\partial\mathcal D} \right\rangle_{ H^{-1/2}(\partial\mathcal D), H^{1/2}(\partial\mathcal D)}.
\end{aligned}
\end{equation}
The associated bounded operator,  defined by $\langle\mathcal A_\delta(z)u,w\rangle=\mathfrak a(u,w;z,\delta)$, is Hermitian for real $z$ and acts on the following $\delta$-independent spaces:
\begin{equation*}
\mathcal A_\delta(z): H^1(\mathcal D) \longrightarrow H^1(\mathcal D)^*. 
\end{equation*}
For every real $z$ such that $z/v^2\in\rho(-\Delta_{\mathrm{Dir}};\mathbb R^d\setminus\overline{\mathcal D})$, the DtN reduction above gives
\begin{equation}
\label{eq_variational_characterization_subwavelength_resonance}
z\in\sigma_p(\mathcal L_\delta)\quad\Longleftrightarrow\quad\ker\mathcal A_\delta(z)\neq\{0\}.
\end{equation}
Our analysis is based on the fact that the relevant spectral information is also encoded in the resolvent of $\mathcal A_\delta(z)$. Whenever $\mathcal{A}_\delta(z)$ is invertible, we denote by
\begin{equation*}
\mathcal R(z,\delta): L^2(\mathcal D) \longrightarrow H^1(\mathcal D)
\end{equation*}
the solution operator determined by
\begin{equation*}
\mathfrak a(\mathcal R(z,\delta)f,w;z,\delta) = (f,w)_{L^2(\mathcal D)} \qquad \text{for every }w\in H^1(\mathcal D).
\end{equation*}
We call $\mathcal{R}(z,\delta)$ the resolvent of $\mathcal{A}_\delta(z)$. 

\section{Frequency-dependent capacitance operator and spectral asymptotics}
\label{sec:capacitance-spectral}

In this section, we derive a discrete approximation of $\mathcal L_\delta$ and obtain its spectral asymptotics as $\delta\to0$. The main reduction theorem is presented in \Cref{sec:asympspec}. Our central tool is the frequency-dependent capacitance operator introduced in \Cref{sec:2.2}, which extends the frequency-dependent capacitance matrix developed in \cite{ammariFrequencydependentCapacitanceMatrix2026} for periodic media to aperiodic defect configurations. In \Cref{sec:exponentialdecay}, we prove that the off-diagonal coefficients of this operator decay exponentially. This locality property underpins the patch method developed in \Cref{sec:pathch} for the efficient computation of non-subwavelength defect eigenfrequencies and their associated eigenmodes.

\subsection{Full capacitance operator}
\label{sec:2.2}

For a given frequency $\omega_0$ and any $\vect{j}\in\Lambda$, define the local Neumann eigenspace
\begin{equation} \label{def:neumann_space}
E_{\vect{j}}(\omega_0) := \ker\Big(-\Delta_{{\rm Neu}, D_{\vect{j}}}-\frac{\omega_0^2}{v_{\vect{j}}^2} \Big), \qquad m_{\vect{j}}:=\dim E_{\vect{j}}(\omega_0).
\end{equation} 
We now fix a nonzero reference frequency $\omega_0$ such that $\omega_0$ is an interior Neumann eigenfrequency of at least one defect resonator, but not of any unperturbed resonator, that is, 
\begin{equation}
\label{gapassump1}
m_{\vect{j}}=0 \quad\text{for every }\vect{j}\notin\mathcal J_{\mathrm{def}},
\qquad
m_{\vect{k}}>0 \quad\text{for some }\vect{k}\in\mathcal J_{\mathrm{def}}.
\end{equation}
A defect site is called active at frequency $\omega_0$ if its interior Neumann eigenspace at $\omega_0$ is nontrivial, \emph{i.e.} $m_{\vect{j}}>0$. In this case, one can choose an $L^2(D_{\vect{j}})$-orthonormal basis $u_{\vect{j},1},\dots,u_{\vect{j},m_{\vect{j}}}$ satisfying
\begin{equation}
\label{neumanneq}
-\Delta u_{\vect{j},\ell} = \frac{\omega_0^2}{v_{\vect{j}}^2}u_{\vect{j},\ell} \quad\text{in }D_{\vect{j}}, \qquad \partial_\nu u_{\vect{j},\ell}=0 \quad\text{on }\partial D_{\vect{j}}.
\end{equation}
Since there are only four possible pairs of $(B_{\vect{j}},v_{\vect{j}})$ as seen from \eqref{def:resonantor} and the corresponding local Neumann spectra are discrete, it follows that
\begin{equation}
\label{gapassump1uniform}
\inf_{\vect{j}\in\Lambda} \operatorname{dist}\Big( \frac{\omega_0^2}{v_{\vect{j}}^2}, \sigma(-\Delta_{\mathrm{Neu}};D_{\vect{j}}) \setminus \Big\{\frac{\omega_0^2}{v_{\vect{j}}^2}\Big\} \Big) > 0.
\end{equation}
For our analysis, we assume $\omega_0^2/v^2$ belongs to the resolvent set of the exterior Dirichlet Laplacian, that is, 
\begin{equation}
\label{gapassump2equiv}
\operatorname{dist} \Big( \frac{\omega_0^2}{v^2}, \sigma(-\Delta_{\mathrm{Dir}}; \mathbb R^d\setminus\overline{\mathcal D}) \Big) >0.
\end{equation}
In the following, we call \eqref{gapassump1} and \eqref{gapassump1uniform} the \emph{interior-resonance conditions}, and \eqref{gapassump2equiv} the \emph{exterior gap condition}.

When the exterior geometry $\mathbb{R}^d\setminus\overline{\mathcal{D}}$ is periodic, \eqref{gapassump2equiv} is equivalent to
\begin{equation}
\label{gapassump2}
\inf_{\vect{\alpha} \in Y^*} \operatorname{dist} \Big( \frac{\omega_0^2}{v^2}, \sigma(-\Delta_{\vect{\alpha},\mathrm{Dir}};Y\setminus\overline D) \Big) > 0,
\end{equation}
Here, $-\Delta_{\vect{\alpha},\mathrm{Dir}}$ denotes the corresponding exterior Dirichlet Floquet--Bloch component. The dual lattice and the Brillouin torus are given by
\begin{equation*}
\begin{aligned}
\Lambda^*&:=\operatorname{span}_{\mathbb Z}\{e_1^*,\dots,e_d^*\},
&e_i^*\cdot e_j&=2\pi\delta_{ij},\\
Y^*&:=\mathbb R^d/\Lambda^*.
\end{aligned}
\end{equation*}
In particular, \eqref{gapassump2equiv} and \cref{prop_d2n_map} imply that $\mathcal T^z$ is analytic in a neighborhood of $\omega_0^2$.

We define 
\begin{equation} \label{def:activeindex}
    \mathcal I(\omega_0) := \left\{ (\vect{j},\ell):\vect{j}\in\Lambda,\ 1\leq\ell\leq m_{\vect{j}} \right\},
\end{equation}
as the index set of active interior Neumann modes at $\omega_0$. With a slight abuse of notation, the trace on $\partial D_{\vect{j}}$ of the Neumann eigenfunction $u_{\vect{j},\ell}$ in \eqref{neumanneq} is still denoted as $u_{\vect{j},\ell}$, extended by zero to the other boundary components of $\partial\mathcal D$. For each $(\vect{j},\ell)\in\mathcal I(\omega_0)$, by the assumption \eqref{gapassump2equiv}, we can define $U_{\vect{j},\ell}\in H^1(\mathbb R^d\setminus\overline{\mathcal D})$ as the unique solution of
\begin{equation}
\label{eq:def-Vjl-updated}
\begin{cases}
\Big(\Delta+\dfrac{\omega_0^2}{v^2}\Big)U_{\vect{j},\ell}=0 &\text{in }\mathbb R^d\setminus\overline{\mathcal D},\\
U_{\vect{j},\ell}=u_{\vect{j},\ell} &\text{on }\partial D_{\vect{j}},\\
U_{\vect{j},\ell}=0 &\text{on }\partial D_{\vect{j}'},\quad \vect{j}'\neq \vect{j}.
\end{cases}
\end{equation}

\begin{definition}
\label{def:freq-cap-operator}
The \emph{full frequency-dependent capacitance operator} at $\omega_0$ is the operator on $\ell^2(\mathcal I(\omega_0))$ specified by the integral kernel
\begin{equation}
\label{eq:def-freq-cap-operator-entry}
\mathcal C_{(\vect{j},\ell),(\vect{j}',\ell')}(\omega_0)
=-\frac{v_{\vect{j}}v_{\vect{j}'}}{2\omega_0}
\int_{\partial D_{\vect{j}}}
\left.\partial_\nu U_{\vect{j}',\ell'}\right|_+
\overline{u_{\vect{j},\ell}}\,\mathrm d\sigma.
\end{equation}
\end{definition}

Equivalently, define, for $a \in \ell^2(\mathcal I(\omega_0))$, 
\begin{equation} \label{def:GandM}
G a := \sum_{(\vect{j},\ell)\in\mathcal I(\omega_0)}a_{\vect{j},\ell}u_{\vect{j},\ell}, \qquad (Ma)_{\vect{j},\ell}:=v_{\vect{j}}^{-2}a_{\vect{j},\ell}.
\end{equation}
The diagonal operator $M$ is bounded positive with a bounded inverse. Then, the capacitance operator can be formulated as 
\begin{equation}
\label{DtNCapacitance}
\mathcal C(\omega_0) = -\frac{1}{2\omega_0} M^{-1/2}G^*\mathcal T^{\omega_0^2}GM^{-1/2}.
\end{equation}
\begin{remark}
    Intuitively speaking, one can understand our defected resonator system by its analogy to the condensed matter system: the active defect sites, recorded by $\mathcal I(\omega_0)$, correspond to the sites on which the electrons live, and the capacitance operator $\mathcal C$ corresponds to the Hamiltonian that governs the motion of electrons.     
\end{remark}

By the trace inequality, the operator $G:\ell^2(\mathcal I(\omega_0)) \to H^{1/2}(\partial\mathcal D)$ is bounded. Hence, by the boundedness of $G$ and $M^{-1}$, as well as \cref{prop_d2n_map}, we have the following result. 

\begin{proposition}
\label{prop:cap-basic}
Under \eqref{gapassump2equiv}, $\mathcal C(\omega_0)$ is a bounded self-adjoint operator on $\ell^2(\mathcal I(\omega_0))$.
\end{proposition}

\begin{remark}
Note that for $(\vect{j},\ell), (\vect{j}',\ell') \in {\mathcal{I}}(\omega_0)$, we have
$$
\partial_\nu U_{\vect{j}^\prime,\ell^\prime}\big|_+ = \mathcal{T}^{\omega_0^2}[U_{\vect{j}^\prime,\ell^\prime} \mathbbm{1}_{\partial D_{\vect{j}^\prime}}],
$$
and therefore, \eqref{eq:def-freq-cap-operator-entry} can be rewritten as 
\begin{equation}
\mathcal C_{(\vect{j},\ell),(\vect{j}',\ell')}(\omega_0)
=-\frac{v_{\vect{j}}v_{\vect{j}'}}{2\omega_0}
\int_{\partial D_{\vect{j}}}
\mathcal{T}^{\omega_0^2}[U_{\vect{j}^\prime,\ell^\prime} \mathbbm{1}_{\partial D_{\vect{j}^\prime}}]
\overline{u_{\vect{j},\ell}}\,\mathrm d\sigma.
\end{equation}

Moreover, note also that $\mathcal{T}^{\omega_0^2}$ is self-adjoint and the assumption \eqref{gapassump2equiv} guarantees that $\omega_0^2$ is in a spectral gap of $- v^2 \Delta$ in $\mathbb{R}^d \setminus \overline{\mathcal{D}}$ with Dirichlet boundary conditions, but does not by itself prove the existence of such gaps in arbitrary geometries. For suitable periodic geometries in which the resonators enclose nearly
isolated exterior cavities, near-touching configurations can produce internal spectral gaps. In the limiting case, when the region between the resonators is closed, the spectrum of $- \Delta$ in $\mathbb{R}^d \setminus \overline{\mathcal{D}}$ with Dirichlet boundary conditions concentrates near the Dirichlet eigenvalues of $-\Delta$ in such a closed region and, consequently, a spectral gap opening occurs. We refer the reader to  \cite{nazarov1,nazarov2} for more details. 
\end{remark} 

\subsection{Line defects and the quasi-periodic capacitance matrix}
\label{subsec:QP}

In this section, we consider $d=2$ and the line defect, \emph{i.e.},  
\begin{equation} \label{eq:defectindex}
\mathcal J_{\mathrm{def}} = \{(j_1,0):j_1\in\mathbb Z\}.
\end{equation}
We assume that every site on the defect line has the same resonator type, denoted by $(D_\star,v_\star)$, in the sense that $D_{\vect{j}} = \vect{j}+D_{\star}$ and $ v_{\vect{j}}= v_\star$ for all $\vect{j}\in \mathcal J_{\mathrm{def}}$ with
\begin{equation*}
(D_\star,v_\star)=
\begin{cases}
(\widetilde D,v_b),&\text{for a shape perturbation},\\
(D,\widetilde v_b),&\text{for a material parameter perturbation}.
\end{cases}
\end{equation*}
The representative defect resonator is $D_{(0,0)}=D_\star$ with $m_\star:=m_{(0,0)}$ given as in \eqref{def:neumann_space}. We also assume $m_\star>0$ and introduce the notation, as shown in \Cref{fig:illustration2},  
\begin{equation*}
\begin{aligned}
\widetilde Y := \bigcup_{j_2\in\mathbb Z}(Y+(0,j_2)),
\qquad \widetilde Y^* := \mathbb R/(2\pi\mathbb Z), \qquad
\widetilde\Omega_{\mathrm{line}} := \widetilde Y \setminus \overline{\bigcup_{j_2\in\mathbb Z}D_{(0,j_2)}}.
\end{aligned}
\end{equation*}
Here, $\widetilde{Y}$ denotes the unit strip.  Let $\Gamma_-$ and $\Gamma_+=\Gamma_-+e_1$ denote the two lateral boundaries of the strip. We define the space 
$
H^1_\alpha(\widetilde\Omega_{\mathrm{line}}) := \{ u\in H^1(\widetilde\Omega_{\mathrm{line}}): u|_{\Gamma_+}=e^{\mathrm i\alpha}u|_{\Gamma_-} \}
$ of functions that are $\alpha$-quasi-periodic along the defect and square-integrable in the transverse direction.

\begin{figure}[!h]
    \centering
\begin{tikzpicture}[
    scale=0.8,
    transform shape,
    dot/.style={
        circle,
        draw,
        minimum size=10mm,
        text width=10mm,   %
        align=center,
        inner sep=0pt,
        font=\scriptsize
    }
]

\foreach \x in {-2,...,2} {
    \foreach \y in {-2,...,2} {

        \pgfmathtruncatemacro{\col}{\x+3}
        \pgfmathtruncatemacro{\row}{\y+3}

        \ifnum\y=0
            \node[dot,fill=red!43] (c\col\row) at (1.5*\x,1.5*\y)
                {${D_\star}$};%
        \else
            \node[dot,fill=Cerulean!30] (c\col\row) at (1.5*\x,1.5*\y)
                {$D$};%
        \fi

        \draw[dashed] (1.5*\x + 0.75,-4.2) -- (1.5*\x + 0.75,4.2);
        \draw[dashed] (-4.2, 1.5*\y + 0.75) -- (4.2, 1.5*\y + 0.75);
    }
}
\draw[dashed] (1.5*-3 + 0.75,-4.2) -- (1.5*-3 + 0.75,4.2);
\draw[dashed] (-4.2, 1.5*-3 + 0.75) -- (4.2, 1.5*-3 + 0.75);

\node at (4.5,0) {\Huge $\innercdots$};
\node at (-4.5,0) {\Huge $\innercdots$};
\node at (0, 4.5) {\Huge $\outervdots$};
\node at (0, -4) {\Huge $\outervdots$};

\node at (-4.5, 4.5) {\Huge $\outerddots$};
\node at (4.5, 4.5) {\Huge $\outeriddots$};
\node at (-4.5, -4) {\Huge $\outeriddots$};
\node at (4.5, -4) {\Huge $\outerddots$};

\node at (0, -5) {\Large crystal};

\def\xstrip{10}

\foreach \y in {-2,...,2} {
    \pgfmathtruncatemacro{\row}{\y+3}

    \ifnum\y=0
        \node[dot,fill=red!43] (copy\row) at (\xstrip,1.5*\y)
            {${D_\star}$};%
    \else
        \node[dot,fill=Cerulean!30] (copy\row) at (\xstrip,1.5*\y)
            {$D$};%
    \fi

    \draw[dashed] ({\xstrip - 0.75},{\y * 1.5 + 0.75}) -- ({\xstrip + 0.75},{\y * 1.5 + 0.75});
}

\draw[dashed] ({\xstrip - 0.75},{4}) -- ({\xstrip - 0.75},{-3.75});
\draw[dashed] ({\xstrip + 0.75},{4}) -- ({\xstrip + 0.75},{-3.75});

\node at (\xstrip, 4.5) {\Huge $\outervdots$};
\node at (\xstrip, -4) {\Huge $\outervdots$};

\node at (\xstrip, -5) {\Large unit strip $\widetilde Y$};

\def\squarex{7}

\draw[dashed] ({\squarex - 1},{4}) -- ({\squarex - 1},{2});
\draw[dashed] ({\squarex + 1},{4}) -- ({\squarex + 1},{2});
\draw[dashed] ({\squarex - 1},{4}) -- ({\squarex + 1},{4});
\draw[dashed] ({\squarex + 1},{2}) -- ({\squarex - 1},{2});

\node[
        circle,
        draw,
        minimum size=12mm,
        text width=10mm,   %
        align=center,
        inner sep=0pt,
        font=\scriptsize] at (7,3){};

\draw[->,thick,shorten >=4pt]

    (4,1.5) -- (6, 2){};

\node at (\squarex, 1.7) {\Large unit cell $Y$};

    \draw[->] (5.4,-4.2) -- (6.4,-4.2) node[right] {\footnotesize $e_1$};
    \draw[->] (5.4,-4.2) -- (5.4,-3.2) node[above] {\footnotesize $e_2$};

\draw[->,thick,shorten >=4pt]

    (6,0) node[right, align=center] {defect line}

    -- (5, 0);

\end{tikzpicture}
\caption{Defect line of resonators $D_\star$ with a shape $\widetilde D$ or a material parameter perturbation $\widetilde v_b$ in an infinite crystal with unperturbed resonators $D$. The unit cell $Y$ is shown, as well as the unit strip $\widetilde Y$ along $e_2$ containing one defect resonator and infinitely many unperturbed resonators on each side.}
\label{fig:illustration2}
\end{figure}

Similarly to \eqref{def:activeindex}, we write 
\begin{equation*}
\widetilde{\mathcal I}(\omega_0) := \{((0,0),\ell):1\leq\ell\leq m_\star\}.
\end{equation*}
Similarly to \eqref{gapassump2},  the exterior gap condition \eqref{gapassump2equiv} is equivalent to
\begin{equation}
\label{gapassump2star}
\inf_{\alpha \in \widetilde Y^*} \operatorname{dist} \Big( \frac{\omega_0^2}{v^2}, \sigma(-\Delta_{\alpha,\mathrm{Dir}}; \widetilde\Omega_{\mathrm{line}}) \Big) > 0.
\end{equation}
For $\alpha\in\widetilde Y^*$ and $1\leq\ell\leq m_\star$, let $V_\ell^\alpha$ be the unique solution in $H^1_\alpha(\widetilde\Omega_{\mathrm{line}})$ of
\begin{equation} \label{eq:def-Vl-line}
\begin{cases}
\Big(\Delta+\dfrac{\omega_0^2}{v^2}\Big)V_\ell^\alpha=0 &\text{in }\widetilde\Omega_{\mathrm{line}},\\
V_\ell^\alpha=u_{(0,0),\ell} &\text{on }\partial D_{(0,0)},\\
V_\ell^\alpha=0 &\text{on }\partial D_{(0,j_2)},\quad j_2\neq0.
\end{cases}
\end{equation}
We define the quasi-periodic capacitance matrix associated with the line defect as follows. 

\begin{definition}
\label{def:freq-cap-QP-line}
The \emph{quasi-periodic frequency-dependent capacitance matrix} at $\omega_0$ is an $m_\star\times m_\star$ Hermitian matrix $\mathcal C^\alpha(\omega_0) =(\mathcal C^\alpha_{\ell,\ell'}(\omega_0))_{\ell,\ell'=1}^{m_\star}$, given by 
\begin{equation}
\label{eq:def-freq-cap-operator-Line}
\mathcal C^\alpha_{\ell,\ell'}(\omega_0) := -\frac{v_\star^2}{2\omega_0} \int_{\partial D_{(0,0)}} \partial_\nu V_{\ell'}^\alpha|_+ \overline{u_{(0,0),\ell}} \,\mathrm d\sigma.
\end{equation}
\end{definition}

Note that the full frequency-dependent capacitance operator $\mathcal C(\omega_0)$ associated with the line defect \eqref{eq:defectindex} is equivalent to the quasi-periodic one $\mathcal C^\alpha(\omega_0)$ by the partial Floquet--Bloch transform:
\begin{equation*}
(\mathcal F_1a)_\ell(\alpha) := \sum_{j_1\in\mathbb Z} e^{-\mathrm i\alpha j_1}a_{(j_1,0),\ell}.
\end{equation*}
Specifically, they are related by the direct-integral decomposition 
\begin{equation*}
\mathcal F_1\mathcal C(\omega_0)\mathcal F_1^{-1} = \int_{\widetilde Y^*}^{\oplus} \mathcal C^\alpha(\omega_0) \,\frac{\mathrm d\alpha}{2\pi},
\end{equation*}
or equivalently expressed in terms of their integral kernels,
\begin{equation}
\label{inversesymbolLN}
\mathcal C_{((j_1,0),\ell),((j_1',0),\ell')}(\omega_0) = \frac{1}{2\pi} \int_{\widetilde Y^*} e^{\mathrm i\alpha(j_1-j_1')} \mathcal C^\alpha_{\ell,\ell'}(\omega_0) \,\mathrm d\alpha.
\end{equation}
In particular, we have the spectral decomposition
\begin{equation*}
\sigma(\mathcal C(\omega_0)) = \bigcup_{\alpha\in\widetilde Y^*} \sigma(\mathcal C^\alpha(\omega_0)).
\end{equation*}

\subsubsection*{The unperturbed bulk crystal}

For numerical comparison with the line-defect bands, it is useful to recall the quasi-periodic frequency-dependent capacitance matrix of the unperturbed crystal at one of its own interior Neumann frequencies $\omega_{\mathrm{b}}$. Note that this construction is conceptually separate from the defect reduction at $\omega_0$: the latter assumes that the unperturbed resonator is not Neumann-resonant at $\omega_0$.  
Let $\omega_{\mathrm b}\neq0$ satisfy
\begin{equation}
\label{fullcrystal}
\frac{\omega_{\mathrm b}^2}{v_b^2} \in\sigma(-\Delta_{\mathrm{Neu}};D), \qquad \inf_{\vect{\alpha} \in Y^*} \operatorname{dist} \Big( \frac{\omega_{\mathrm b}^2}{v^2}, \sigma(-\Delta_{\vect{\alpha}, \mathrm{Dir}}; Y \setminus\overline D) \Big)>0.
\end{equation}
Let $w_1,\dots,w_p$ be an $L^2(D)$-orthonormal basis of the corresponding Neumann eigenspace. For $\vect{\alpha}\in Y^*$, define $W_\ell^{\vect{\alpha}}$ by
\begin{equation}
\label{eq:def-Vl}
\begin{cases}
\Big(\Delta+\dfrac{\omega_{\mathrm b}^2}{v^2}\Big)W_\ell^{\vect{\alpha}}=0 &\text{in }Y\setminus\overline D,\\
W_\ell^{\vect{\alpha}}=w_\ell &\text{on }\partial D,\\
W_\ell^{\vect{\alpha}} &\text{is }\vect{\alpha}\text{-quasi-periodic}.
\end{cases}
\end{equation}
The associated bulk capacitance matrix is defined as follows.

\begin{definition}
\label{def:freq-cap-QP}
The $p\times p$ matrix $\mathcal C_{\mathrm{bulk}}^{\vect{\alpha}}(\omega_{\mathrm b})$ has entries
\begin{equation}
\label{eq:def-freq-cap-operatorQP}
(\mathcal C_{\mathrm{bulk}}^{\vect{\alpha}})_{\ell,\ell'}(\omega_{\mathrm b}) := -\frac{v_b^2}{2\omega_{\mathrm b}} \int_{\partial D} \partial_\nu W_{\ell'}^{\vect{\alpha}}|_+ \overline{w_\ell} \,\mathrm d\sigma.
\end{equation}
\end{definition}
If $\lambda_{\mathrm b,1}^{\vect{\alpha}},\dots,\lambda_{\mathrm b,p}^{\vect{\alpha}}$ are the eigenvalues of $\mathcal C_{\mathrm{bulk}}^{\vect{\alpha}}$, then, as $\delta \rightarrow 0$, the bulk Bloch eigenfrequencies satisfy
\begin{equation} \label{approx:alphaj}
\omega_{\mathrm b,j}^{\vect{\alpha}}(\delta) = \omega_{\mathrm b} +\delta\lambda_{\mathrm b,j}^{\vect{\alpha}} +\mathcal O(\delta^2), \qquad j=1,\dots,p,
\end{equation}
uniformly in $\vect{\alpha} \in Y^*$ under the second condition in \eqref{fullcrystal}; see \cite[Theorem 4.3]{ammariFrequencydependentCapacitanceMatrix2026}. 

This asymptotic expansion \eqref{approx:alphaj} is used to approximately determine the bulk bands in \Cref{sec:numerical}. We will also illustrate the exponential decay of the off-diagonal entries of the frequency-dependent capacitance operator $\mathcal C_{\mathrm{bulk}}(\omega_{\mathrm b})$ of the unperturbed crystal. Using the inverse Floquet--Bloch transform, it can be obtained from $\mathcal C_{\mathrm{bulk}}^{\vect{\alpha}}(\omega_{\mathrm b})$ that: 
\begin{equation}
\label{Inversesymbol}
\bigl(\mathcal C_{\mathrm{bulk}}(\omega_{\mathrm b})\bigr)_{ (\vect{j},\ell),(\vect{j}',\ell')} = \frac{|Y|}{(2\pi)^d} \int_{Y^*} e^{\mathrm i\vect{\alpha}\cdot(\vect{j}-\vect{j}')} (\mathcal C_{\mathrm{bulk}}^{\vect{\alpha}})_{\ell,\ell'}(\omega_{\mathrm b}) \,\mathrm d\vect{\alpha}.
\end{equation}

\subsection{Norm-resolvent convergence and spectral asymptotics} \label{sec:asympspec}

In this section, we derive a discrete approximation of the continuous operator $\mathcal L_\delta$ near a reference frequency $\omega_0$, using the frequency-dependent capacitance operator $\mathcal C(\omega_0)$ defined in
\cref{def:freq-cap-operator}. We establish a norm-resolvent convergence result showing that $\mathcal C(\omega_0)$ governs the leading-order spectral behavior of $\mathcal L_\delta$ as $\delta\to0$. As a consequence, for compactly supported defects, the eigenvalues of $\mathcal C(\omega_0)$ determine the leading-order shifts of the defect eigenfrequencies bifurcating from $\omega_0$. For a periodic line defect, the eigenvalues of its quasi-periodic capacitance matrices determine the guided defect bands bifurcating from $\omega_0$ within a band gap of the unperturbed crystal. 

To this end, we introduce the rescaled spectral parameter $z_{\delta,\zeta}:=\omega_0^2+\delta\zeta$ with $\zeta$ in a compact set $\mathbb K$ of the resolvent set of $2\omega_0\mathcal C(\omega_0)$, \emph{i.e.}, $\zeta \in \mathbb K\Subset\rho(2\omega_0\mathcal C(\omega_0))$. 
The following \cref{thm_cont_to_disc_resolvent_converge} shows that the resolvent $\mathcal R(z_{\delta,\zeta},\delta)$ is well defined for all sufficiently small $\delta$, uniformly in $\zeta\in\mathbb K$, and establishes that the rescaled resolvent $\delta\mathcal R(z_{\delta,\zeta},\delta)$ converges in the operator norm to the effective resolvent associated with the discrete capacitance operator as $\delta\to0$.

The proof of \cref{thm_cont_to_disc_resolvent_converge} adapts the strategy developed in \cite[Appendix C]{ammari2026resolvent_con} for the subwavelength case $\omega_0=0$. More precisely, one decomposes $H^1(\mathcal D)$ into the subspace spanned by the interior-resonant Neumann modes and its complement, estimates the resulting four operator blocks, and then eliminates the complementary block. When $\omega_0=0$, the relevant Neumann mode on each resonator is constant, which lies at the bottom of the local Neumann spectrum. Thus, the complementary block is coercive by the standard Poincar\'e inequality; see \cite[(C.4) and (C.6)]{ammari2026resolvent_con}. In the current non-subwavelength setting, however, $\omega_0^2/v_{\vect{j}}^2$ is an interior Neumann eigenvalue, so the complementary block is generally indefinite, and this coercivity argument is not applicable. We instead establish the uniform bounded invertibility of the complementary operator in \cref{lem:complementary-invertibility} by the spectral calculus.

We start by introducing some notation. We define the operator 
\begin{equation}
\label{eq_resonator_projection}
(Pu)(\vect{j},\ell) := \int_{D_{\vect{j}}} u\,\overline{u_{\vect{j},\ell}} \,\mathrm dx, \qquad (\vect{j},\ell)\in\mathcal I(\omega_0),
\end{equation}
extracting the coefficients of the interior-resonant Neumann eigenmodes on each resonator, and its adjoint 
$P^*:\ell^2(\mathcal I(\omega_0))\to L^2(\mathcal D)$ by 
\begin{equation*}
P^*a := \sum_{(\vect{j},\ell)\in\mathcal I(\omega_0)} a_{\vect{j},\ell}\, u_{\vect{j},\ell}\mathbbm{1}_{D_{\vect{j}}}, \qquad a=(a_{\vect{j},\ell})_{(\vect{j},\ell)\in\mathcal I(\omega_0)}.
\end{equation*}
We also introduce the associated $L^2(\mathcal D)$-orthogonal projection onto the closed subspace spanned by $\{u_{\vect{j},\ell}\}_{(\vect{j},\ell)\in\mathcal I(\omega_0)}$ and its complementary projection 
\begin{equation*}
\Pi:=P^*P, \qquad Q:=I-\Pi. 
\end{equation*}
Let $\mathcal Q:=QH^1(\mathcal D)$ be the complementary subspace. Noting that the interior-resonant Neumann modes have uniformly bounded $H^1$-norms, $\Pi$ restricts to a bounded projection on $H^1(\mathcal D)$, and we have
\begin{equation} \label{eq:decomposition}
H^1(\mathcal D) = P^*\ell^2(\mathcal I(\omega_0))\oplus\mathcal{Q}. 
\end{equation}
We also define $\mathcal A_\perp:\mathcal Q\to\mathcal Q^*$ by
\begin{equation} \label{eq:def-complementary-operator}
\left\langle\mathcal A_\perp q,r\right\rangle
:=\sum_{\vect{j}\in\Lambda} \int_{D_{\vect{j}}} \left( \nabla q\cdot\nabla\overline r -\frac{\omega_0^2}{v_{\vect{j}}^2}q\overline r \right)\,\mathrm dx, \qquad q,r\in\mathcal {Q}.
\end{equation}

\begin{lemma}\label{lem:complementary-invertibility}
Under the condition \eqref{gapassump1uniform}, the operator $\mathcal A_\perp:\mathcal Q\to\mathcal Q^*$ is boundedly invertible, that is, for some $C_\perp>0$, 
\begin{equation} \label{eq:complementary-inverse-bound}
\|\mathcal A_\perp^{-1}\|_{\mathcal B(\mathcal Q^*,\mathcal Q)} \leq C_\perp.
\end{equation}
\end{lemma}

\begin{proof}
Introduce
\begin{equation*}
A:=\bigoplus_{\vect{j}\in\Lambda} \bigl(-\Delta_{{\rm Neu}, D_{\vect{j}}}\bigr), \qquad
K:=\bigoplus_{\vect{j}\in\Lambda} \kappa_{\vect{j}}I_{D_{\vect{j}}}, \qquad
\kappa_{\vect{j}}:=\frac{\omega_0^2}{v_{\vect{j}}^2}\,,
\end{equation*}
where $-\Delta_{{\rm Neu}, D_{\vect{j}}}$ denotes the nonnegative Neumann Laplacian on $L^2(D_{\vect{j}})$ and $I_{D_{\vect{j}}}$ denotes the identity operator on this space. Then, $A$ is nonnegative and self-adjoint on $L^2(\mathcal D)$, with form domain $H^1(\mathcal D)$, whereas $K$ is bounded and self-adjoint. Moreover, $A$ and $K$ commute, and we have 
$\mathcal Q=H^1(\mathcal D)\cap\bigl(\ker(A-K)\bigr)^{\perp_{L^2}}$. Let $\mathcal H_\perp:=QL^2(\mathcal D)$ and define the isometry $$T:=(I+A)^{1/2}: \mathcal Q\to\mathcal H_\perp.$$ 
This allows us to represent $\mathcal A_\perp$ on $\mathcal H_\perp$ by the bounded self-adjoint operator:
\begin{equation} \label{def:Biso}
B:=(A-K)(I+A)^{-1}\big|_{\mathcal H_\perp},
\end{equation}
in the sense that $\left\langle\mathcal A_\perp q,r\right\rangle\ = \left\langle B\,Tq,Tr\right\rangle_{L^2(\mathcal D)}$ for all $q,r\in\mathcal Q$.
From \eqref{gapassump1uniform}, there exists $c_0 > 0$ such that for any $\mu\in\sigma(-\Delta_{\mathrm{Neu}};D_{\vect{j}})$ with $\mu\neq\kappa_{\vect{j}}$, 
\begin{equation*}
\frac{|\mu-\kappa_{\vect{j}}|}{1+\mu} \geq c_0 > 0.
\end{equation*}
By the representation \eqref{def:Biso} and the spectral calculus, we have $\|B^{-1}\|\leq c_0^{-1}$, which completes the proof. 
\end{proof}

\begin{theorem}
\label{thm_cont_to_disc_resolvent_converge}
Let $\mathbb K\Subset\rho(2\omega_0\mathcal C(\omega_0))$ be compact. Under \eqref{gapassump1uniform} and \eqref{gapassump2equiv}, there exists $\delta_{\mathbb K}>0$ such that $\mathcal A_\delta(\omega_0^2+\delta\zeta)$ is invertible for any $\zeta\in\mathbb K$ and $0<\delta<\delta_{\mathbb K}$. In this case, the resolvent $\mathcal R(\omega_0^2+\delta\zeta,\delta)$ is well defined, and there exists $C_{\mathbb K}>0$ such that
\begin{equation}
\label{eq_cont_to_disc_resolvent_converge}
\Bigl\| \delta\mathcal R(\omega_0^2+\delta\zeta,\delta) -P^*M^{-1/2}\bigl(2\omega_0\mathcal C(\omega_0)-\zeta\bigr)^{-1}M^{-1/2}P \Bigr\|_{\mathcal B(L^2(\mathcal D))} \leq C_{\mathbb K}\delta.
\end{equation}
\end{theorem}

\begin{remark}
Since $\mathcal C(\omega_0)$ is self-adjoint by \cref{prop:cap-basic}, any compact set $\mathbb K$ with $\mathbb K\cap\mathbb R=\emptyset$ is contained in $\rho(2\omega_0\mathcal C(\omega_0))$. Thus, the restriction on $\mathbb K$ in \cref{thm_cont_to_disc_resolvent_converge} includes the one in \cite[Theorem 2.1]{ammari2026resolvent_con} and additionally permits $\mathbb K$ to intersect the real axis. 
\end{remark}

\begin{proof}
Throughout, $C_{\mathbb K}$ denotes a positive constant depending only on $\mathbb K$, which may change line by line. 

\medskip 
\noindent 
\emph{Step 1: first-order expansion of the DtN map.} Since $\mathbb K$ is compact, we can restrict $\delta$ so that
\begin{equation*}
\frac{\delta}{v^2}\sup_{\zeta\in\mathbb K}|\zeta| < \frac{1}{2} \operatorname{dist}\Big( \frac{\omega_0^2}{v^2}, \sigma(-\Delta_{\mathrm{Dir}}; \mathbb R^d\setminus\overline{\mathcal D})\Big).  
\end{equation*} 
It follows that $z_{\delta,\zeta}/v^2$ remains in the resolvent set $\rho(-\Delta_{\mathrm{Dir}}; \mathbb R^d\setminus\overline{\mathcal D})$, uniformly for $\zeta\in\mathbb K$. By \cref{prop_d2n_map}, the map $z\mapsto\mathcal T^z$ is analytic on a neighborhood containing all such $z_{\delta,\zeta}$, and
\begin{equation}
\label{eq:dtn-first-order-bound}
\sup_{\zeta\in\mathbb K} \left\| \mathcal T^{z_{\delta,\zeta}} -\mathcal T^{\omega_0^2} \right\|_{\mathcal B(H^{1/2},H^{-1/2})} \leq C_{\mathbb K}\delta
\end{equation}
for sufficiently small $\delta$. 

\medskip 
\noindent 
\emph{Step 2: block decomposition.} We next express the operator pencil $\mathcal{A}_\delta$ in \eqref{eq_sesquilinear_form_subwavelength_resonance}  with
respect to the decomposition \eqref{eq:decomposition}. We consider $\mathcal X:=\ell^2(\mathcal I(\omega_0))\oplus\mathcal Q$ with its product norm, and define 
\begin{equation*}
J:\mathcal X\to H^1(\mathcal D), \qquad J(a,q):=P^*a+q.
\end{equation*}
Then $J$ is a bounded isomorphism with $J^{-1}u=(Pu,Qu)$. Let
$J^\#:H^1(\mathcal D)^*\to\mathcal X^*$ be the induced anti-dual
isomorphism defined by $\left\langle J^\#F,(b,r)\right\rangle :=\left\langle F,J(b,r)\right\rangle$. Identifying $\mathcal X^*$ canonically with $\ell^2(\mathcal I(\omega_0))\oplus\mathcal Q^*$, we define 
\begin{equation*}
\widehat{\mathcal A}_\delta(\zeta) := J^\#\mathcal A_\delta(z_{\delta,\zeta})J : \mathcal X\longrightarrow\mathcal X^*,
\end{equation*}
which has the block representation
\begin{equation*}
\widehat{\mathcal A}_\delta(\zeta) =
\begin{pmatrix}
\mathcal A_{PP}^{\delta,\zeta} & \mathcal A_{PQ}^{\delta,\zeta}\\
\mathcal A_{QP}^{\delta,\zeta} & \mathcal A_{QQ}^{\delta,\zeta}
\end{pmatrix},
\end{equation*}
where the four blocks are defined by, using the form \eqref{eq_sesquilinear_form_subwavelength_resonance}, 
\begin{equation*}
\begin{alignedat}{2}
\bigl(\mathcal A_{PP}^{\delta,\zeta}a,b\bigr)_{\ell^2(\mathcal I(\omega_0))} &:=\mathfrak a(P^*a,P^*b;z_{\delta,\zeta},\delta), \qquad&
\bigl(\mathcal A_{PQ}^{\delta,\zeta}q,b\bigr)_{\ell^2(\mathcal I(\omega_0))} &:=\mathfrak a(q,P^*b;z_{\delta,\zeta},\delta), \\
\left\langle\mathcal A_{QP}^{\delta,\zeta}a,r\right\rangle_{\mathcal Q^*,\mathcal Q} &:=\mathfrak a(P^*a,r;z_{\delta,\zeta},\delta), & \left\langle\mathcal A_{QQ}^{\delta,\zeta}q,r\right\rangle_{\mathcal Q^*,\mathcal Q} &:=\mathfrak a(q,r;z_{\delta,\zeta},\delta),
\end{alignedat}
\end{equation*}
where $a,b\in\ell^2(\mathcal I(\omega_0))$ and $q,r\in\mathcal Q$. Recalling \eqref{DtNCapacitance}, we define 
\begin{equation} \label{eq:def-Deff}
\mathcal D_{\mathrm{eff}}(\zeta) := M^{1/2} \bigl(2\omega_0\mathcal C(\omega_0)-\zeta\bigr) M^{1/2} = -G^*\mathcal T^{\omega_0^2}G-\zeta M,
\end{equation}
the second equality being \eqref{DtNCapacitance}. We claim that, uniformly for $\zeta\in\mathbb K$,
\begin{align}
\left\| \mathcal A_{PP}^{\delta,\zeta}-\delta\mathcal D_{\mathrm{eff}}(\zeta) \right\|_{\mathcal B(\ell^2)} &\leq C_{\mathbb K}\delta^2, \label{eq:block-estimate-PP}\\
\left\|\mathcal A_{PQ}^{\delta,\zeta}\right\|_{\mathcal B(\mathcal Q,\ell^2)} &\leq C_{\mathbb K}\delta,
\label{eq:block-estimate-PQ}\\
\left\|\mathcal A_{QP}^{\delta,\zeta}\right\|_{\mathcal B(\ell^2,\mathcal Q^*)} &\leq C_{\mathbb K}\delta,
\label{eq:block-estimate-QP}\\
\left\| \mathcal A_{QQ}^{\delta,\zeta}-\mathcal A_\perp \right\|_{\mathcal B(\mathcal Q,\mathcal Q^*)} &\leq C_{\mathbb K}\delta.\label{eq:block-estimate-QQ}
\end{align}

\medskip 
\noindent 
\emph{Step 3: proof of the block estimates \eqref{eq:block-estimate-PP}--\eqref{eq:block-estimate-QQ}.} Recalling  \eqref{neumanneq}, since $\partial_\nu u_{\vect{j},\ell}=0$ on $\partial D_{\vect{j}}$, Green's identity gives
\begin{equation}
\label{eq:green-neumann}
\int_{D_{\vect{j}}}\nabla u_{\vect{j},\ell}\cdot\nabla\overline w\,\mathrm dx = \frac{\omega_0^2}{v_{\vect{j}}^2}\int_{D_{\vect{j}}}u_{\vect{j},\ell}\overline w\,\mathrm dx, \qquad w\in H^1(D_{\vect{j}}).
\end{equation}
Applying \eqref{eq:green-neumann} with $w=P^*b$ and using the $L^2(D_{\vect{j}})$-orthonormality of $u_{\vect{j},\ell}$, we have 
\begin{equation*}
\mathfrak a(P^*a,P^*b;z_{\delta,\zeta},\delta) = -\delta\zeta\,(Ma,b)_{\ell^2} -\delta \left\langle \mathcal T^{z_{\delta,\zeta}}Ga,Gb \right\rangle .
\end{equation*}
Comparing this with \eqref{eq:def-Deff} gives 
\begin{equation*}
\mathcal A_{PP}^{\delta,\zeta}-\delta\mathcal D_{\mathrm{eff}}(\zeta) = -\delta\,G^*\bigl(\mathcal T^{z_{\delta,\zeta}}-\mathcal T^{\omega_0^2}\bigr)G,
\end{equation*}
and \eqref{eq:block-estimate-PP} follows from \eqref{eq:dtn-first-order-bound} together with the boundedness of $G$.

Next, let $q\in\mathcal Q$, so that $Pq=0$; that is, $q$ is $L^2(D_{\vect{j}})$-orthogonal to all $u_{\vect{j},\ell}$. Noting that $v_{\vect{j}}$ is constant on $D_{\vect{j}}$, the identity \eqref{eq:green-neumann} with $w=q$ shows  
\begin{align*}
\mathfrak a(P^*a,q;z_{\delta,\zeta},\delta) &= -\delta \left\langle \mathcal T^{z_{\delta,\zeta}}[Ga], q |_{\partial \mathcal{D}}\right\rangle,\\
\mathfrak a(q,P^*b;z_{\delta,\zeta},\delta) &= -\delta \left\langle \mathcal T^{z_{\delta,\zeta}} [q |_{\partial \mathcal{D}}],Gb \right\rangle.
\end{align*}
Moreover, the map $\mathcal T^{z_{\delta,\zeta}}$ is also uniformly bounded by Step~1. Then, \eqref{eq:block-estimate-PQ} and
\eqref{eq:block-estimate-QP} follow. Finally, comparing
\eqref{eq_sesquilinear_form_subwavelength_resonance} with
\eqref{eq:def-complementary-operator},
\begin{equation*}
\left\langle\bigl(\mathcal A_{QQ}^{\delta,\zeta}-\mathcal A_\perp\bigr)q,r\right\rangle = -\delta\zeta\sum_{\vect{j}\in\Lambda}\int_{D_{\vect{j}}}\frac{q\overline r}{v_{\vect{j}}^2}\,\mathrm dx -\delta\left\langle\mathcal T^{z_{\delta,\zeta}}[q|_{\partial \mathcal{D}}],  r|_{\partial \mathcal{D}}\right\rangle,
\end{equation*}
we can conclude \eqref{eq:block-estimate-QQ}.

\medskip 
\noindent 
\emph{Step 4: Schur complement and inversion.} By
\cref{lem:complementary-invertibility} and
\eqref{eq:block-estimate-QQ}, a Neumann-series argument shows that
$\mathcal A_{QQ}^{\delta,\zeta}$ is invertible for all sufficiently small $\delta$, uniformly in $\zeta\in\mathbb K$, and
\begin{equation}
\label{eq:QQ-inverse-uniform}
\sup_{\zeta\in\mathbb K} \left\| (\mathcal A_{QQ}^{\delta,\zeta})^{-1}
\right\|_{\mathcal B(\mathcal Q^*,\mathcal Q)} \leq 2C_\perp.
\end{equation}
Its Schur complement is therefore well defined and, by \eqref{eq:block-estimate-PP}--\eqref{eq:block-estimate-QP} and \eqref{eq:QQ-inverse-uniform}, 
\begin{align}
\mathcal S_\delta(\zeta) :={}& \mathcal A_{PP}^{\delta,\zeta}-\mathcal A_{PQ}^{\delta,\zeta}(\mathcal A_{QQ}^{\delta,\zeta})^{-1}\mathcal A_{QP}^{\delta,\zeta}\nonumber\\
={}& \delta\mathcal D_{\mathrm{eff}}(\zeta)+\mathcal O_{\mathcal B(\ell^2)}(\delta^2),
\label{eq:schur-first-order}
\end{align}
uniformly for $\zeta\in\mathbb K$. Since $\mathbb K\Subset\rho(2\omega_0\mathcal C(\omega_0))$ and $M$ is boundedly invertible, the operator
\begin{equation} \label{eq:Deffinverse}
\mathcal D_{\mathrm{eff}}(\zeta)^{-1}=M^{-1/2}(2\omega_0\mathcal C(\omega_0)-\zeta)^{-1}M^{-1/2}
\end{equation}
exists and is bounded uniformly on $\mathbb K$. Again, by Neumann-series, we have
\begin{equation}
\label{eq:schur-inverse}
\mathcal S_\delta(\zeta)^{-1}=\delta^{-1}\mathcal D_{\mathrm{eff}}(\zeta)^{-1}+\mathcal O_{\mathcal B(\ell^2)}(1).
\end{equation}
It follows that 
\begin{equation}
\label{eq:transformed-block-inverse}
\widehat{\mathcal A}_\delta(\zeta)^{-1}=
\begin{pmatrix}
\delta^{-1}\mathcal D_{\mathrm{eff}}(\zeta)^{-1}+\mathcal O_{\mathcal B(\ell^2)}(1)
& \mathcal O_{\mathcal B(\mathcal Q^*,\ell^2)}(1)\\
\mathcal O_{\mathcal B(\ell^2,\mathcal Q)}(1)
& \mathcal O_{\mathcal B(\mathcal Q^*,\mathcal Q)}(1)
\end{pmatrix},
\end{equation}
similarly to \cite[Appendix C.2, (C.13), (C.15)--(C.17)]{ammari2026resolvent_con}.

\medskip 
\noindent 
\emph{Step 5: conclusion.} Since $J$ and $J^\#$ are isomorphisms, \eqref{eq:transformed-block-inverse} implies that $\mathcal A_\delta(z_{\delta,\zeta})$ is invertible. Multiplying \eqref{eq:transformed-block-inverse} by $\delta$ therefore gives
\begin{equation*}
\delta\mathcal R(z_{\delta,\zeta},\delta)=P^*\mathcal D_{\mathrm{eff}}(\zeta)^{-1}P+\mathcal O_{\mathcal B(L^2(\mathcal D),H^1(\mathcal D))}(\delta),
\end{equation*}
uniformly for $\zeta\in\mathbb K$. Substituting \eqref{eq:Deffinverse} and using the continuous embedding $H^1(\mathcal D)\hookrightarrow L^2(\mathcal D)$ proves \eqref{eq_cont_to_disc_resolvent_converge}. We complete the proof. 
\end{proof}

\subsection{Characterizations of defect eigenvalues and bands} \label{sec:characterizations}

In the regular regime, where the interior-resonance conditions \eqref{gapassump1} and \eqref{gapassump1uniform}, together with the exterior gap condition \eqref{gapassump2equiv}, hold, \cref{thm_cont_to_disc_resolvent_converge} shows that the spectral problem \eqref{eq:scattering-problem} near $z=\omega_0^2$ is governed, to leading order in $\delta$, by the frequency-dependent capacitance operator $\mathcal C(\omega_0)$. We now derive, as corollaries of this reduction, 
the spectral asymptotics for compact and line defects. For a compact defect, that is, when $\mathcal J_{\mathrm{def}}$ is finite, each eigenvalue $\lambda$ of $\mathcal C(\omega_0)$ gives rise to a cluster of defect eigenfrequencies, bifurcating from $\omega_0$ with common leading-order shift $\delta\lambda$; see \cref{cor:compact-asymptotics}. For a straight periodic line defect, the quasi-periodic capacitance matrices \eqref{eq:def-freq-cap-operator-Line} play the same role fiberwise and determine the leading-order dispersion of the guided defect bands bifurcating from $\omega_0$; see \cref{cor:line-asymptotics}. These results extend those of \cite{ammariSubwavelengthGuidedModes2021,erik2019} to the non-subwavelength regime.

To make it clear that these defect frequencies and bands lie in a spectral gap of the unperturbed crystal, we first establish the existence of a bulk gap. Let $\mathcal D_{\mathrm{bulk}}:=\bigcup_{\vect{j}\in\Lambda}(D+\vect{j})$ and $\Omega_{\mathrm{bulk}}:=\mathbb R^d\setminus\overline{\mathcal D_{\mathrm{bulk}}}$ be the configuration and its complement for the unperturbed medium. Similarly to the perturbed case \eqref{def:perturbedop}, define the coefficients
\begin{equation*}
a_{\delta,\mathrm{bulk}}:=\delta\mathbbm{1}_{\Omega_{\mathrm{bulk}}}+\mathbbm{1}_{\mathcal D_{\mathrm{bulk}}},
\qquad
m_{\delta,\mathrm{bulk}}:=\frac{\delta}{v^2}\mathbbm{1}_{\Omega_{\mathrm{bulk}}}+\frac{1}{v_b^2}\mathbbm{1}_{\mathcal D_{\mathrm{bulk}}},
\end{equation*}
and let $\mathcal L_\delta^{\mathrm{bulk}}:=-m_{\delta,\mathrm{bulk}}^{-1}\nabla\cdot(a_{\delta,\mathrm{bulk}}\nabla)$ be the elliptic operator of the periodic unperturbed crystal in $L^2(\mathbb R^d,m_{\delta,\mathrm{bulk}}\,\mathrm dx)$. Let $-\Delta_{\mathrm{Dir}}$ be the Dirichlet Laplacian in $\Omega_{\mathrm{bulk}}$, and let $-\Delta_{\vect{\alpha},\mathrm{Dir}}$, $\vect{\alpha}\in Y^*$, denote its Floquet fibers on $Y\setminus\overline D$, with Dirichlet conditions on $\partial D$ and $\vect{\alpha}$-quasi-periodic conditions on $\partial Y$. Thus, the Floquet--Bloch decomposition gives
\begin{equation*}
\sigma(-\Delta_{\mathrm{Dir}};\Omega_{\mathrm{bulk}})
=\overline{\bigcup_{\vect{\alpha}\in Y^*}
\sigma(-\Delta_{\vect{\alpha},\mathrm{Dir}};Y\setminus\overline D)},
\end{equation*}
implying that the bulk Laplacian 
$-\Delta_{\mathrm{Dir}}$ has a purely essential spectrum:
\begin{equation}
\label{eq:bulk-purely-essential}
\sigma(-\Delta_{\mathrm{Dir}};\Omega_{\mathrm{bulk}})=\sigma_{\mathrm{ess}}(-\Delta_{\mathrm{Dir}};\Omega_{\mathrm{bulk}}).
\end{equation}

\begin{lemma}\label{lem:bulk-exterior-spectrum}
For a compact defect, denote the exterior domain by
$\Omega_{\mathrm{comp}}:=\mathbb R^d\setminus\overline{\mathcal D}$.
Then
\begin{equation}
\label{eq:exterior-compact-essential-spectrum}
\sigma_{\mathrm{ess}}(-\Delta_{\mathrm{Dir}};\Omega_{\mathrm{comp}})
=\sigma_{\mathrm{ess}}(-\Delta_{\mathrm{Dir}};\Omega_{\mathrm{bulk}})
=\sigma(-\Delta_{\mathrm{Dir}};\Omega_{\mathrm{bulk}}).
\end{equation}
Suppose that $d=2$ and that the line defect is of the form \eqref{eq:defectindex}. Denote the corresponding exterior domain by $\Omega_{\mathrm{line}}:=\mathbb R^2\setminus\overline{\mathcal D}$, so that $\widetilde\Omega_{\mathrm{line}}=\widetilde Y\cap\Omega_{\mathrm{line}}$, and let $\widetilde\Omega_{\mathrm{bulk}}:=\widetilde Y\cap\Omega_{\mathrm{bulk}}$. 
For every $\alpha\in\widetilde Y^* \simeq [0,2\pi)$, let $-\Delta_{\alpha,\mathrm{Dir}}$ denote the Dirichlet Laplacian with $\alpha$-quasi-periodic boundary conditions on its lateral boundaries. Then
\begin{equation}
\label{eq:exterior-line-fiber-essential-spectrum}
\sigma_{\mathrm{ess}}(-\Delta_{\alpha,\mathrm{Dir}};\widetilde\Omega_{\mathrm{line}}) =\sigma_{\mathrm{ess}}(-\Delta_{\alpha,\mathrm{Dir}};\widetilde\Omega_{\mathrm{bulk}})
=\sigma(-\Delta_{\alpha,\mathrm{Dir}};\widetilde\Omega_{\mathrm{bulk}}).
\end{equation}
Therefore, 
\begin{equation}
\label{eq:bulk-exterior-spectrum-inclusion}
\begin{aligned}
\sigma(-\Delta_{\mathrm{Dir}};\Omega_{\mathrm{bulk}})
&\subset\sigma(-\Delta_{\mathrm{Dir}};\Omega_{\mathrm{comp}})
&&\text{for a compact defect},\\
\sigma(-\Delta_{\mathrm{Dir}};\Omega_{\mathrm{bulk}})
&\subset\sigma(-\Delta_{\mathrm{Dir}};\Omega_{\mathrm{line}})
&&\text{for a straight line defect}.
\end{aligned}
\end{equation}
\end{lemma}

\begin{proof}
For the compact defect, by definition, the domains $\Omega_{\mathrm{comp}}$ and $\Omega_{\mathrm{bulk}}$ coincide outside a compact set. Then, the associated essential spectrum of $-\Delta_{\mathrm{Dir}}$ is stable, by \cite[Theorem~2.3]{lenzDecompositionPrinciple2019} (see also \cite{figotinLocalizedClassicalWaves1997}), which, together with \eqref{eq:bulk-purely-essential}, proves \eqref{eq:exterior-compact-essential-spectrum} and the first inclusion in \eqref{eq:bulk-exterior-spectrum-inclusion}. 

For the line defect, the partial Floquet transform along $e_1$ reduces the problem, for each $\alpha\in\widetilde Y^*$, to the Dirichlet Laplacians on $\widetilde\Omega_{\mathrm{line}}$ and $\widetilde\Omega_{\mathrm{bulk}}$. These strip geometries coincide outside a bounded region in the transverse direction. Similarly, since the essential spectrum is stable under compact perturbations, we have $\sigma_{\mathrm{ess}}(-\Delta_{\alpha,\mathrm{Dir}};\widetilde\Omega_{\mathrm{line}}) =\sigma_{\mathrm{ess}}(-\Delta_{\alpha,\mathrm{Dir}};\widetilde\Omega_{\mathrm{bulk}})$; see \cite[Section~3.1]{ammariSubwavelengthGuidedModes2021}. Then, \eqref{eq:exterior-line-fiber-essential-spectrum} follows by the transverse periodicity of $\widetilde\Omega_{\mathrm{bulk}}$. Taking the unions over $\alpha$ gives the second inclusion in \eqref{eq:bulk-exterior-spectrum-inclusion}, completing the proof. 
\end{proof}

It follows from \cref{lem:bulk-exterior-spectrum} and \eqref{gapassump2equiv} that, for the compact and straight-line defects considered above, the following exterior bulk gap condition holds:
\begin{equation} \label{eq:bulk-exterior-gap}
d_{\mathrm{ext}}^{\mathrm{bulk}}:=\operatorname{dist}\left(\frac{\omega_0^2}{v^2},\sigma(-\Delta_{\mathrm{Dir}};\Omega_{\mathrm{bulk}})\right)>0.
\end{equation}
For a general noncompact defect, \eqref{eq:bulk-exterior-gap} should instead be assumed separately. We next show that this condition opens a bulk gap of fixed width around $\omega_0^2$ for all sufficiently small contrasts.

\begin{lemma} \label{lem:uniform-bulk-gap}
Suppose that
\begin{equation*}
\frac{\omega_0^2}{v_b^2}\notin\sigma(-\Delta_{\mathrm{Neu}};D)
\end{equation*}
and that \eqref{eq:bulk-exterior-gap} holds. Then there exist $\eta_{\mathrm b}>0$ and $\delta_{\mathrm b}>0$ such that
\begin{equation}
\label{eq:uniform-bulk-gap}
I_{\mathrm b}:=(\omega_0^2 - \eta_{\mathrm b},\omega_0^2 + \eta_{\mathrm b})\subset\rho(\mathcal L_\delta^{\mathrm{bulk}})
\end{equation}
for every $0<\delta<\delta_{\mathrm b}$.
\end{lemma}

\begin{proof}
We let $d_{\mathrm{int}}^{\mathrm{bulk}}:=\operatorname{dist}\left(\omega_0^2/v_b^2,\sigma(-\Delta_{\mathrm{Neu}};D)\right)>0$, and then choose $\eta_{\mathrm b}>0$ sufficiently small such that for $|z-\omega_0^2|\leq \eta_{\mathrm b}$,
\begin{equation}
\label{eq:uniform-bulk-separation}
\begin{aligned}
\operatorname{dist}\left(\frac{z}{v_b^2},\sigma(-\Delta_{\mathrm{Neu}};D)\right)&\geq\frac{d_{\mathrm{int}}^{\mathrm{bulk}}}{2},\\
\inf_{\vect{\alpha}\in Y^*}\operatorname{dist}\left(\frac{z}{v^2},\sigma \big(-\Delta_{\vect{\alpha},\mathrm{Dir}};Y\setminus\overline D \big)\right)&\geq\frac{d_{\mathrm{ext}}^{\mathrm{bulk}}}{2}.
\end{aligned}
\end{equation}
For $\vect{\alpha}\in Y^*$, let $\mathcal T_{\mathrm{bulk}}^{\vect{\alpha},z}$ denote the corresponding quasi-periodic exterior DtN map on $\partial D$. Similarly to \eqref{eq_variational_characterization_subwavelength_resonance}, we can define the operator pencil
\begin{equation*}
\mathcal A_{\delta,\mathrm{bulk}}^{\vect{\alpha}}(z)=\mathcal A_{\mathrm{int}}(z)-\delta\mathcal B^{\vect{\alpha}}(z):H^1(D)\longrightarrow H^1(D)^*,
\end{equation*}
such that its characteristic values give the bulk eigenvalues, where
\begin{align*}
\left\langle\mathcal A_{\mathrm{int}}(z)u,w\right\rangle &:=\int_D\left(\nabla u\cdot\nabla\overline w-\frac{z}{v_b^2}u\overline w\right)\,\mathrm dx,\\
\left\langle\mathcal B^{\vect{\alpha}}(z)u,w\right\rangle &:=\left\langle\mathcal T_{\mathrm{bulk}}^{\vect{\alpha},z}[u|_{\partial D}],w|_{\partial D}\right\rangle_{H^{-1/2}(\partial D),H^{1/2}(\partial D)}.
\end{align*}
The first estimate in \eqref{eq:uniform-bulk-separation} gives a uniform bound for $\mathcal A_{\mathrm{int}}(z)^{-1}$ from $H^1(D)^*$ to $H^1(D)$. The second estimate, the compactness of $Y^*$, and an analogous version of \cref{prop_d2n_map} for $\mathcal T_{\mathrm{bulk}}^{\vect{\alpha},z}$ give a uniform bound for $\mathcal B^{\vect{\alpha}}(z)$ from $H^1(D)$ to $H^1(D)^*$. It follows that 
\begin{equation*}
\mathcal A_{\delta,\mathrm{bulk}}^{\vect{\alpha}}(z)=\mathcal A_{\mathrm{int}}(z)\left(I-\delta\mathcal A_{\mathrm{int}}(z)^{-1}\mathcal B^{\vect{\alpha}}(z)\right)
\end{equation*}
is uniformly invertible in $\vect{\alpha}$ and $z$ for all sufficiently small $\delta$. Recalling the Floquet--Bloch decomposition, we have proved \eqref{eq:uniform-bulk-gap}.
\end{proof}

\subsubsection{Compact defects}
\label{sec:compact_defect}

In the case of compact defects where $\mathcal J_{\mathrm{def}}$ is finite, we have that  $\mathcal I(\omega_0)$ is finite and $\mathcal C(\omega_0)$ is a finite-dimensional Hermitian matrix, and the following holds.

\begin{corollary} \label{cor:compact-asymptotics}
Let $\omega_0 >0 $ satisfy \eqref{gapassump1}, \eqref{gapassump1uniform}, and \eqref{gapassump2equiv}. 
Let $\lambda$ be an eigenvalue of $\mathcal C(\omega_0)$ of multiplicity $m$. For all sufficiently small $\delta>0$, there are exactly $m$ defect eigenfrequencies, counted with multiplicity, that bifurcate from $\omega_0$ with leading-order correction $\lambda$. More precisely, we have 
\begin{equation} \label{asymp:compact}
\omega_{\delta,k} = \omega_0+\delta\lambda+\mathcal O(\delta^2), \qquad k=1,\dots,m.
\end{equation}
Moreover, for all sufficiently small $\delta$, 
\begin{equation} \label{eq:ingaocond}
    \omega_{\delta,k}^2 \in I_{\mathrm b} \subset\rho(\mathcal L_\delta^{\mathrm{bulk}})\,,
\end{equation}
where $I_{\mathrm b}$ is given in \eqref{eq:uniform-bulk-gap}.  Thus, these defect eigenfrequencies lie in a spectral gap of the unperturbed bulk crystal.
\end{corollary}

\begin{proof}
Let $\zeta_0:=2\omega_0\lambda$ and choose a closed disk $\overline B$ centered at $\zeta_0$ such that its boundary contains no point of $2\omega_0\sigma(\mathcal C(\omega_0))$ and its interior contains no other distinct eigenvalue. The block estimates in the proof of \cref{thm_cont_to_disc_resolvent_converge} remain uniformly valid for $\zeta\in\overline B$. In particular, we have 
\begin{equation}
\label{eq:compact-schur-analytic}
\delta^{-1}\mathcal S_\delta(\zeta) = \mathcal D_{\mathrm{eff}}(\zeta)
+\delta\mathcal E_\delta(\zeta), \qquad
\sup_{\zeta\in\overline B\,,\ 0<\delta<\delta_0} \|\mathcal E_\delta(\zeta)\|<\infty.
\end{equation}
Recall that $\mathcal D_{\mathrm{eff}}(\zeta) = M^{1/2}\bigl(2\omega_0\mathcal C(\omega_0)-\zeta\bigr)M^{1/2}$ and that $\det\mathcal D_{\mathrm{eff}}(\zeta)$ has a zero of order $m$ at $\zeta_0$ and no other zero in $B$. Since the matrices are finite-dimensional, the uniform estimate in \eqref{eq:compact-schur-analytic} and the continuity of the determinant give 
\begin{equation*}
\sup_{\zeta\in\partial B}\left| \det\bigl(\delta^{-1}\mathcal S_\delta(\zeta)\bigr) -\det\mathcal D_{\mathrm{eff}}(\zeta) \right| = \mathcal{O}(\delta). 
\end{equation*}
Rouch\'e's theorem therefore shows that $\mathcal S_\delta$ has exactly $m$ characteristic values $\zeta_{\delta,1},\dots,\zeta_{\delta,m}$ in $B$, counted with algebraic multiplicity.

If $\mathcal S_\delta(\zeta_{\delta,k})a_{\delta,k}=0$ and $\|a_{\delta,k}\|_{\ell^2}=1$, then \eqref{eq:compact-schur-analytic} gives
\begin{equation*}
\|\mathcal D_{\mathrm{eff}}(\zeta_{\delta,k})a_{\delta,k}\|_{\ell^2} = \mathcal{O}(\delta). 
\end{equation*}
Letting $b_{\delta,k}:=M^{1/2}a_{\delta,k}$, the above estimate implies $\left\| \bigl(2\omega_0\mathcal C(\omega_0)-\zeta_{\delta,k}\bigr) b_{\delta,k} \right\|_{\ell^2}  = \mathcal{O}(\delta)$, 
which further gives, by the spectral theorem for the self-adjoint operator
$2\omega_0\mathcal C(\omega_0)$, 
\begin{equation*}
\operatorname{dist}\bigl(\zeta_{\delta,k},2\omega_0\sigma(\mathcal C(\omega_0))
\bigr) = \mathcal{O}(\delta). 
\end{equation*}
It follows that 
\begin{equation*}
\omega_{\delta,k}^2 = \omega_0^2+2\delta\omega_0\lambda+\mathcal O(\delta^2). 
\end{equation*}
Since $\omega_0>0$, taking the positive square root proves \eqref{asymp:compact}. Combining this bound with \cref{lem:uniform-bulk-gap} proves the in-gap condition \eqref{eq:ingaocond}, completing the proof. 
\end{proof}

\subsubsection{Line defects}
Let $\widetilde{m}_\delta:= m_\delta|_{\widetilde{Y}}$ and let $\mathcal L_\delta^\alpha$ be the self-adjoint operator in $L^2(\widetilde Y, \widetilde{m}_\delta\,\mathrm dx)$ associated with the restriction of the form \eqref{def:formLdelta} to the quasi-periodic space:
\begin{equation*}
H^1_\alpha(\widetilde Y):=\left\{u\in H^1(\widetilde Y): u|_{\Gamma_+}=e^{\mathrm i\alpha}u|_{\Gamma_-}\right\}.
\end{equation*}
For $\vect{\alpha}=(\alpha_1,\alpha_2)\in Y^*$, let $\mathcal L_{\delta,\mathrm{bulk}}^{\vect{\alpha}}$ denote the full Floquet--Bloch fiber of $\mathcal L_\delta^{\mathrm{bulk}}$. Then, the projected bulk spectrum associated with $\alpha\in\widetilde Y^*$ is given by 
\begin{equation} \label{eq:projected-bulk-spectrum}
\Sigma_{\mathrm{bulk},\delta}(\alpha) := \overline{\bigcup \{\sigma\bigl(\mathcal L_{\delta,\mathrm{bulk}}^{\vect{\alpha}}\bigr)\,;\ \vect{\alpha}=(\alpha_1,\alpha_2)\in Y^*\,,\  \alpha_1=\alpha}\}.
\end{equation}
Similarly to \cref{lem:bulk-exterior-spectrum}, we have 
\begin{equation}
\label{eq:line-essential-spectrum}
\sigma_{\mathrm{ess}}(\mathcal L_\delta^\alpha) = \Sigma_{\mathrm{bulk},\delta}(\alpha) \subset \sigma(\mathcal L_\delta^{\mathrm{bulk}}), \qquad \alpha\in\widetilde Y^*. 
\end{equation}
Hence, by \Cref{lem:uniform-bulk-gap}, $I_{\mathrm b}$ is disjoint from $\Sigma_{\mathrm{bulk},\delta}(\alpha)$ uniformly in $\alpha$ for all sufficiently small $\delta$.

\begin{corollary} \label{cor:line-asymptotics}
For the line-defect setting in \Cref{subsec:QP}, let $\omega_0>0$ satisfy \eqref{gapassump1}, \eqref{gapassump1uniform}, and \eqref{gapassump2equiv}. Let $\lambda_1^\alpha\leq\cdots\leq\lambda_{m_\star}^\alpha$ be the ordered eigenvalues of $\mathcal C^\alpha(\omega_0)$, counted with multiplicity. For any small enough $\delta > 0$ and $\alpha\in\widetilde Y^*$, the operator $\mathcal L_\delta^\alpha$ has exactly $m_\star$ eigenvalues $z_{\delta,j}^\alpha := \bigl(\omega_j^\alpha(\delta)\bigr)^2$, $j=1,\dots,m_\star$, bifurcating from $\omega_0^2$, counted with multiplicity. Moreover, $\omega_j^\alpha(\delta) > 0$ are band functions that are continuous in $\alpha$ and satisfy
\begin{equation} \label{asymp:line}
\sup_{\alpha\in\widetilde Y^*}\max_{1\leq j\leq m_\star} \left|\omega_j^\alpha(\delta)-\omega_0
-\delta\lambda_j^\alpha\right| = \mathcal{O}(\delta^2),
\end{equation}
and lie in a projected bulk gap uniformly in $\alpha$ and $j$:
\begin{equation} \label{eq:line-in-gap}
z_{\delta,j}^\alpha \in\left(\omega_0^2-\frac{\eta_{\mathrm b}}2, \omega_0^2+\frac{\eta_{\mathrm b}}2\right)\subset I_{\mathrm b},
\qquad
\operatorname{dist}\left(z_{\delta,j}^\alpha, \Sigma_{\mathrm{bulk},\delta}(\alpha)\right) > 0\,.
\end{equation}
In particular, we have 
\begin{equation*}
\bigcup\Big\{z_{\delta,j}^\alpha\,;\ \alpha\in\widetilde Y^*\Big\} \subset\sigma(\mathcal L_\delta) \bigcap I_{\mathrm b}, \qquad j=1,\dots,m_\star.
\end{equation*}
If $\alpha\mapsto\lambda_j^\alpha$ is nonconstant, then $\omega_j^\alpha(\delta)$ is not a flat band for all sufficiently small $\delta$.
\end{corollary}

\begin{proof}
For each fixed $\alpha$, we apply the block decomposition and Schur-complement argument used in the proof of \cref{thm_cont_to_disc_resolvent_converge}. The exterior gap condition \eqref{gapassump2star} and the compactness of $\widetilde Y^*$ ensure that all estimates are uniform in $\alpha$. The standard perturbation theory for Hermitian matrix pencils then yields exactly $m_\star$ real roots, counted with multiplicity, and the uniform expansion \eqref{asymp:line}.
The in-gap conclusion \eqref{eq:line-in-gap} follows from \cref{lem:uniform-bulk-gap} and \eqref{eq:line-essential-spectrum}, while the nonflatness assertion follows directly from \eqref{asymp:line}. The remaining details follow the arguments used in the proofs of \cref{thm_cont_to_disc_resolvent_converge,cor:compact-asymptotics} and are omitted.
\end{proof}

\subsection{Exponential decay of the off-diagonal coefficients}\label{sec:exponentialdecay}

In this section, we show that the frequency-dependent capacitance operator $\mathcal{C}(\omega_0)$ exhibits an exponential off-diagonal decay rate. Intuitively, the element in the row $(\vect{j},\ell)$ and the column $(\vect{j}',\ell')$ measures the normal flux induced on $\partial D_{\vect{j}}$ by the exterior field generated from $\partial D_{\vect{j}'}$, then the exterior gap makes this interaction exponentially local. 

We begin with a localized resolvent estimate, following from the standard Combes--Thomas argument; see, for example, \cite[Lemma~12]{figotin1996localization}. 
Let $k_0^2:=\omega_0^2/v^2$. Since $D,\widetilde D\Subset Y$, we choose neighborhoods $\mathcal O_{\vect{j}}\Subset Y+\vect{j}$ and cutoff functions $\chi_{\vect{j}}\in W^{1,\infty}(\mathbb R^d)$ such that
\begin{equation*}
0\leq\chi_{\vect{j}}\leq1,
\qquad
\chi_{\vect{j}}=1\text{ near }\partial D_{\vect{j}},
\qquad
\operatorname{supp}\chi_{\vect{j}}\subset\mathcal O_{\vect{j}},
\end{equation*}
and $\sup_{\vect{j}}\|\chi_{\vect{j}}\|_{W^{1,\infty}}<\infty$. 
For a domain $\Omega$, we denote by $H^{-1}(\Omega)$ the anti-dual of $H_0^1(\Omega)$, where $H_0^1(\Omega)$ is equipped with its full $H^1$-norm.

\begin{lemma} \label{lem:localized-CT}
Under \eqref{gapassump2equiv}, there exist $C,\beta>0$ such that for all $\vect{j},\vect{j}'\in\Lambda$, 
\begin{equation} \label{eq:localized-CT}
\left\| \chi_{\vect{j}} (-\Delta_{\mathrm{Dir}} -k_0^2)^{-1} \chi_{\vect{j}'} \right\|_{\mathcal B(H^{-1}(\mathbb R^d\setminus\overline{\mathcal D}), H_0^1(\mathbb R^d\setminus\overline{\mathcal D}))} \leq Ce^{-\beta|\vect{j}-\vect{j}'|}.
\end{equation}
\end{lemma}

The main result of this section is as follows. 

\begin{theorem}
\label{thm:decay}
Suppose the relation \eqref{gapassump2equiv} holds. Then there exist $C,\beta>0$ such that
\begin{equation}
\label{eq_cap_decay}
\left| \mathcal C_{(\vect{j},\ell),(\vect{j}',\ell')}(\omega_0) \right| \leq Ce^{-\beta|\vect{j}-\vect{j}'|}
\end{equation}
for all $(\vect{j},\ell),(\vect{j}',\ell')\in\mathcal I(\omega_0)$.
\end{theorem}

\begin{proof}
We first introduce uniformly bounded local right inverses of the trace maps. More precisely, for each
$\vect{j}\in\Lambda$, we choose
\begin{equation*}
\mathcal E_{\vect{j}}:H^{1/2}(\partial D_{\vect{j}}) \longrightarrow H^1(\mathbb R^d\setminus\overline{\mathcal D})
\end{equation*}
such that $\mathcal E_{\vect{j}}h$ has trace $h$ on $\partial D_{\vect{j}}$, vanishes on all other boundary components, and satisfies
\begin{equation*}
\|\mathcal E_{\vect{j}}h\|_{H^1(\mathbb R^d\setminus\overline{\mathcal D})} \leq C\|h\|_{H^{1/2}(\partial D_{\vect{j}})}\,, \qquad {\rm supp}(\mathcal E_{\vect{j}}h) \subset \{\chi_{\vect{j}}=1\}
\end{equation*}
with $C$ independent of $\vect{j}$. Using the zero-extension convention for $u_{\vect{j},\ell}$ introduced after \eqref{def:activeindex}, the $L^2$-normalization, \eqref{neumanneq}, and the uniform trace inequality give
\begin{equation*}
\sup_{(\vect{j},\ell)\in\mathcal I(\omega_0)} \|u_{\vect{j},\ell}\|_{H^{1/2}(\partial D_{\vect{j}})}<\infty. 
\end{equation*}
Define compactly supported functions:
\begin{equation*}
f_{\vect{j},\ell} :=(\Delta+k_0^2)\mathcal E_{\vect{j}}u_{\vect{j},\ell} \in H^{-1}(\mathbb R^d\setminus\overline{\mathcal D})\,, \qquad {\rm supp}(f_{\vect{j},\ell}) \subset \{\chi_{\vect{j}}=1\},
\end{equation*}
satisfying
\begin{equation} \label{eq:uniform-local-source}
\sup_{(\vect{j},\ell)\in\mathcal I(\omega_0)} \|f_{\vect{j},\ell}\|_{H^{-1}(\mathbb R^d\setminus\overline{\mathcal D})} <\infty.
\end{equation}
Noting that $U_{\vect{j},\ell}-\mathcal E_{\vect{j}}u_{\vect{j},\ell}$ has zero trace on $\partial\mathcal D$, we have 
\begin{equation}  \label{eq_resolvent_expression_cap}
U_{\vect{j},\ell} = \mathcal E_{\vect{j}}u_{\vect{j},\ell} +(-\Delta_{\mathrm{Dir}} -k_0^2)^{-1} f_{\vect{j},\ell}.
\end{equation}
For $\vect{j}\neq\vect{j}'$, the local lifting $\mathcal E_{\vect{j}'}u_{\vect{j}',\ell'}$ associated with the source resonator $D_{\vect{j}'}$ vanishes in a neighborhood of the target $D_{\vect{j}}$. Applying \cref{lem:localized-CT} to the remaining resolvent term in \eqref{eq_resolvent_expression_cap} gives
\begin{equation}
\label{eq:field-localization}
\|\chi_{\vect{j}}U_{\vect{j}',\ell'}\|_{H^1(\mathbb R^d\setminus\overline{\mathcal D})} \leq Ce^{-\beta|\vect{j}-\vect{j}'|}, \qquad \vect{j}\neq\vect{j}'.
\end{equation}
Since $U_{\vect{j}',\ell'}$ satisfies the Helmholtz equation in a neighborhood of $\partial D_{\vect{j}}$, the local normal trace estimate yields
\begin{equation*}
\left\| \left.\partial_\nu U_{\vect{j}',\ell'}\right|_+
\right\|_{H^{-1/2}(\partial D_{\vect{j}})} \leq Ce^{-\beta|\vect{j}-\vect{j}'|}, \qquad \vect{j}\neq\vect{j}'\,,
\end{equation*}
with constants $C$ and $\beta$ uniform in $\vect{j}$. 
Substitution into \eqref{eq:def-freq-cap-operator-entry} proves \eqref{eq_cap_decay} when $\vect{j}\neq\vect{j}'$. The elements $\mathcal C_{(\vect{j},\ell),(\vect{j}',\ell')}$ with $\vect{j}=\vect{j}'$ are uniformly bounded by \cref{prop:cap-basic}, completing the proof.
\end{proof}

\section{Patch approximation and local computation}\label{sec:pathch}

The exponential locality established in \cref{thm:decay} has two distinct computational consequences. First, the full capacitance operator is exponentially compressible: an \emph{interaction truncation radius} $N$ of logarithmic size suffices for a given operator accuracy. Second, each retained coefficient can be approximated by solving a local Helmholtz problem on a \emph{patch of radius} $R$. We will establish these facts in this section. 

At the reference frequency $\omega_0$, we prove that the local construction converges exponentially in $N$ and $R$, provided that the growing patch problems satisfy the uniform resolvent bound \eqref{eq:patch-stability}; see \cref{lem:dipatch}. This additional stability condition does not follow from the infinite-volume exterior gap \eqref{gapassump2equiv}, since the artificial patch boundary may introduce eigenvalues close to the reference frequency. 

We also introduce the auxiliary complex shift $\omega_0^2+\mathrm{i}\gamma$. The resulting absorption makes every outgoing finite-cluster problem in free space invertible with $\mathcal O(\gamma^{-1})$ resolvent bound; see \cref{lem:uniform-complex-stability}. Hermitian symmetrization then eliminates the first-order regularization bias, yielding an exponentially convergent stabilized approximation without any real-frequency finite-cluster stability assumption; see \cref{thm:stabilized-finite-cluster}. These results lead to the error and complexity estimates for approximating defect eigenfrequencies established in \Cref{subsec:spectral-complexity}.  

We first establish the exponential compressibility of $\mathcal{C}(\omega_0)$ as a consequence of \cref{thm:decay}. 
Since only finitely many resonator types occur (\Cref{subsec:continuous-problem}), we have 
\begin{equation} \label{eq:boundmj}
    m_{\max}:=\sup_{\vect{j}\in\Lambda}m_{\vect{j}}<\infty,
\end{equation}
for the multiplicity $m_{\vect{j}}$ in \eqref{def:neumann_space}. For $N\geq0$, define the truncated operator by
\begin{equation} \label{eq:exact-truncation}
\bigl(\mathcal C_N^{\mathrm{tr}}(\omega_0)\bigr)_{(\vect{j},\ell),(\vect{j}',\ell')} :=
\begin{cases}
    \mathcal C_{(\vect{j},\ell),(\vect{j}',\ell')}(\omega_0),&|\vect{j}-\vect{j}'|\leq N,\\ 
0,&|\vect{j}-\vect{j}'|>N.
\end{cases}
\end{equation}

\begin{proposition}[Exponential operator compression]
\label{prop:exact-truncation}
Under assumption \eqref{gapassump2equiv}, there exist $C,\mu>0$ such that
\begin{equation} \label{eq:exact-truncation-error}
\left\|\mathcal C(\omega_0)-\mathcal C_N^{\mathrm{tr}}(\omega_0)\right\|_{
\mathcal B(\ell^2(\mathcal I(\omega_0)))}
\leq Ce^{-\mu N}.
\end{equation}
Hence, $N=\mathcal O(\log\varepsilon^{-1})$ suffices to achieve the error $\mathcal{O}(\varepsilon)$. 
\end{proposition}

\begin{proof}
For each $\vect{j}'\in\Lambda$, the number of lattice sites in the shell $\{\vect{j}:n\leq|\vect{j}-\vect{j}'|<n+1\}$ is bounded by $\mathcal{O}((1+n)^{d-1})$. Hence, the decay \eqref{eq_cap_decay} and the bound \eqref{eq:boundmj} on $m_{\vect{j}}$ give
\begin{equation} \label{eq:sumrowcolume}
\sup_{(\vect{j}',\ell')\in\mathcal I(\omega_0)} \sum_{\substack{(\vect{j},\ell)\in\mathcal I(\omega_0)\\|\vect{j}-\vect{j}'|>N}}
\left|\mathcal C_{(\vect{j},\ell),(\vect{j}',\ell')}(\omega_0)\right| \leq
C\sum_{n>N}(1+n)^{d-1}e^{-\beta n}
\leq Ce^{-\mu N}
\end{equation}
for a constant $C$ and any fixed $0<\mu<\beta$. The same estimate
holds for the row sums, and Schur's test
\cite[Chapter~4]{halmosSunderBoundedIntegral1978} proves
\eqref{eq:exact-truncation-error}.
\end{proof}

Thus, $\mathcal C(\omega_0)$ is exponentially approximable in operator norm by finite-range block operators. For a straight line defect, $\mathcal C_N^{\mathrm{tr}}(\omega_0)$ is a block-banded operator of bandwidth $N$.

\subsection{Real-frequency local approximations}
\label{subsec:real-frequency-patches}

\Cref{prop:exact-truncation} has shown that only interactions between nearby resonators need to be retained. However, evaluating each retained coefficient $\mathcal C_{(\vect{j},\ell),(\vect{j}',\ell')}(\omega_0)$ still requires solving a global exterior problem and is therefore computationally expensive. To obtain a fully localized approximation, we replace the global fields defining these coefficients by fields computed on bounded patches. Unlike the harmonic local problems considered in \cite{ammari2026resolvent_con}, the local problems arising here are Helmholtz problems at the nonzero real frequency $\omega_0$. Their convergence therefore requires an additional uniform resolvent bound for the growing patch problems.

Following \cite{ammari2026resolvent_con}, for simplicity we take $\mathsf P_{\vect{j}}^R\subset\mathbb R^d$ to be a bounded cube or rectangular box formed by a finite union of lattice cells, centered at $D_{\vect{j}}$, and chosen to contain every cell $Y+\vect{k}$ with $|\vect{k}-\vect{j}|\leq R$. We set
\begin{equation*}
\Omega_{\vect{j}}^R:=\mathsf P_{\vect{j}}^R\cap\bigl(\mathbb R^d\setminus\overline{\mathcal D}\bigr).
\end{equation*}
The perforated exterior domain $\Omega_{\vect{j}}^R$ has artificial outer boundary $\Gamma_{\vect{j}}^R:=\partial\mathsf P_{\vect{j}}^R$. Let $\nu_\Gamma$ denote the unit normal on $\Gamma_{\vect{j}}^R$ pointing outward from $\mathsf P_{\vect{j}}^R$. We choose the patches so that, for constants $c,C>0$ independent of $\vect{j}$ and $R$,
\begin{equation*}
\operatorname{dist}(D_{\vect{j}},\Gamma_{\vect{j}}^R)\geq cR,
\end{equation*}
and that each patch is contained in a ball of radius $\mathcal{O}(R)$ centered at $D_{\vect{j}}$. Then, each patch contains at most $\mathcal{O}(R^d)$ resonators, and each resonator belongs to at most $\mathcal{O}(R^d)$ patches of radius $R$.

For $(\vect{j}',\ell')\in\mathcal I(\omega_0)$, define the local source field $U_{\vect{j}',\ell'}^R$ by
\begin{equation}
\begin{cases}
\Big(\Delta+\dfrac{\omega_0^2}{v^2}\Big)U_{\vect{j}',\ell'}^R=0 &\text{in }\Omega_{\vect{j}'}^R,\\
U_{\vect{j}',\ell'}^R=u_{\vect{j}',\ell'}&\text{on }\partial D_{\vect{j}'},\\
U_{\vect{j}',\ell'}^R=0&\text{on the other resonator boundaries in }\mathsf P_{\vect{j}'}^R,\\
U_{\vect{j}',\ell'}^R=0&\text{on }\Gamma_{\vect{j}'}^R.
\end{cases}
\label{eq:local-patch}
\end{equation}
For integers $1\leq N\leq R$ and $(\vect{j},\ell)\in\mathcal I(\omega_0)$, define
\begin{equation}
\label{def:patch}
\bigl(\mathcal C_{N,R}^{\mathrm{patch}}(\omega_0)\bigr)_{(\vect{j},\ell),(\vect{j}',\ell')} :=
\begin{cases}
-\dfrac{v_{\vect{j}}v_{\vect{j}'}}{2\omega_0} \Big\langle \left.\partial_\nu U_{\vect{j}',\ell'}^R\right|_+, u_{\vect{j},\ell} \Big\rangle_{H^{-1/2}(\partial D_{\vect{j}}),H^{1/2}(\partial D_{\vect{j}})}, &|\vect{j}-\vect{j}'|\leq N,\\[2mm]
0,&|\vect{j}-\vect{j}'|>N.
\end{cases}
\end{equation}
We now introduce the required stability assumption, \emph{i.e.}, a uniform positive distance from $\omega_0^2/v^2$ to the Dirichlet spectra of patches,
\begin{equation}
\label{eq:patch-stability}
    \sup_{\vect{j} \in\Lambda,R\geq1} \Big\| \Big(-\Delta_{\mathrm{Dir},\Omega_{\vect{j}}^R} -\frac{\omega_0^2}{v^2}\Big)^{-1} \Big\|_{H^{-1}(\Omega_{\vect{j}}^R)\to H^1_0(\Omega_{\vect{j}}^R)} <\infty.
\end{equation}

\begin{theorem}[Exponential convergence of the patch approximation] \label{lem:dipatch}
Suppose \eqref{gapassump2equiv} and \eqref{eq:patch-stability} hold. Then there exist $C,\beta,\mu>0$ such that, for all sufficiently large $R$ and all integers $1\leq N\leq R$,
\begin{equation}
\label{est:dirpatch}
\left\| \mathcal C(\omega_0)-\mathcal C_{N,R}^{\mathrm{patch}}(\omega_0) \right\|_{\mathcal B(\ell^2(\mathcal I(\omega_0)))} \leq C\left(e^{-\beta N}+e^{-\mu R}\right).
\end{equation}
\end{theorem}

\begin{proof}
For $(\vect{j}',\ell')\in\mathcal I(\omega_0)$, the localized Combes--Thomas argument used to prove \eqref{eq:field-localization}, together with local elliptic and trace estimates, gives
\begin{equation} \label{field:est}
\|U_{\vect{j}',\ell'}\|_{H^{1/2}(\Gamma_{\vect{j}'}^R)}
+\|\partial_{\nu_\Gamma}U_{\vect{j}',\ell'}\|_{H^{-1/2}(\Gamma_{\vect{j}'}^R)}
\leq Ce^{-\mu_0R},
\end{equation}
for some $\mu_0>0$; the polynomial growth in the number of boundary cells has been absorbed by decreasing the exponential rate. The difference $W_{\vect{j}',\ell'}^R:=U_{\vect{j}',\ell'}-U_{\vect{j}',\ell'}^R$ solves the homogeneous Helmholtz equation in $\Omega_{\vect{j}'}^R$, vanishes on every resonator boundary, and has trace $U_{\vect{j}',\ell'}$ on $\Gamma_{\vect{j}'}^R$. 
Similarly to the proof of \cref{thm:decay}, subtracting the lifting of $U_{\vect{j}',\ell'}|_{\Gamma_{\vect{j}'}^R}$ gives an element of $H_0^1(\Omega_{\vect{j}'}^R)$, and then applying \eqref{eq:patch-stability} to this difference yields
\begin{equation*}
\|W_{\vect{j}',\ell'}^R\|_{H^1(\Omega_{\vect{j}'}^R)} \leq Ce^{-\mu_0R}.
\end{equation*}
The local normal-trace estimate therefore shows that each retained coefficient is approximated with an error bounded by $Ce^{-\mu_0R}$. Since $N\leq R$, the polynomial growth in the number of retained coefficients can be absorbed into $e^{-\mu R}$ for some $0<\mu<\mu_0$. The Schur row and column sums of the omitted coefficients are bounded by $Ce^{-\beta N}$ by \eqref{eq_cap_decay}. Schur's test and \cref{prop:exact-truncation} then yield \eqref{est:dirpatch}.
\end{proof}

For the numerical implementation in $d=2$, let $\mathcal D_{\vect{j}'}^R$ be the finite union of the resonators contained in $\mathsf P_{\vect{j}'}^R$. We replace the artificial Dirichlet condition by the free-space outgoing formulation: for $(\vect{j}',\ell')\in\mathcal I(\omega_0)$, let $U_{\vect{j}',\ell'}^{R,\mathrm{out}}$ solve
\begin{equation}
\begin{cases}
\left(\Delta+\dfrac{\omega_0^2}{v^2}\right)U_{\vect{j}',\ell'}^{R,\mathrm{out}}=0 &\text{in }\mathbb R^2\setminus\overline{\mathcal D_{\vect{j}'}^R},\\
U_{\vect{j}',\ell'}^{R,\mathrm{out}}=u_{\vect{j}',\ell'}&\text{on }\partial D_{\vect{j}'},\\
U_{\vect{j}',\ell'}^{R,\mathrm{out}}=0&\text{on the other components of } \partial\mathcal D_{\vect{j}'}^R,\\
U_{\vect{j}',\ell'}^{R,\mathrm{out}}&\text{satisfies the Sommerfeld radiation condition}.
\end{cases}
\label{eq:local-patch-radiation}
\end{equation}
Define $\mathcal C_{N,R}^{\mathrm{out}}(\omega_0)$ by \eqref{def:patch}, with $U_{\vect{j}',\ell'}^R$ replaced by $U_{\vect{j}',\ell'}^{R,\mathrm{out}}$. This problem is directly implemented using the free-space Green function, but its stability must be controlled uniformly as the cluster grows. 

Let $\mathcal T_{\vect{j},R}^{\mathrm{out}}(\omega_0)$ be the free-space outgoing DtN map on $\Gamma_{\vect{j}}^R$ defined with respect to
$\nu_\Gamma$, and let
\begin{equation*}
\mathcal V_{\vect{j},R} := \left\{ w\in H^1(\Omega_{\vect{j}}^R): w=0\text{ on every resonator boundary in }\mathsf P_{\vect{j}}^R \right\}.
\end{equation*}
We introduce the variational operator associated with \eqref{eq:local-patch-radiation},
\begin{equation*}
\mathcal B_{\vect{j},R}^{\mathrm{out}}:\mathcal V_{\vect{j},R}\longrightarrow\mathcal V_{\vect{j},R}^*,
\end{equation*}
defined, for $w,\phi\in\mathcal V_{\vect{j},R}$, by
\begin{equation*}
\langle\mathcal B_{\vect{j},R}^{\mathrm{out}}w,\phi\rangle_{\mathcal V_{\vect{j},R}^*,\mathcal V_{\vect{j},R}} := \int_{\Omega_{\vect{j}}^R} \Big( \nabla w\cdot\nabla\overline\phi -\frac{\omega_0^2}{v^2}w\overline\phi \Big)\,\mathrm dx - \left\langle \mathcal T_{\vect{j},R}^{\mathrm{out}}(\omega_0)w, \phi \right\rangle_{H^{-1/2}(\Gamma_{\vect{j}}^R),H^{1/2}(\Gamma_{\vect{j}}^R)}.
\end{equation*}
For a bounded operator $A$, we use the convention $\operatorname{Im}A:=(A-A^*)/(2\mathrm i)$. We again assume the uniform nonresonance condition for growing clusters, similarly to \eqref{eq:patch-stability}, 
\begin{equation} \label{eq:outgoing-patch-stability}
\sup_{\vect{j}\in\Lambda, R\geq1} \Big(\|\mathcal T_{\vect{j},R}^{\mathrm{out}}(\omega_0)\|_{\mathcal B(H^{1/2}(\Gamma_{\vect{j}}^R),H^{-1/2}(\Gamma_{\vect{j}}^R))} + \| (\mathcal B_{\vect{j},R}^{\mathrm{out}})^{-1}\|_{\mathcal B(\mathcal V_{\vect{j},R}^*,\mathcal V_{\vect{j},R})} \Big) < \infty.
\end{equation}

\begin{theorem}[Outgoing patch error] \label{thm:outgoing-patch} 
Suppose \eqref{gapassump2equiv} and \eqref{eq:outgoing-patch-stability} hold. Then there exist $C,\beta,\mu>0$ such that, for sufficiently large $R$ and integers $1\leq N\leq R$,
\begin{equation}\label{eq:outgoing-patch-error}
\|\mathcal C(\omega_0)-\mathcal C_{N,R}^{\mathrm{out}}(\omega_0)\|_{\mathcal B(\ell^2(\mathcal I(\omega_0)))} \leq C\left(e^{-\beta N}+e^{-\mu R}\right),
\end{equation}
and hence, $\|\operatorname{Im}\mathcal C_{N,R}^{\mathrm{out}}(\omega_0)\| \leq C\left(e^{-\beta N}+e^{-\mu R}\right)$.
\end{theorem}

\begin{proof}
For $(\vect{j}',\ell')\in\mathcal I(\omega_0)$, the restriction of the global field on $\Gamma_{\vect{j}'}^R$ has outgoing-boundary residual:
\begin{equation*}
\eta_{\vect{j}',\ell'}^R := \partial_{\nu_\Gamma} U_{\vect{j}',\ell'} -\mathcal T_{\vect{j}',R}^{\mathrm{out}}(\omega_0)U_{\vect{j}',\ell'}.
\end{equation*}
The two terms in the residual are bounded by \eqref{field:est} and the first bound in \eqref{eq:outgoing-patch-stability}. Hence,
\begin{equation*}
\|\eta_{\vect{j}',\ell'}^R\|_{H^{-1/2}(\Gamma_{\vect{j}'}^R)} \leq Ce^{-\mu_0R}.
\end{equation*}
The difference between $U_{\vect{j}',\ell'}$ and $U_{\vect{j}',\ell'}^{R,\mathrm{out}}$, restricted to the patch, vanishes on every resonator boundary and is the solution generated by this residual. The second bound in \eqref{eq:outgoing-patch-stability} therefore gives an $H^1(\Omega_{\vect{j}'}^R)$ error of order $e^{-\mu_0R}$. The normal-trace estimate and the same Schur-sum argument as in the proof of \cref{lem:dipatch}, together with the truncation tail in \cref{prop:exact-truncation}, yield \eqref{eq:outgoing-patch-error}. 
\end{proof}

\subsection{Complex-frequency stabilization}\label{subsec:complex-stabilization}

The real-frequency approximation in \Cref{subsec:real-frequency-patches} gives an exponential rate, but the uniform stability \eqref{eq:patch-stability} or \eqref{eq:outgoing-patch-stability} does not follow from the infinite-volume exterior gap alone and may deteriorate when a growing finite cluster develops a resonance near $\omega_0$. We now remove this additional real-frequency finite-cluster stability assumption by adding a vanishing complex absorption.

Let $z_0:=\omega_0^2$. By \eqref{gapassump2equiv}, we choose $r>0$ so that
\begin{equation}
\label{eq:analytic-capacitance-neighborhood}
    \overline{B(z_0,r)}\subset \left\{z\in\mathbb C:\frac{z}{v^2}\in \rho\bigl(-\Delta_{\mathrm{Dir}};\mathbb R^d\setminus\overline{\mathcal D}\bigr)\right\}.
\end{equation}
For $z\in B(z_0,r)$, define the auxiliary analytic family
\begin{equation} \label{eq:analytic-capacitance-family}
\mathcal C_{z_0}(z):=-\frac{1}{2\omega_0} M^{-1/2}G^*\mathcal T^zGM^{-1/2}.
\end{equation}
Then $\mathcal C_{z_0}(z_0)=\mathcal C(\omega_0)$, where $G$ and $M$ are given as in \eqref{def:GandM}. Choose $0<\gamma_0<r$ and let
\begin{equation} \label{eq:complex-shift}
z_\gamma:=z_0+\mathrm i\gamma, \qquad 0<\gamma\leq\gamma_0.
\end{equation}
The following observation shows that the Hermitian symmetrization reduces the regularization error from the first to the second order.

\begin{lemma} \label{lem:schwarz-reflection}
For $z,\overline z\in B(z_0,r)$, we have 
\begin{equation} \label{eq:schwarz-reflection}
\mathcal C_{z_0}(\overline z)=\mathcal C_{z_0}(z)^*,
\end{equation}
and 
\begin{equation} \label{eq:second-order-bias}
\left\|\mathcal C(\omega_0)-\frac12\left( \mathcal C_{z_0}(z_\gamma)+\mathcal C_{z_0}(z_\gamma)^* \right)\right\|_{\mathcal B(\ell^2(\mathcal I(\omega_0)))} = \mathcal{O} (\gamma^2). 
\end{equation}
\end{lemma}

\begin{proof}
Noting $(\mathcal T^z)^*=\mathcal T^{\overline z}$, we have \eqref{eq:schwarz-reflection} from \eqref{eq:analytic-capacitance-family}. The map in \eqref{eq:analytic-capacitance-family} is analytic in the operator norm on a neighborhood of $z_0$. Then, Taylor's formula and \eqref{eq:schwarz-reflection} readily give
\begin{equation*}
\frac12\left(\mathcal C_{z_0}(z_0+\mathrm i\gamma) +\mathcal C_{z_0}(z_0-\mathrm i\gamma)\right) -\mathcal C_{z_0}(z_0) = \mathcal O_{\mathcal B(\ell^2)}(\gamma^2). \qedhere
\end{equation*}
\end{proof}

For a source site $\vect{j}'\in\Lambda$ and $R\geq1$, define the source-centered finite cluster and its exterior by
\begin{equation} \label{eq:finite-cluster-domain}
\begin{aligned}
\Lambda_{\vect{j}'}^R &:=\left\{\vect{k}\in\Lambda: |\vect{k}-\vect{j}'|\leq R\right\},\\
\mathcal D_{\vect{j}'}^{R,\mathrm{fc}} &:=\bigcup_{\vect{k}\in\Lambda_{\vect{j}'}^R}D_{\vect{k}},
\qquad
\Omega_{\vect{j}'}^{R,\mathrm{fc}} :=\mathbb R^d\setminus \overline{\mathcal D_{\vect{j}'}^{R,\mathrm{fc}}}.
\end{aligned}
\end{equation}
For $(\vect{j}',\ell')\in\mathcal I(\omega_0)$, let $U_{\vect{j}',\ell'}^\gamma\in H^1(\mathbb R^d\setminus\overline{\mathcal D})$ be the unique regularized global field satisfying
\begin{equation} \label{eq:global-complex-field}
\begin{cases}
\Big(\Delta+\dfrac{z_\gamma}{v^2}\Big)
U_{\vect{j}',\ell'}^\gamma=0
&\text{in }\mathbb R^d\setminus\overline{\mathcal D},\\
U_{\vect{j}',\ell'}^\gamma=u_{\vect{j}',\ell'}
&\text{on }\partial D_{\vect{j}'},\\
U_{\vect{j}',\ell'}^\gamma=0
&\text{on }\partial D_{\vect{k}},\quad \vect{k}\neq\vect{j}'.
\end{cases}
\end{equation}
For later use, we set $U_{\vect{j}',\ell'}^0:=U_{\vect{j}',\ell'}$.
Equivalently, the entries of the auxiliary operator are
\begin{equation*}
\bigl(\mathcal C_{z_0}(z_\gamma)\bigr)_{(\vect{j},\ell),(\vect{j}',\ell')}
=-\frac{v_{\vect{j}}v_{\vect{j}'}}{2\omega_0}
\Big\langle \left.\partial_\nu U_{\vect{j}',\ell'}^\gamma\right|_+, u_{\vect{j},\ell} \Big\rangle_{H^{-1/2}(\partial D_{\vect{j}}),H^{1/2}(\partial D_{\vect{j}})}.
\end{equation*}
The corresponding finite-cluster field
$U_{\vect{j}',\ell'}^{R,\gamma}\in
H^1(\Omega_{\vect{j}'}^{R,\mathrm{fc}})$ is the unique solution of
\begin{equation}
\label{eq:finite-cluster-complex-field}
\begin{cases}
\Big(\Delta+\dfrac{z_\gamma}{v^2}\Big)
U_{\vect{j}',\ell'}^{R,\gamma}=0
&\text{in }\Omega_{\vect{j}'}^{R,\mathrm{fc}},\\
U_{\vect{j}',\ell'}^{R,\gamma}=u_{\vect{j}',\ell'}
&\text{on }\partial D_{\vect{j}'},\\
U_{\vect{j}',\ell'}^{R,\gamma}=0
&\text{on }\partial D_{\vect{k}},\quad
\vect{k}\in\Lambda_{\vect{j}'}^R\setminus\{\vect{j}'\}.
\end{cases}
\end{equation}
Here, the $H^1$ regularity selects the exponentially decaying solution associated with the branch of $\sqrt{z_\gamma}$ having a positive imaginary part.

For integers $1\leq N\leq R$, define the complex finite-cluster operator on $\ell^2(\mathcal I(\omega_0))$ by
\begin{equation}
\label{eq:complex-finite-cluster-operator}
\bigl(\mathcal C_{N,R,\gamma}^{\mathrm{fc}}\bigr)_{(\vect{j},\ell),(\vect{j}',\ell')} :=
\begin{cases}
-\dfrac{v_{\vect{j}}v_{\vect{j}'}}{2\omega_0}
\Big\langle \left.\partial_\nu U_{\vect{j}',\ell'}^{R,\gamma}\right|_+,
u_{\vect{j},\ell}\Big\rangle_{H^{-1/2}(\partial D_{\vect{j}}),H^{1/2}(\partial D_{\vect{j}})},
&|\vect{j}-\vect{j}'|\leq N,\\[2mm]
0,&|\vect{j}-\vect{j}'|>N.
\end{cases}
\end{equation}
Its Hermitian symmetrization is
\begin{equation} \label{eq:hermitian-stabilized-operator}
\widehat{\mathcal C}_{N,R,\gamma} :=\frac12\left( \mathcal C_{N,R,\gamma}^{\mathrm{fc}} +(\mathcal C_{N,R,\gamma}^{\mathrm{fc}})^* \right).
\end{equation}
We first give the stability result from the complex absorption.

\begin{lemma} \label{lem:uniform-complex-stability}
There exists $C>0$, independent of $\vect{j}'$, $R$, and $0<\gamma\leq\gamma_0$, such that
\begin{equation}
\label{eq:uniform-complex-resolvent}
\Big\|\Big( -\Delta_{\mathrm{Dir},\Omega_{\vect{j}'}^{R,\mathrm{fc}}} -\frac{z_\gamma}{v^2}\Big)^{-1}\Big\|_{H^{-1}(\Omega_{\vect{j}'}^{R,\mathrm{fc}}) \to H_0^1(\Omega_{\vect{j}'}^{R,\mathrm{fc}})} \leq\frac{C}{\gamma}.
\end{equation}
\end{lemma}

Note that the argument for \cref{thm:decay} also applies to $U_{\vect{j}',\ell'}^\gamma$, which gives the exponential decay uniformly in $\gamma$. The estimate follows directly from the spectral calculus, and we omit the proof. Moreover, since $\{z_\gamma/v^2: 0\leq\gamma\leq\gamma_0\}$ is a compact subset of the resolvent set of the exterior Dirichlet Laplacian, the Combes--Thomas and lifting arguments used in the proof of \cref{thm:decay} apply uniformly to $U_{\vect{j}',\ell'}^\gamma$. In particular, these fields decay exponentially away from their source resonators, with constants independent of $\gamma$.

\begin{lemma}
\label{lem:uniform-complex-localization}
There exist $C,\beta>0$, independent of $(\vect{j}',\ell')\in\mathcal I(\omega_0)$, $\vect{j}\in\Lambda$, and $0\leq\gamma\leq\gamma_0$, such that
\begin{equation}
\label{eq:uniform-complex-localization}
\Big\|\chi_{\vect{j}}U_{\vect{j}',\ell'}^\gamma\Big\|_{H^1(\mathbb R^d\setminus\overline{\mathcal D})} +\Big\|\partial_\nu U_{\vect{j}',\ell'}^\gamma\big|_+ \Big\|_{H^{-1/2}(\partial D_{\vect{j}})} \leq Ce^{-\beta|\vect{j}-\vect{j}'|}.
\end{equation}
\end{lemma}

\begin{proposition} \label{prop:complex-field-comparison}
There exist $C,\mu_0>0$ such that, for all $(\vect{j}',\ell')\in\mathcal I(\omega_0)$, $R\geq1$, and $0<\gamma\leq\gamma_0$,
\begin{equation}
\label{eq:complex-field-comparison}
\Big\|U_{\vect{j}',\ell'}^{R,\gamma}
-\widetilde U_{\vect{j}',\ell'}^\gamma\Big\|_{
H^1(\Omega_{\vect{j}'}^{R,\mathrm{fc}})}
\leq C\gamma^{-1}e^{-\mu_0 R}.
\end{equation}
Here $\widetilde U_{\vect{j}',\ell'}^\gamma$ is obtained from the global field $U_{\vect{j}',\ell'}^\gamma$ by zero extension to $D_{\vect{k}}$ with $\vect{k}\notin\Lambda_{\vect{j}'}^R$.
\end{proposition}

\begin{proof}
Noting $W_{\vect{j}',\ell'}^{R,\gamma} :=U_{\vect{j}',\ell'}^{R,\gamma} -\widetilde U_{\vect{j}',\ell'}^\gamma \in H_0^1(\Omega_{\vect{j}'}^{R,\mathrm{fc}})$, define the distribution $F_{\vect{j}',\ell'}^{R,\gamma}$ by 
\begin{equation} \label{eq:omitted-boundary-residual}
\begin{aligned}
\Big\langle F_{\vect{j}',\ell'}^{R,\gamma},\varphi\Big\rangle:={}&\int_{\Omega_{\vect{j}'}^{R,\mathrm{fc}}}\left(\nabla\widetilde U_{\vect{j}',\ell'}^\gamma\cdot\nabla\overline\varphi-\frac{z_\gamma}{v^2}\widetilde U_{\vect{j}',\ell'}^\gamma\overline\varphi\right)\,\mathrm dx,
\qquad \varphi\in H_0^1(\Omega_{\vect{j}'}^{R,\mathrm{fc}}).
\end{aligned}
\end{equation}
It is easy to see from the estimate \eqref{eq:uniform-complex-localization} and by integrating by parts that 
\begin{equation} \label{eq:residual-hminus1}
    \left\|F_{\vect{j}',\ell'}^{R,\gamma}\right\|_{ H^{-1}(\Omega_{\vect{j}'}^{R,\mathrm{fc}})} \leq C e^{-\mu_0 R}.
\end{equation}
Moreover, we have 
\begin{equation*}
\Big(-\Delta_{\mathrm{Dir},\Omega_{\vect{j}'}^{R,\mathrm{fc}}} -\frac{z_\gamma}{v^2}\Big)
W_{\vect{j}',\ell'}^{R,\gamma} = -F_{\vect{j}',\ell'}^{R,\gamma}.
\end{equation*}
Applying \cref{lem:uniform-complex-stability} and \eqref{eq:residual-hminus1} proves \eqref{eq:complex-field-comparison}.
\end{proof}

\begin{theorem}[Stabilized finite-cluster approximation] \label{thm:stabilized-finite-cluster}
Under the assumption \eqref{gapassump2equiv}, there exist $C,\beta,\mu>0$ such that, for all integers $1\leq N\leq R$ and $0<\gamma\leq\gamma_0$,
\begin{equation}
\label{eq:complex-operator-error}
\left\|\mathcal C_{z_0}(z_\gamma)
-\mathcal C_{N,R,\gamma}^{\mathrm{fc}}\right\|_{
\mathcal B(\ell^2(\mathcal I(\omega_0)))}
\leq C\left(e^{-\beta N}+\gamma^{-1}e^{-\mu R}\right), 
\end{equation}
and hence
\begin{equation} \label{eq:stabilized-main-error}
\big\|\mathcal C(\omega_0)-\widehat{\mathcal C}_{N,R,\gamma}\big\|_{\mathcal B(\ell^2(\mathcal I(\omega_0)))} \leq C\left(\gamma^2+e^{-\beta N}+\gamma^{-1}e^{-\mu R}\right).
\end{equation}
Here $\mathcal C_{N,R,\gamma}^{\mathrm{fc}}$ and its symmetrization $\widehat{\mathcal C}_{N,R,\gamma}$ are given in \eqref{eq:complex-finite-cluster-operator} and \eqref{eq:hermitian-stabilized-operator}, respectively. 
\end{theorem}

\begin{proof}
For retained row and column indices, the field estimate \eqref{eq:complex-field-comparison} and the uniform normal-trace estimate give an entrywise error bounded by $C\gamma^{-1}e^{-\mu_0R}$. The uniform complex-frequency version of \eqref{eq_cap_decay} gives an exponentially small tail for the omitted entries. Letting
$E_{N,R,\gamma}:=\mathcal C_{z_0}(z_\gamma)
-\mathcal C_{N,R,\gamma}^{\mathrm{fc}}$, the lattice growth yields
\begin{equation*}
\sup_{\vect{j}'}\#\{\vect{j}:|\vect{j}-\vect{j}'|\leq N\} +\sup_{\vect{j}}\#\{\vect{j}':|\vect{j}-\vect{j}'|\leq N\} \leq C(1+N)^d,
\end{equation*}
and, since $N\leq R$, 
\begin{equation*}
\sup_{(\vect{j}',\ell')}\sum_{(\vect{j},\ell)}
\left|(E_{N,R,\gamma})_{(\vect{j},\ell),(\vect{j}',\ell')}\right| +\sup_{(\vect{j},\ell)}\sum_{(\vect{j}',\ell')}
 \left|(E_{N,R,\gamma})_{(\vect{j},\ell),(\vect{j}',\ell')}\right| \leq C\left(e^{-\beta N}+\gamma^{-1}e^{-\mu R}\right),
\end{equation*}
for $0<\mu<\mu_0$. Again, Schur's test proves \eqref{eq:complex-operator-error}. The estimate \eqref{eq:stabilized-main-error} then follows from \cref{lem:schwarz-reflection}. The proof is complete. 
\end{proof}

\begin{corollary}[Exponential convergence of the regularized patch approximation] \label{cor:optimized-complex-shift}
Under the assumptions of \cref{thm:stabilized-finite-cluster}, 
for all sufficiently large $R$ and integers $1\leq N\leq R$, choose
\begin{equation} \label{eq:optimized-gamma}
\gamma_R:=e^{-\mu R/3}.
\end{equation}
Then, for some $C > 0$, 
\begin{equation} \label{eq:optimized-stabilized-error}
\big\|\mathcal C(\omega_0) -\widehat{\mathcal C}_{N,R,\gamma_R}\big\| \leq C\left(e^{-\beta N}+e^{-2\mu R/3}\right).
\end{equation}
\end{corollary}

\begin{remark} \label{rem:role-complex-shift}
We emphasize that the complex shift introduced in this section serves as a regularization for its stable computation. 
It provides a robust alternative when the growing real-frequency patch problems fail to be uniformly stable or become poorly conditioned.
\end{remark}

\subsection{Error analysis and complexity estimates} \label{subsec:spectral-complexity}
We first derive the spectral approximation from the operator approximation \eqref{eq:stabilized-main-error}. Since both $\mathcal C(\omega_0)$ and $\widehat{\mathcal C}_{N,R,\gamma}$ are bounded self-adjoint operators, we have 
\begin{equation}\label{eq:spectral-hausdorff}
    d_{\mathrm H}\big(\sigma(\mathcal C(\omega_0)),\sigma(\widehat{\mathcal C}_{N,R,\gamma})\big) \leq C \left(\gamma^2+e^{-\beta N}+\gamma^{-1}e^{-\mu R}\right),
\end{equation}
where $d_{\mathrm H}$ denotes Hausdorff distance. In the compact case, both operators are finite-dimensional Hermitian matrices, and then Weyl's  inequality gives the same bound for their ordered eigenvalues, counted with multiplicity. 

For a straight line defect, we choose clusters of the same shape centered at each source resonator, obtained from one another by translations along the defect. Then $\widehat{\mathcal C}_{N,R,\gamma}$ constructed above is a block Laurent operator of interaction range $N$. We denote its Floquet symbol by $\widehat{\mathcal C}_{N,R,\gamma}^\alpha$, from the Floquet transform. 
The direct-integral decompositions and \eqref{eq:stabilized-main-error} give the analogous estimate in the supremum over $\alpha$:
\begin{equation}\label{eq:fiber-stabilized-error}
    \sup_{\alpha\in\widetilde Y^*} \big\| \mathcal C^\alpha(\omega_0) - \widehat{\mathcal C}_{N,R,\gamma}^\alpha \big\| \leq C\left(\gamma^2+e^{-\beta N}+\gamma^{-1}e^{-\mu R}\right).
\end{equation}

We next combine these estimates with the first-order spectral asymptotics. In the compact-defect setting, fix an eigenvalue $\lambda$ of $\mathcal C(\omega_0)$ of multiplicity $m$, and let $\omega_{\delta,1},\dots,\omega_{\delta,m}$ be the associated defect eigenfrequencies from \cref{cor:compact-asymptotics}. Choose an open interval $I_\lambda$ whose closure meets $\sigma(\mathcal C(\omega_0))$ only at $\lambda$. Whenever the operator error in \eqref{eq:stabilized-main-error} is sufficiently small, Weyl's inequality shows that $\widehat{\mathcal C}_{N,R,\gamma}$ has exactly $m$ eigenvalues in $I_\lambda$, counted with multiplicity; we denote them by $\widehat\lambda_{1,N,R,\gamma} ,\dots,\widehat\lambda_{m,N,R,\gamma} $. For a line defect, let $\widehat\lambda_{k,N,R,\gamma}^\alpha$ denote the $k$th ordered eigenvalue of $\widehat{\mathcal C}_{N,R,\gamma}^\alpha$. Weyl's inequality together with \cref{cor:compact-asymptotics,cor:line-asymptotics} then yields
\begin{equation}\label{eq:end-to-end-compact}
\max_{1\leq k\leq m}\left|\omega_{\delta,k}-\left(\omega_0+\delta\widehat\lambda_{k,N,R,\gamma} \right)\right| \leq C\left(\delta+ \gamma^2+  e^{-\beta N}+ \gamma^{-1}e^{-\mu R}\right) \delta
\end{equation}
for the compact cluster whenever the approximate cluster above is defined, and, uniformly for integers $1\leq N\leq R$ and $0<\gamma\leq\gamma_0$,
\begin{equation}\label{eq:end-to-end-line}
\sup_{\alpha\in\widetilde Y^*}\max_{1\leq k\leq m_\star}\left|\omega_k^\alpha(\delta)-\left(\omega_0+\delta\widehat\lambda_{k,N,R,\gamma}^\alpha\right)\right| \leq C\left(\delta + \gamma^2 + e^{-\beta N} + \gamma^{-1}e^{-\mu R}\right) \delta
\end{equation}
for a straight line defect. The choice of $\gamma=\sqrt\delta$ gives the following result, by \eqref{eq:end-to-end-compact} and \eqref{eq:end-to-end-line}. It shows that the $\mathcal O(\delta^2)$ remainder in the first-order high-contrast expansions of the defect eigenfrequencies in \cref{cor:compact-asymptotics,cor:line-asymptotics} is preserved when both the interaction truncation radius and the patch radius are of order $|\log\delta|$.

\begin{corollary} \label{cor:delta-log-patch}
In the compact-defect setting, fix an eigenvalue $\lambda$ of $\mathcal C(\omega_0)$ of multiplicity $m$ and use the cluster notation introduced above. For all sufficiently small $\delta>0$, set $\gamma=\sqrt\delta$ and choose integers $N$ and $R$ satisfying
\begin{equation}\label{eq:NR-delta-choice}
    N\geq\frac{1}{\beta}|\log\delta|,
    \qquad
    R\geq\max\left\{N,\frac{3}{2\mu}|\log\delta|\right\}.
\end{equation}
Then, for the compact eigenvalue cluster associated with $\lambda$,
\begin{equation}\label{eq:delta2-compact}
    \max_{1\leq k\leq m}\left|\omega_{\delta,k}-\left(\omega_0+\delta\widehat\lambda_{k,N,R,\sqrt\delta} \right)\right| = \mathcal O(\delta^2),
\end{equation}
and for a straight line defect,
\begin{equation}\label{eq:delta2-line}
    \sup_{\alpha\in\widetilde Y^*}\max_{1\leq k\leq m_\star}\left|\omega_k^\alpha(\delta)-\left(\omega_0+\delta\widehat\lambda_{k,N,R,\sqrt\delta}^\alpha\right)\right| = \mathcal O(\delta^2).
\end{equation}
\end{corollary}

We finally quantify the computational complexity of a finite compression.
Define the set of resonant sites by
\begin{equation*}
\mathcal J(\omega_0):=\{\vect{j}\in\Lambda:m_{\vect{j}}>0\}.
\end{equation*}
Choose a finite computational subset $\mathcal J_{N_{\mathrm{act}}}\subset\mathcal J(\omega_0)$ with $|\mathcal J_{N_{\mathrm{act}}}|=N_{\mathrm{act}}$, and define the
corresponding resonant-mode index set by
\begin{equation*}
    \mathcal I_{N_{\mathrm{act}}}:=\{(\vect{j},\ell)\in\mathcal I(\omega_0):\vect{j}\in\mathcal J_{N_{\mathrm{act}}}\}.
\end{equation*}
For the family of computational subsets under consideration, assume that there exist $C>0$ and $d_{\mathrm{act}}\leq d$, independent of $N_{\mathrm{act}}$ and $s$, such that
\begin{equation}\label{eq:active-growth}
    \sup_{\vect{j}\in\Lambda} \#\left\{\vect{j}'\in\mathcal J_{N_{\mathrm{act}}}:
|\vect{j}-\vect{j}'|\leq s \right\} \leq C(1+s)^{d_{\mathrm{act}}}, \qquad s\geq0.
\end{equation}
Here $d_{\mathrm{act}}$ can be viewed as the intrinsic dimension of the active defect set, and this condition always holds with $d_{\mathrm{act}}=d$ because $\mathcal J_{N_{\mathrm{act}}}\subset\Lambda$ and $\Lambda$ has $d$-dimensional lattice growth. 
A smaller exponent may be used when the active sets satisfy the corresponding uniform lower-dimensional growth bound.

For a matrix $A$ indexed by $\mathcal I(\omega_0)$, write $A[\mathcal I_{N_{\mathrm{act}}}]$ for its principal submatrix indexed by $\mathcal I_{N_{\mathrm{act}}}$. Since both $\mathcal C_N^{\mathrm{tr}}(\omega_0)$ and $\widehat{\mathcal C}_{N,R,\gamma}$ have interaction range $N$, either principal submatrix contains at most
\begin{equation}\label{eq:nonzero-entry-count}
    \mathcal{O}\big(m_{\max}^2N_{\mathrm{act}}(1+N)^{d_{\mathrm{act}}}\big)
\end{equation}
nonzero scalar entries.  
For a given high contrast $\delta$, choose $N$ and $R$ proportional to $|\log\delta|$ so that \eqref{eq:NR-delta-choice} holds, and set $\gamma=\sqrt\delta$. Then,
\begin{equation}\label{eq:stabilized-nnz}
\#\Big(\widehat{\mathcal C}_{N,R,\sqrt\delta}[\mathcal I_{N_{\mathrm{act}}}]\Big) =\mathcal O\left(N_{\mathrm{act}}|\log\delta|^{d_{\mathrm{act}}}\right).
\end{equation}
Here and below, $\#(A)$ denotes the number of nonzero entries of a finite matrix $A$. 
For a finite segment of a straight or bent line satisfying the uniform one-dimensional growth bound \eqref{eq:active-growth}, one has $d_{\mathrm{act}}= 1$ and therefore
\begin{equation}\label{eq:line-nnz}
    \#\Big(\widehat{\mathcal C}_{N,R,\sqrt\delta}[\mathcal I_{N_{\mathrm{act}}}]\Big) =\mathcal O\left(N_{\mathrm{act}}|\log\delta|\right),
\end{equation}
instead of the $\mathcal O(N_{\mathrm{act}}^2)$ entries of a dense matrix. It follows that the storage and matrix--vector costs of the finite effective matrix $\widehat{\mathcal C}_{N,R,\sqrt\delta}[\mathcal I_{N_{\mathrm{act}}}]$ are linear in $N_{\mathrm{act}}$, up to a polylogarithmic factor.

While the matrix sparsity is controlled by the interaction radius $N$, each local PDE is posed on a patch containing $\mathcal O(R^d)=\mathcal O(|\log\delta|^d)$ resonators; its discretization-dependent solution cost is not included in the above estimate.

\section{Numerical results}\label{sec:numerical}

\subsection{Construction of the frequency-dependent capacitance operator}

We set $d=2$ and let $D$ be a disk of radius $r$. Its positive Neumann eigenvalues are $k^2$, where $k>0$ satisfies $J'_{\mathfrak m}(kr)=0$ for some multipole order $\mathfrak m$. The multipole order determines the type of mode (monopole, dipole, quadrupole, etc.) through the angular factors $e^{\pm\mathrm i\mathfrak m\theta}$. The radially symmetric modes $\mathfrak m=0$ have multiplicity one, $m=1$, whereas the higher-order modes $\mathfrak m\geq1$ have multiplicity two, $m=2$; see \cite{rohlederInequalitiesNeumannDirichlet2025}. 

The root index enumerates the positive radial roots and does not affect the multiplicity. In addition, the zero Neumann eigenvalue has the constant monopole mode on every connected resonator. This zero mode underlies the subwavelength regime and is present for every resonator shape, not only disks. At all frequencies, every resonator remains part of the exterior boundary-value problem. However, the effective capacitance operator is indexed only by the active modes in $\mathcal I(\omega_0)$; in the subwavelength regime, every connected resonator contributes its constant mode, whereas at a nonzero reference frequency, only resonators that are resonant at $\omega_0$ contribute active indices.

\begin{table}[!h]
\centering
\begin{tabular}{|c|c|c|c|c|c|}
\hline
Roots & $J_0'(x)$ & $J_1'(x)$ & $J_2'(x)$ & $J_3'(x)$ & $J_4'(x)$ \\ \hline
1 & 3.8317 & 1.8412 & 3.0542 & 4.2012 & 5.3175 \\ \hline
2 & 7.0156 & 5.3314 & 6.7061 & 8.0152 & 9.2824 \\ \hline
3 & 10.1735 & 8.5363 & 9.9695 & 11.3459 & 12.6819 \\ \hline
4 & 13.3237 & 11.7060 & 13.1704 & 14.5858 & 15.9641 \\ \hline
\end{tabular}
\caption{The first few positive roots of $J'_{\mathfrak m}$ for $0\leq\mathfrak m\leq4$.}
\label{table_bessel}
\end{table} 

We choose the resonator parameters so that the defect resonators are active at the reference interior Neumann frequency $\omega_0$, while the unperturbed resonators are nonresonant, as in \eqref{gapassump1}. The projected bulk spectrum is sampled below to check the location of the associated defect frequencies relative to the bulk bands numerically. If a finite configuration contains $N_{\mathrm{act}}$ active resonators, each of multiplicity $m$, then its capacitance matrix has size $mN_{\mathrm{act}}\times mN_{\mathrm{act}}$. For an infinite configuration, under the exterior gap condition \eqref{gapassump2equiv}, the capacitance operator acts on $\ell^2(\mathcal J(\omega_0);\mathbb C^m)$ and has exponentially decaying $m\times m$ blocks.

We consider the first two non-subwavelength regimes: the dipole and quadrupole resonances. These correspond to $\mathfrak m=1$ and $\mathfrak m=2$, respectively, and therefore have eigenvalue multiplicity $m=2$. In the line-waveguide experiments, there is one defect resonator per longitudinal period. For a line waveguide, $\mathcal C^\alpha(\omega_0)$ is consequently a $2\times2$ Floquet symbol acting on the two active modes in the reference strip (see \Cref{fig:patch}) and incorporating all longitudinal interactions through its $\alpha$-dependence.

\begin{figure}[!h]
    \centering
\begin{tikzpicture}[
    scale=0.8,
    transform shape,
    dot/.style={
        circle,
        draw,
        minimum size=10mm,
        text width=9mm,   %
        align=center,
        inner sep=0pt,
        font=\scriptsize
    }
]

\foreach \x in {-5,-4,4,5,6} {
    \foreach \y in {-3,...,3} {
        \ifnum\y=0
            \node[dot,
                draw=red!14,
                fill=red!16
            ] at (1.25*\x,1.25*\y) {};
        \else
            \node[dot,
                draw=Cyan!19,
                fill=Cyan!13
            ] at (1.25*\x,1.25*\y) {};
        \fi
    }
}

\foreach \x in {-5,...,6} {
    \foreach \y in {-4, 4} {
        \node[
            dot,
            draw=Cyan!19,
            fill=Cyan!13
        ] at (1.25*\x,1.25*\y) {};
    }
}

\foreach \x in {-3,...,3} {
    \foreach \y in {-3,...,3} {

        \pgfmathtruncatemacro{\outer}{(\x==-3) || (\x==3) || (\y==-3) || (\y==3)}

        \ifnum\outer=1
            \def\borderstyle{dashed}
        \else
            \def\borderstyle{}
        \fi

        \pgfmathtruncatemacro{\col}{\x+3}
        \pgfmathtruncatemacro{\row}{\y+3}

        \ifnum\y=0
            \node[dot,fill=red!43,\borderstyle] (c\col\row) at (1.25*\x,1.25*\y)
                {${\x},{\y}$};
        \else
            \node[dot,fill=Cerulean!30,\borderstyle] (c\col\row) at (1.25*\x,1.25*\y)
                {${\x},{\y}$};
        \fi
    }
}

\draw[
    decorate,
    decoration={brace,amplitude=6pt}
]
(-4.35,-4.25) -- (-4.35,-0.70)
node[midway,left=9pt,align=center] {cladding \\ $N_\mathrm{clad} = 3$};

\draw[
    decorate,
    decoration={brace,amplitude=6pt}
]
(-4.35,-0.5) -- (-4.35,0.50)
node[midway,left=9pt] {defect};

\draw[
    decorate,
    decoration={brace,amplitude=6pt}
]
(-4.35,0.70) -- (-4.35,4.25)
node[midway,left=9pt,align=center] {cladding \\ $N_\mathrm{clad} = 3$};

\draw[->,thick,shorten >=4pt]

    (5,3.75) node[right, align=center] {outer fringe layer \\ $N_\mathrm{fringe} = 1$}

    -- (c66);

\draw[->,thick,shorten >=4pt]

    (5,0) node[right, align=center] {$2N + 1$ defect \\ resonators, \\ $N = 3$}

    -- (c63);

\foreach \y [count=\k from 0] in {-3,...,3} {
    \pgfmathtruncatemacro{\xx}{3-\k}
    \ifnum\xx=0
        \node[dot,fill=red!43]  (copy\k) at (1.75*6,1.25*\xx + 0.)
        {$0,{\xx}$};
    \else
        \node[dot,fill=Cerulean!30]  (copy\k) at (1.75*6,1.25*\xx + 0.)
        {$0,{\xx}$};
    \fi
}

\node[dot,
    draw=Cyan!19,
    fill=Cyan!13
] at (1.75*6,1.25*4 + 0.) {};

\node[dot,
    draw=Cyan!19,
    fill=Cyan!13
] at (1.75*6,-1.25*4 + 0.) {};

\node[below=9mm of copy6, align=center] {reference strip};
\node[below=9mm of c30, align=center] {2D crystal};

\draw[dashed]
  ($(copy0.east)!0.5!(copy0.east)+(1mm,13mm)$) --
  ($(copy6.east)!0.5!(copy6.east)+(1mm,-13mm)$);

\draw[dashed]
  ($(copy0.west)!0.5!(copy0.west)+(-1mm,13mm)$) --
  ($(copy6.west)!0.5!(copy6.west)+(-1mm,-13mm)$);

\node[rotate=90] at ($(copy4.west)+(-4mm,13mm)$) {\textsc{QBC}};
\node[rotate=90] at ($(copy4.east)+(4mm,13mm)$) {\textsc{QBC}};

    \draw[->] (6.4,-4.7) -- (7.4,-4.7) node[right] {\footnotesize $e_1$};
    \draw[->] (6.4,-4.7) -- (6.4,-3.7) node[above] {\footnotesize $e_2$};

\end{tikzpicture}
\caption{Selection of a patch from a crystal with a horizontal line defect (in the $e_1$ direction). The parameter $N_{\mathrm{fringe}}$ is the thickness of the boundary fringe; these resonators are included in the boundary-integral calculation but omitted from the capacitance matrix. On the right is the reference strip: the infinite strip centered at $j_1=0$, with quasi-periodic boundary conditions (QBC) on its sides. As for the rectangular patch, a vertical finite-cladding approximation can be chosen in the reference strip. Since it has no edge defect resonators, no fringe rows are needed: it contains one defect resonator and $N_{\mathrm{clad}}$ cladding resonators on each side.}\label{fig:patch}
\end{figure}  

\setlength{\fboxsep}{0pt}

\newcommand{\blk}{%
  \colorbox{Cerulean!55}{\phantom{\rule{4mm}{4mm}}}%
}

\newcommand{\redblk}{%
  \colorbox{red!57}{\phantom{\rule{4mm}{4mm}}}%
}

With the numbering convention in \Cref{fig:patch}, the numerical boundary-integral system contains both the inactive cladding resonators (blue) and the active defect resonators (red). All of them are needed to compute the exterior fields. The effective capacitance operator $\mathcal C$, however, is already compressed to the active index set $\mathcal I(\omega_0)$; the blue resonators therefore do not contribute rows or columns to $\mathcal C$. They are auxiliary unknowns in the unreduced boundary system, not zero blocks of the effective operator.

For a finite patch, the same compression is performed after solving the boundary-integral system. The fringe and cladding resonators remain in that system but are omitted from the effective matrix. In the patch shown in \Cref{fig:patch}, the five retained active resonators consequently produce a $5\times5$ block capacitance matrix, as illustrated in \Cref{fig:blockform}.

\begin{figure}[!h]
\[
\mathsf A_N \sim
\begin{pmatrix}
\outerddots & \outervdots & \outervdots & \outervdots & \outervdots & \outervdots & \outervdots & \outervdots & \outeriddots \\

\outercdots & \blk & \blk & \blk & \blk & \blk & \blk & \blk & \outercdots \\

\outercdots & \blk & \innerddots & \innervdots & \innervdots & \innervdots & \inneriddots & \blk & \outercdots \\

\outercdots & \blk & \innercdots & \redblk & \redblk & \redblk & \innercdots & \blk & \outercdots \\

\outercdots & \blk & \innercdots & \redblk & \redblk & \redblk & \innercdots & \blk & \outercdots \\

\outercdots & \blk & \innercdots & \redblk & \redblk & \redblk & \innercdots & \blk & \outercdots \\

\outercdots & \blk & \inneriddots & \innervdots & \innervdots & \innervdots & \innerddots & \blk & \outercdots \\

\outercdots & \blk & \blk & \blk & \blk & \blk & \blk & \blk & \outercdots \\

\outeriddots & \outervdots & \outervdots & \outervdots & \outervdots & \outervdots & \outervdots & \outervdots & \outerddots
\end{pmatrix} \longrightarrow \mathcal C_N^{\mathrm{out}} \sim \begin{pmatrix}
\redblk & \redblk & \redblk & \redblk & \redblk \\
\redblk & \redblk & \redblk & \redblk & \redblk \\
\redblk & \redblk & \redblk & \redblk & \redblk \\
\redblk & \redblk & \redblk & \redblk & \redblk \\
\redblk & \redblk & \redblk & \redblk & \redblk
\end{pmatrix}.
\]
    \centering
    \caption{Schematic unreduced boundary-integral system and the outgoing active-mode capacitance matrix obtained from the patch in \Cref{fig:patch}. Each block of the reduced capacitance matrix on the right has size $m\times m$; the block sizes in the unreduced boundary-integral system on the left depend on the boundary discretization.}
    \label{fig:blockform}
\end{figure}

To approximate the quasi-periodic capacitance matrix, we retain $N_{\mathrm{clad}}$ cladding resonators on each side of the central defect resonator in the reference strip, with quasi-periodicity along the defect and an open exterior in the transverse direction. We use \cref{def:freq-cap-QP-line} and quasi-periodic layer-potential techniques to compute this finite-cladding approximation of $\mathcal C^\alpha$; see \cite{cbms}. An inverse fast Fourier transform (IFFT) of the sampled matrices then approximates the real-space interaction blocks for each horizontal shift $j_1-j_1^\prime$ in \eqref{inversesymbolLN}. Both the finite cladding and the finite quasi-momentum sampling contribute approximation errors.

Finite-cluster approximations of the fields $U_{\vect{j},\ell}$ for $(\vect{j},\ell)\in \mathcal I(\omega_0)$ are computed using layer-potential techniques. The frequency-dependent capacitance matrix is assembled accordingly using \eqref{eq:def-freq-cap-operator-entry}.

For the source-centered window construction illustrated in \Cref{fig:subpatchfig}, $N$ specifies the extent of the active sites along each arm, while $W$ is the half-width of the square interaction window. The local boundary-integral problem includes an additional fringe of $N_{\mathrm{fringe}}$ layers, giving a square of half-width $W+N_{\mathrm{fringe}}$. Window configurations that differ only by translation along a straight arm can be reused.

For the line-waveguide computations below, we use the following parameters, with the radii and interior speeds noted in the corresponding subsections:

\begin{itemize}
    \item number of quadrature points per resonator $N_\mathrm{pts} = 64$;
    \item quasimomentum discretization $N_\alpha = 31$;
    \item number of projected quasimomenta $N_{\alpha_2} = 13$;
    \item contrast parameter $\delta = 0.001$;
    \item speed in the background medium $v = 1$;
    \item reference resonator in the plots $(j_1, j_2) = (0,0).$
\end{itemize}

\begin{figure}
\begin{tikzpicture}[
    circle style/.style={
        draw=black,
        line width=0.5pt,
        circle,
        minimum size=0.4cm,
        inner sep=0pt
    }
]

\begin{scope}[shift={(-3.5,0)}]

\filldraw[fill=Cerulean!30, draw=black, thick] (0.2,0.2) rectangle (5.3, 5.3);

\foreach \x in {0,...,9} {
    \foreach \y in {0,...,9} {
        \ifnum\x=3 \ifnum\y>2
            \node[draw=red!70, fill=red!43,
            line width=0.5pt,
            circle,
            minimum size=0.4cm,
            inner sep=0pt] at
                (\x*0.5 + 0.5, \y*0.5 + 0.5) {};
            \else 
            
            \ifnum\y=3 \ifnum\x>2
                \node[draw=fill=red!70, fill=red!43,
                line width=0.5pt,
                circle,
                minimum size=0.4cm,
                inner sep=0pt] at
                    (\x*0.5 + 0.5, \y*0.5 + 0.5) {};
                \else \node[circle style] at (\x*0.5 + 0.5, \y*0.5 + 0.5) {}; \fi
                \else \node[circle style] at (\x*0.5 + 0.5, \y*0.5 + 0.5) {}; \fi
            \fi
            \else
            \ifnum\y=3 \ifnum\x>2
            \node[draw=red!70, fill=red!43,
            line width=0.5pt,
            circle,
            minimum size=0.4cm,
            inner sep=0pt] at
                (\x*0.5 + 0.5, \y*0.5 + 0.5) {};
            \else \node[circle style] at (\x*0.5 + 0.5, \y*0.5 + 0.5) {}; \fi
            \else \node[circle style] at (\x*0.5 + 0.5, \y*0.5 + 0.5) {}; \fi
        \fi
    }
}
\end{scope}

\draw[->,thick,shorten >=4pt]

    (2.3,2.6) node[right, align=center] {} -- (3.3, 2.6);

\begin{scope}[shift={(3.5,0)}]

\filldraw[fill=Cerulean!30, draw=black, thick] (0.2,0.2) rectangle (5.3, 5.3);

\foreach \x in {0,...,9} {
    \foreach \y in {0,...,9} {
        \ifnum\x=3 \ifnum\y>2
            \node[draw=red!70, fill=red!43,
            line width=0.5pt,
            circle,
            minimum size=0.4cm,
            inner sep=0pt, text=black,
            font=\scriptsize] at
                (\x*0.5 + 0.5, \y*0.5 + 0.5) {{$\the\numexpr3-\y\relax$}};
            \else 
            
            \ifnum\y=3 \ifnum\x>2
                \node[draw=red!70, fill=red!43,
                line width=0.5pt,
                circle,
                minimum size=0.4cm,
                inner sep=0pt, text=black,
            font=\scriptsize] at
                    (\x*0.5 + 0.5, \y*0.5 + 0.5) {{$\the\numexpr\x-3)\relax$}};
                \else \node[circle style] at (\x*0.5 + 0.5, \y*0.5 + 0.5) {}; \fi
                \else \node[circle style] at (\x*0.5 + 0.5, \y*0.5 + 0.5) {}; \fi
            \fi
            \else
            \ifnum\y=3 \ifnum\x>2
            \node[draw=red!70, fill=red!43,
            line width=0.5pt,
            circle,
            minimum size=0.4cm,
            inner sep=0pt, text=black,
            font=\scriptsize] at
                (\x*0.5 + 0.5, \y*0.5 + 0.5) {{$\the\numexpr\x-3\relax$}};
            \else \node[circle style] at (\x*0.5 + 0.5, \y*0.5 + 0.5) {}; \fi
            \else \node[circle style] at (\x*0.5 + 0.5, \y*0.5 + 0.5) {}; \fi
        \fi
    }
}
\filldraw[dashed, fill opacity=0., fill=blue, draw=white, thick] (1.25,3.25) rectangle (2.75,4.75);
\filldraw[dashed, fill opacity=0., fill=purple, draw=white, thick] (3.25,1.25) rectangle (4.75,2.75);
\filldraw[dashed, fill opacity=0., fill=red, draw=white, thick] (1.25,1.25) rectangle (2.75,2.75);

\end{scope}

\foreach \g in {0,...,4} {
    \filldraw[fill=Cerulean!30, draw=black, thick] (\g * 2.0 + 0.2 - 2.25,0.2 - 2.5) rectangle (\g * 2.0 + 1.8 - 2.25, 1.3 - 2);
    \foreach \x in {0,...,2} {
        \foreach \y in {0,...,2} {
        \ifnum\numexpr\x*\y\relax=1
            \node[draw=red!70, fill=red!43,
            line width=0.5pt,
            circle,
            minimum size=0.4cm,
            inner sep=0pt] at
                (\g*2.0 + \x*0.5 - 1.75, -2 + \y*0.5) {};
        \else
            \node[circle style] at
                (\g*2.0 + \x*0.5 - 1.75, -2 + \y*0.5) {};
            \ifnum\g<3 \ifnum\x=1 \ifnum\y=2 
                \node[draw=red,
                    line width=0.5pt,
                    circle,
                    minimum size=0.4cm,
                    inner sep=0pt] at
                        (\g*2.0 + \x*0.5 - 1.75, -2 + \y*0.5) {};
            \fi \fi \fi

            \ifnum\g>1 \ifnum\x=2 \ifnum\y=1
                \node[draw=red,
                    line width=0.5pt,
                    circle,
                    minimum size=0.4cm,
                    inner sep=0pt] at
                        (\g*2.0 + \x*0.5 - 1.75, -2 + \y*0.5) {};
            \fi \fi \fi
        \fi
        }
    }
}
\node[draw=red,
                    line width=0.5pt,
                    circle,
                    minimum size=0.4cm,
                    inner sep=0pt] at
                        (0*2.0 + 1*0.5 - 1.75, -2 + 0*0.5) {};
\node[text=black,
    font=\scriptsize] at (0*2.0 + 1*0.5 - 1.75, -2 + 1*0.5) {$-2$};
\node[draw=red,
                    line width=0.5pt,
                    circle,
                    minimum size=0.4cm,
                    inner sep=0pt] at
                        (1*2.0 + 1*0.5 - 1.75, -2 + 0*0.5) {};

\node[draw=red,
                    line width=0.5pt,
                    circle,
                    minimum size=0.4cm,
                    inner sep=0pt] at
                        (1*2.0 + 2*0.5 - 1.75, -2 + 0*0.5) {};
\node[text=black,
    font=\scriptsize] at (1*2.0 + 1*0.5 - 1.75, -2 + 1*0.5) {$-1$};
\node[text=black,
    font=\scriptsize] at (2*2.0 + 1*0.5 - 1.75, -2 + 1*0.5) {$0$};
\node[draw=red,
                    line width=0.5pt,
                    circle,
                    minimum size=0.4cm,
                    inner sep=0pt] at
                        (3*2.0 + 0*0.5 - 1.75, -2 + 1*0.5) {};

\node[draw=red,
                    line width=0.5pt,
                    circle,
                    minimum size=0.4cm,
                    inner sep=0pt] at
                        (3*2.0 + 0*0.5 - 1.75, -2 + 2*0.5) {};
\node[text=black,
    font=\scriptsize] at (3*2.0 + 1*0.5 - 1.75, -2 + 1*0.5) {$1$};
\node[draw=red,
                    line width=0.5pt,
                    circle,
                    minimum size=0.4cm,
                    inner sep=0pt] at
                        (4*2.0 + 0*0.5 - 1.75, -2 + 1*0.5) {};
\node[text=black,
    font=\scriptsize] at (4*2.0 + 1*0.5 - 1.75, -2 + 1*0.5) {$2$};
\end{tikzpicture}
\caption{Schematic of a bent waveguide with defect arms extending $N=6$ lattice spacings from the corner, together with source-centered square windows of half-width $W=1$ and no fringe. There is only one source per subpatch: the colored and numbered red resonator in the middle, so the uncolored defect resonators with a red border in the subpatches will not contribute as sources in the subpatch calculations. There are 5 unique subpatches in this configuration. The two arms of the L-bend can be named using positive and negative integers, respectively. The centers of the unique subpatches are numbered corresponding to the resonator in the larger patch. Because the neighborhood in the subpatch is the same as for subpatches $2$ and $-2$ for resonators with index $|j_1^\prime - j_1 - j_2^\prime + j_2| \geq 2$, in this example, we reuse the subpatches $-2$ and $2$ for the vertical and horizontal arms, respectively.}
\label{fig:subpatchfig}
\end{figure}

\subsection{Line waveguide}

First, we sample the projected bulk spectrum of the unperturbed crystal. This helps identify candidate band gaps and choose suitable material parameters and radii of the defects that place the defect frequencies in these band gaps. The crystal is periodic in two directions $e_1$ and $e_2$; however, adding the line defect will break the $e_2$-periodicity. For each $\alpha_1\in[0,\pi]$, we sample $N_{\alpha_2}$ values of $\alpha_2\in[0,\pi]$ and identify candidate bulk frequencies using the smallest singular values of the periodic transmission matrix, with a transverse minimum-to-median ratio below $0.25$. This provides a numerical check of the location of the defect frequencies relative to the sampled projected bulk spectrum. We plot the bands from $0$ to $\pi$ in \Cref{fig:all_bandsv} and \Cref{fig:all_bands} since the plots are symmetric around $\pi$ due to the inversion symmetry of the lattice.

Using the approach in \cite{ammariSubwavelengthGuidedModes2021}, we approximate the reference defect bands by locating local minima, as a function of frequency, of the smallest singular value of a multipole discretization of the operator $\mathcal M(\omega,\delta;\alpha_1)$ in \eqref{eq:MdensityLN}. The bands are colored differently to highlight the fact that the initial guess was based on the dipole and quadrupole resonances in \Cref{table_bessel}. Red corresponds to the dipole $\mathfrak m = 1$, and purple to the quadrupole $\mathfrak m = 2$ initial guess.

When the radii differ, following \cite{cbms,ammariSubwavelengthGuidedModes2021,erik2019}, we use a fictitious-source method. We create a replacement circular resonator of the unperturbed crystal inclusion radius $r$ but with modified boundary sources such that the interior and exterior fields of the replacement resonator and the defect resonator of radius $\widetilde r$ agree in the common interior and common exterior regions, respectively. When they are the same, we set ${r = \widetilde r}$ and use different material parameters, as explained below.

The capacitance-based defect bands are obtained by plotting the real parts of the frequency predictions in \eqref{asymp:line} for each quasi-momentum using the quasi-periodic capacitance matrix.

In the implementation, the capacitance calculations use the exterior wavenumber $k=(\omega_0+\mathrm i\varepsilon)/v$ with $\varepsilon=10^{-3}$, without Hermitian symmetrization. The reference $\mathcal M$ calculations use an imaginary frequency shift of $10^{-5}$, while the projected bulk scan has a default shift of $10^{-4}$. The band plots use half-zone quasi-momentum samples, whereas the IFFT uses a full-period grid with the duplicate endpoint omitted.

Next, we compare the spatial decay of the patch capacitance blocks with that of the blocks obtained by the inverse Floquet--Bloch transform of $\mathcal C^\alpha$. Since each block is a $2\times 2$ matrix, we look at the Frobenius norms. In the case of one participating defect resonator per period, the interacting resonator index always has the form $(j_1, 0)$ as the defects lie on a line. That leaves only two parameters $j_1$ and $j_1^\prime$ for the capacitance operator entries. Note that the periodicity of {the infinite line-defect problem} ensures that the interaction between $(j_1 + 1,0)$ and $(j_1^\prime + 1, 0)$ is the same as the interaction between $(j_1,0)$ and $(j_1^\prime, 0)$; that is, the entries depend only on the {signed separation} $j_1 - j_1^\prime$ between the two interacting defect resonators. Thus, $\mathcal C$ is a block Laurent operator and, {under the assumptions of} \cref{thm:decay}, is exponentially approximable by banded operators.

In \Cref{fig:patch_vs_floquetv} and \Cref{fig:patch_vs_floquet}, a series of {$[2(N+N_{\mathrm{fringe}})+1]\times(2N_{\mathrm{clad}}+1)$} patches is created with the middle resonators being the line defect. {Here, $2N+1$ is the number of retained defect resonators, and the same finite cluster is used for all source columns.} The patches {use a free-space outgoing formulation} and the free-space boundary integral operators, as explained in the previous section. The figures show the dependence on $N$ and {$N_{\mathrm{clad}}$ at fixed $N_{\mathrm{fringe}}=1$} and compare the results against the inverse Floquet--Bloch transform of the quasi-periodic capacitance matrix $\mathcal C^\alpha$. {The comparison tests the agreement and spatial decay of the block norms obtained from these two numerical approximations.}

\subsubsection{Waveguiding by detuning the material parameters}
\label{linemat}

Let the radii and the speeds be such that

\begin{itemize}
\item the crystal resonator radius $r = \widetilde r = 0.350$;
    \item the speed in the unperturbed resonators $v_b = 1$;
    \item the speed in the defected resonators $\widetilde v_b = 0.350 / 0.455$.
\end{itemize}

{The computed defect bands lie within gaps of the sampled projected bulk spectrum.} The defect bands using \cite{ammariSubwavelengthGuidedModes2021} for the dipole and quadrupole are shown as red and purple circles in \Cref{fig:all_bandsv}. The results obtained using the capacitance matrix are overlaid as black points. The two methods show excellent agreement, and, as predicted, we see two bands for both.

The patches are all created with {one fringe layer at each longitudinal end}. The reference IFFT was computed using a quasi-periodic frequency-dependent capacitance matrix $\mathcal C^\alpha$ with {6 cladding resonators on each side of the defect}. We vary {$N$ and $N_{\mathrm{clad}}$}. In both the dipole and the quadrupole case, {the block norms show close visual agreement down to magnitudes of approximately $10^{-10}$}. The circles {for the largest patches ($N=10$) are drawn behind those for the smallest patches ($N=2$)}. In this way, we can see the last few norms for each $N$. Some markers are hidden by overlapping datasets. For the largest dipole patch, some block norms decrease to about $10^{-15}$, while the IFFT tail levels off at a larger value; see \Cref{fig:dip_decayv}. The quadrupole patch norms in \Cref{fig:quad_decayv} decay more slowly and reach about $10^{-13}$ for the largest patch. These small values describe coupling magnitudes and do not provide a direct approximation-error estimate. In both cases, the block norms exhibit rapid spatial decay before reaching their numerical tails.

\begin{figure}[htbp]
    \centering
    \begin{minipage}[b]{0.5\textwidth}
        \centering
        \includegraphics[width=\linewidth]{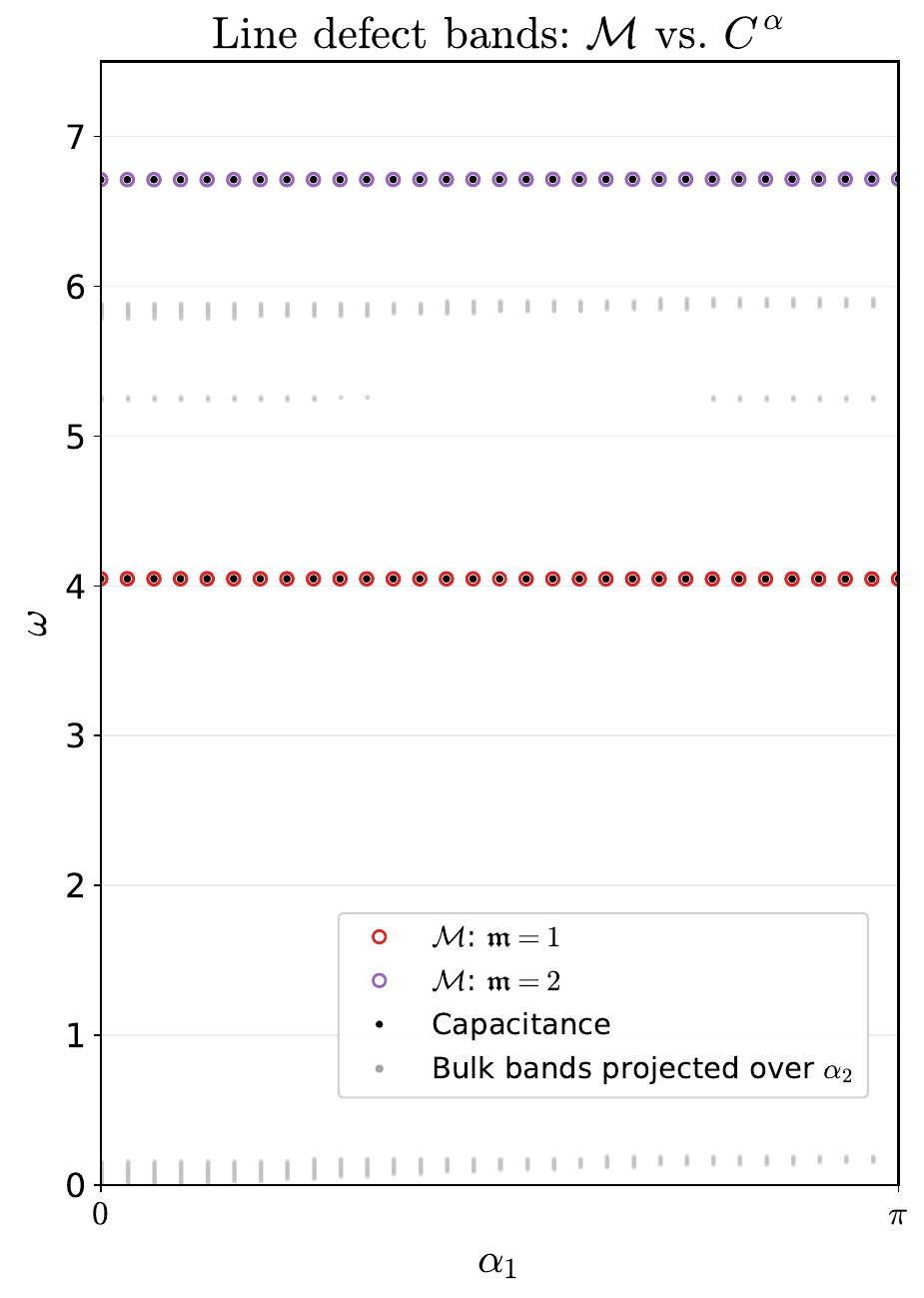}
        \label{fig:bandsv}
    \end{minipage}
    \hfill
    \raisebox{0.5cm}{%
    \begin{minipage}[b]{0.48\textwidth}
        \centering
        \includegraphics[width=\linewidth]{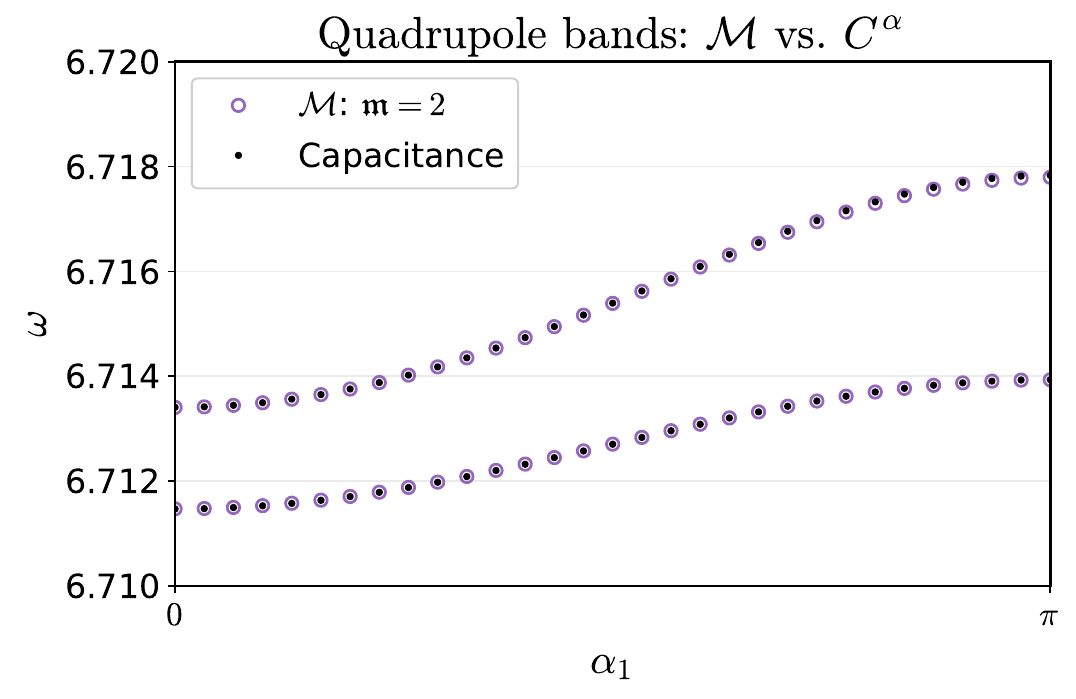}
        \label{fig:quadv}

        \vspace{-0.3cm}

        \includegraphics[width=\linewidth]{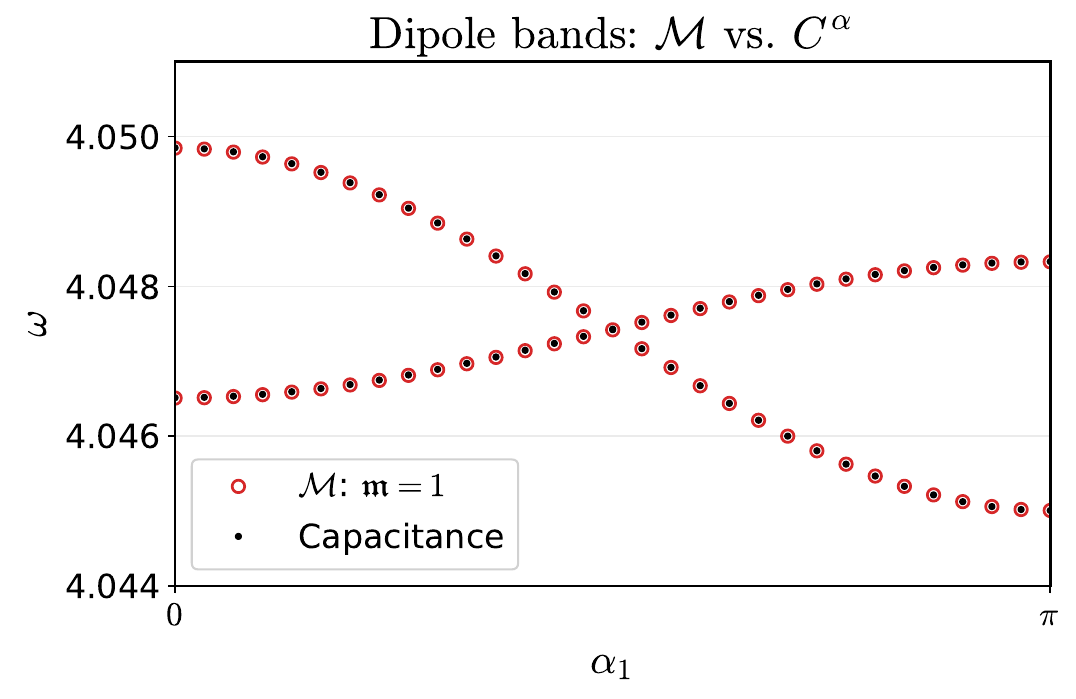}
        \label{fig:dipv}
    \end{minipage}%
    }
    \caption{{Sampled projected bulk spectrum} and non-subwavelength defect bands for defected resonators of different material parameters with $\alpha_1$ from $0$ to $\pi$, with the dipole and quadrupole bands zoomed in.}
    \label{fig:all_bandsv}
\end{figure}

\begin{figure}[!h]
        \centering
        \begin{subfigure}{0.8\textwidth}
        \includegraphics[width=\linewidth]{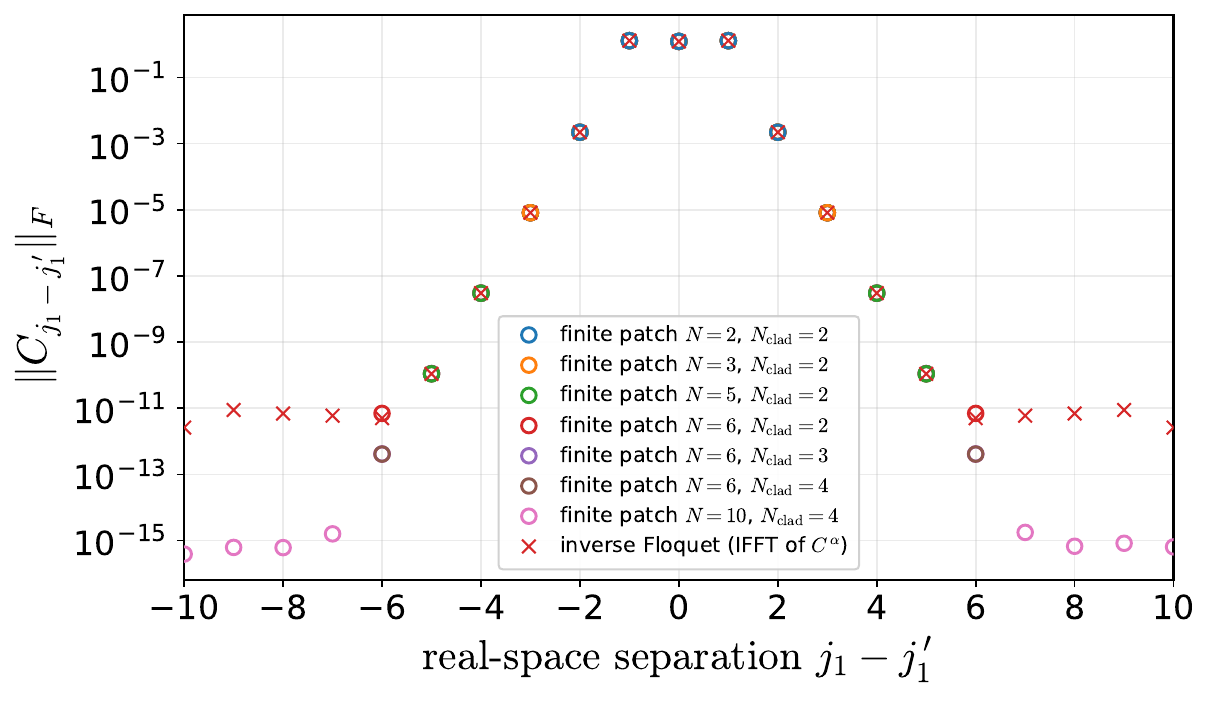}
        \caption{Dipole patch vs. inverse Floquet--Bloch transform of $\mathcal C^\alpha$}
        \label{fig:dip_decayv}
        \end{subfigure}
        \begin{subfigure}{0.8\textwidth}
        \includegraphics[width=\linewidth]{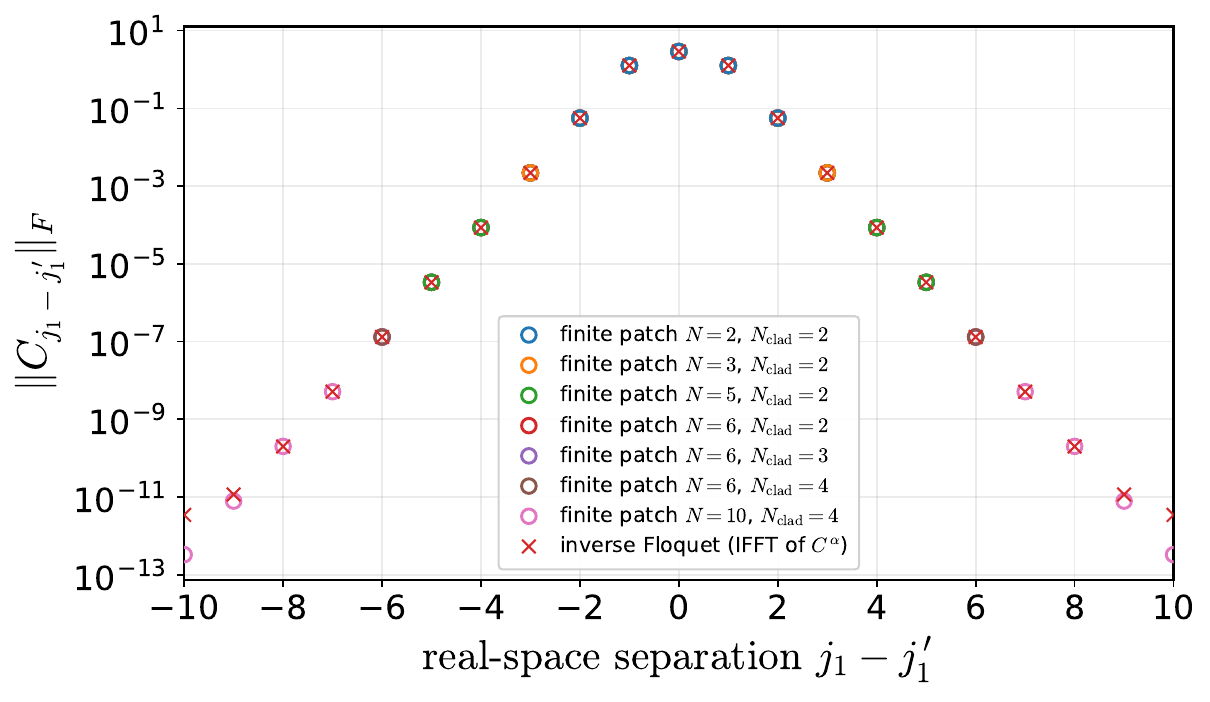}
        \caption{Quadrupole patch vs. inverse Floquet--Bloch transform of $\mathcal C^\alpha$}
        \label{fig:quad_decayv}
        \end{subfigure}
    \caption{Capacitance operator {block-norm decay for finite patches with $2N+1$ retained defect resonators} vs. the inverse Floquet--Bloch transform of $\mathcal C^\alpha$ with defect resonators with different material parameters. {The IFFT reference uses $N_{\mathrm{clad}}=6$ resonators on each side of the defect; the finite-patch values are shown in the legends. All patches use $N_{\mathrm{fringe}}=1$.}}
    \label{fig:patch_vs_floquetv}
\end{figure}

\subsubsection{Waveguiding by detuning the size of resonators}
\label{linerad}

Let the radii and interior speeds be such that

\begin{itemize}
    \item the crystal resonator radius $r = 0.350$;
    \item the defect resonator radius $\widetilde r = 0.455$;
    \item the speed in all resonators (defected and unperturbed) $v_b = \widetilde v_b = 1$.
\end{itemize}

This results in the same interior Neumann frequencies as in \Cref{linemat}.

Again, the dipole and the quadrupole are represented, respectively, as red and purple circles in \Cref{fig:all_bands}, and the capacitance-based results are overlaid as black points. This time, the bands span larger frequency intervals, and the quadrupole bands show a deviation at the lower frequencies, which are the furthest away from the base frequency $\omega_0 = 6.7126.$

For the patch decay in \Cref{fig:patch_vs_floquet}, we have similar results as before. Away from the numerical tails, the block norms obtained from the IFFT of $\mathcal C^\alpha$ and the finite patches show close agreement. The patch size and cladding thickness affect the smallest computed block norms. In \Cref{fig:quad_decay}, the spatial decay becomes slower once the block norms fall below approximately $10^{-7}$, in both the IFFT and patch computations. 

\begin{figure}[!h]
    \centering
    \begin{minipage}[b]{0.5\textwidth}
        \centering
        \includegraphics[width=\linewidth]{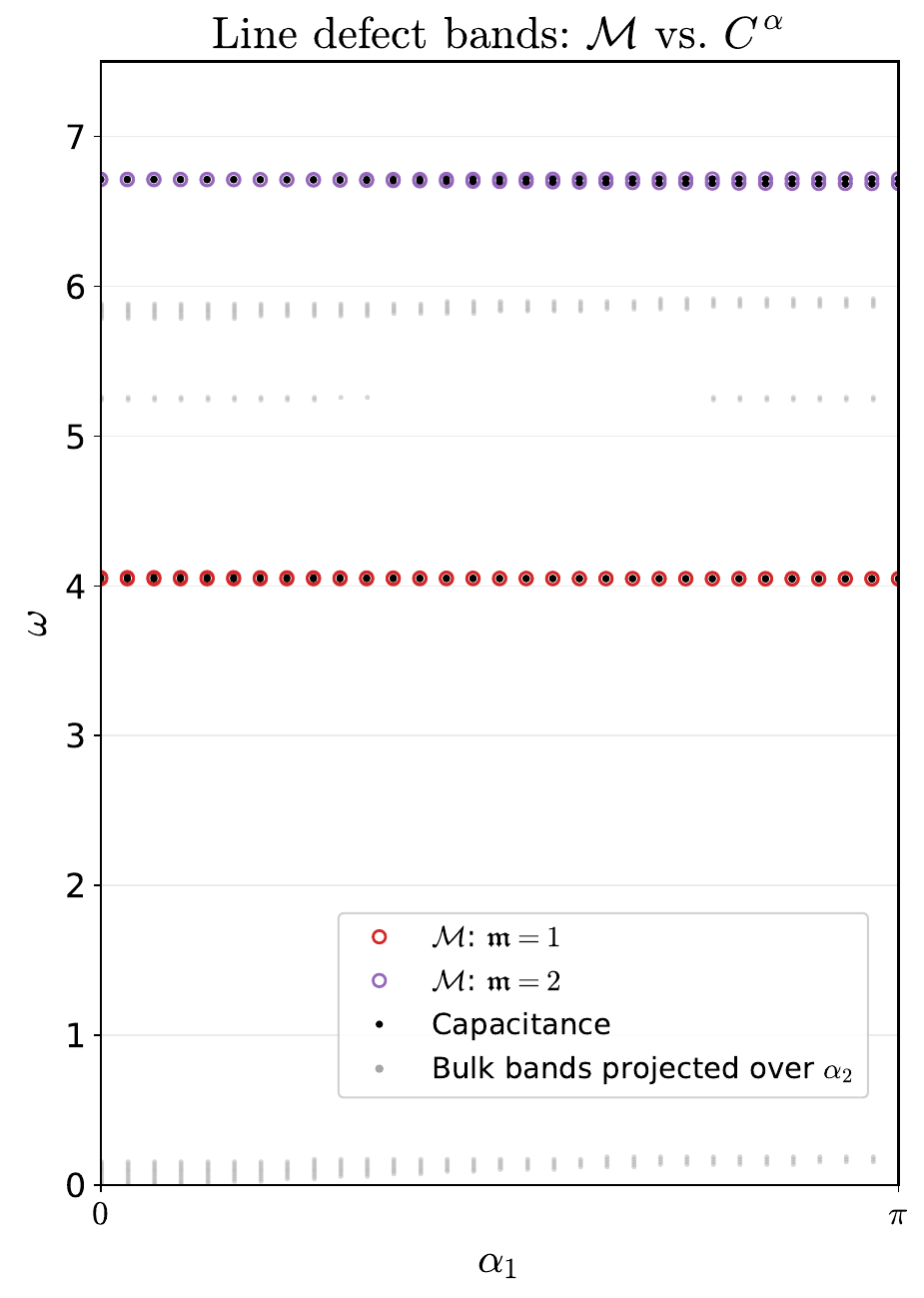}
        \label{fig:bands}
    \end{minipage}
    \hfill
    \raisebox{0.3cm}{%
    \begin{minipage}[b]{0.48\textwidth}
        \centering
        \includegraphics[width=\linewidth]{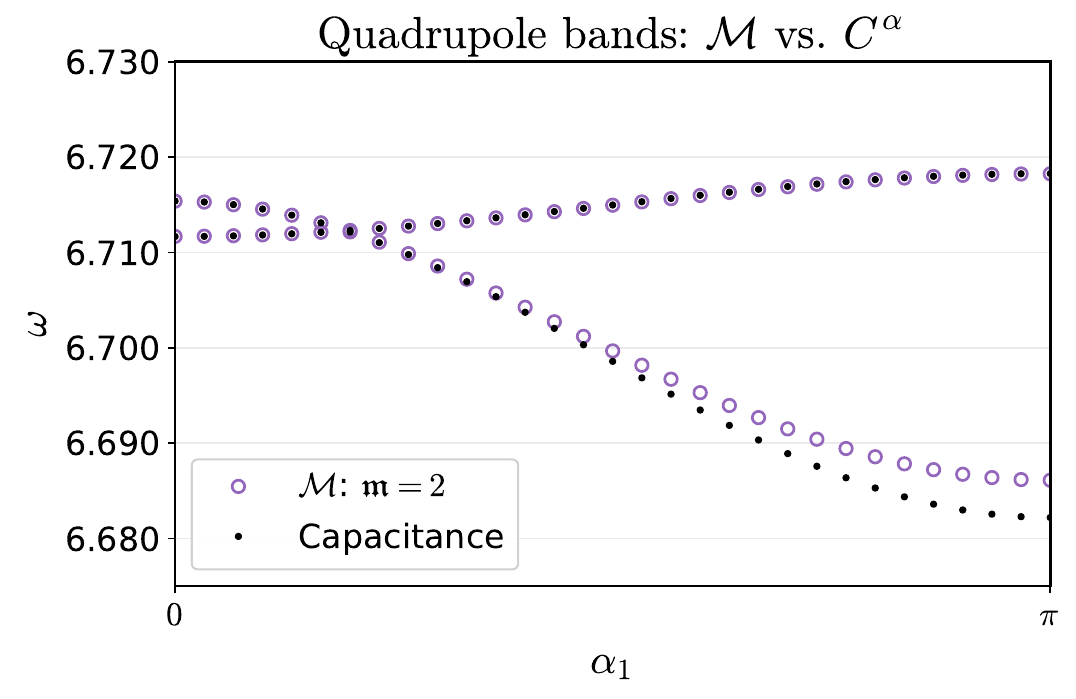}
        \label{fig:quad}

        \includegraphics[width=\linewidth]{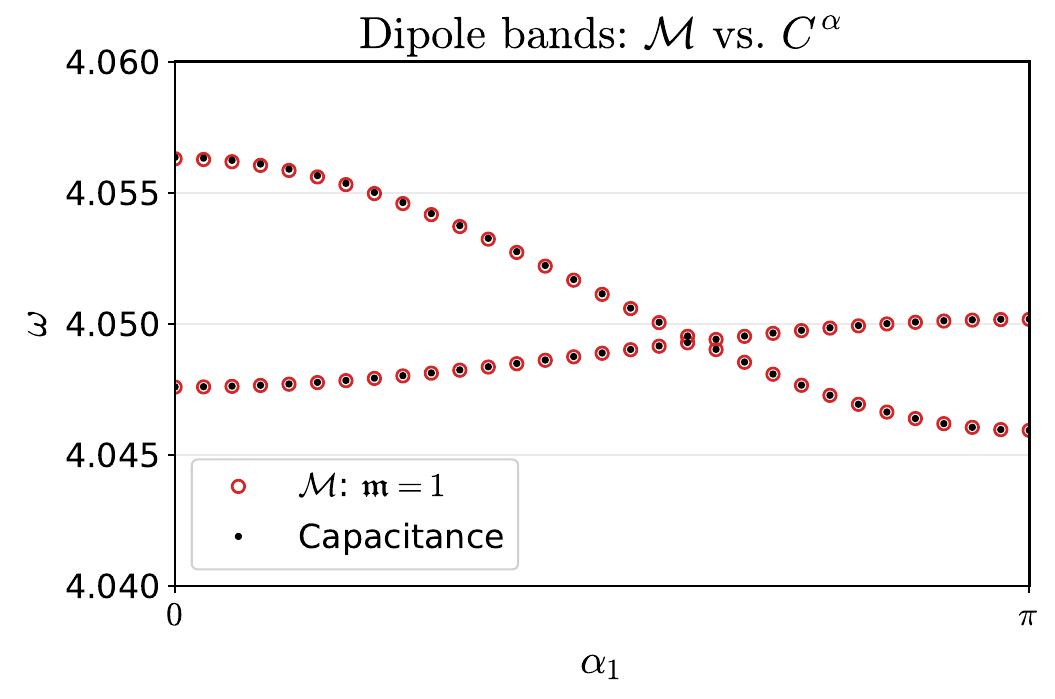}
        \label{fig:dip}
    \end{minipage}%
    }
    \caption{{Sampled projected bulk spectrum} and non-subwavelength defect bands for defected resonators of a different radius with $\alpha_1$ from $0$ to $\pi$, with the dipole and quadrupole bands zoomed in.}
    \label{fig:all_bands}
\end{figure}

\subsection{Bent waveguide}

\begin{figure}[!h]
      \centering
        \begin{subfigure}[t]{0.8\linewidth}
        \includegraphics[width=\linewidth]{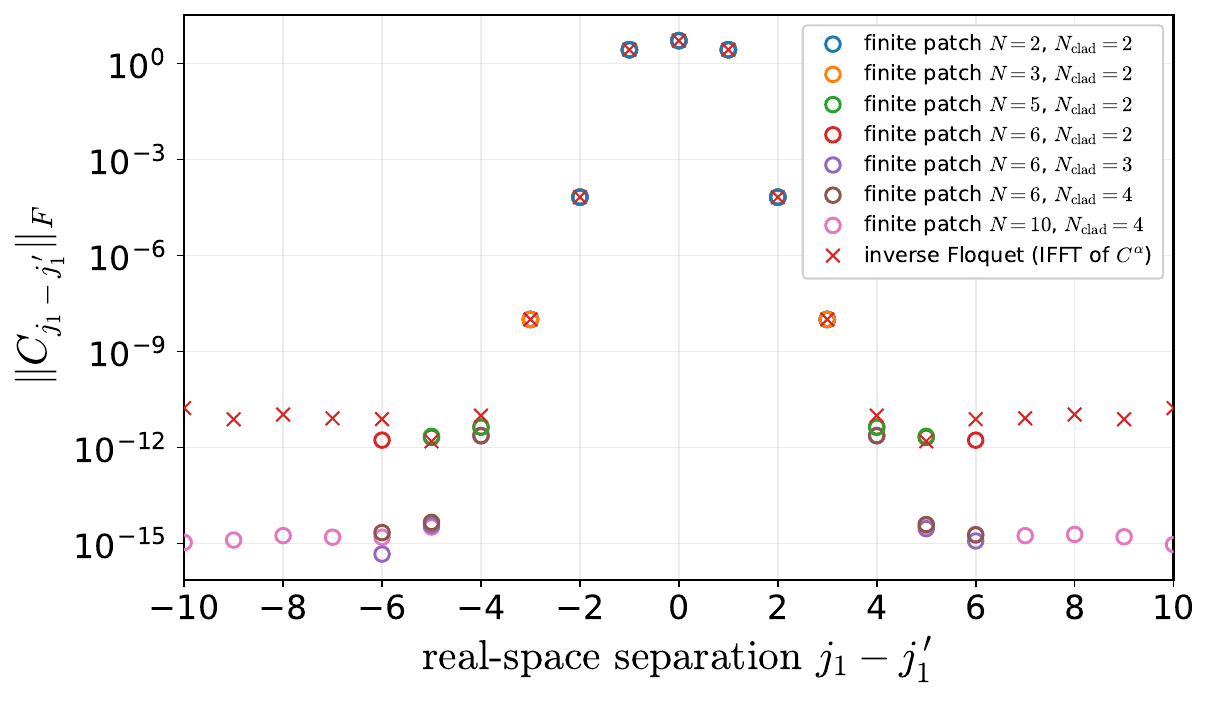}
        \caption{Dipole patch vs. inverse Floquet--Bloch transform of $\mathcal C^\alpha$}
        \label{fig:dip_decay}
     \end{subfigure}
     \begin{subfigure}[t]{0.8\linewidth}
        \includegraphics[width=\linewidth]{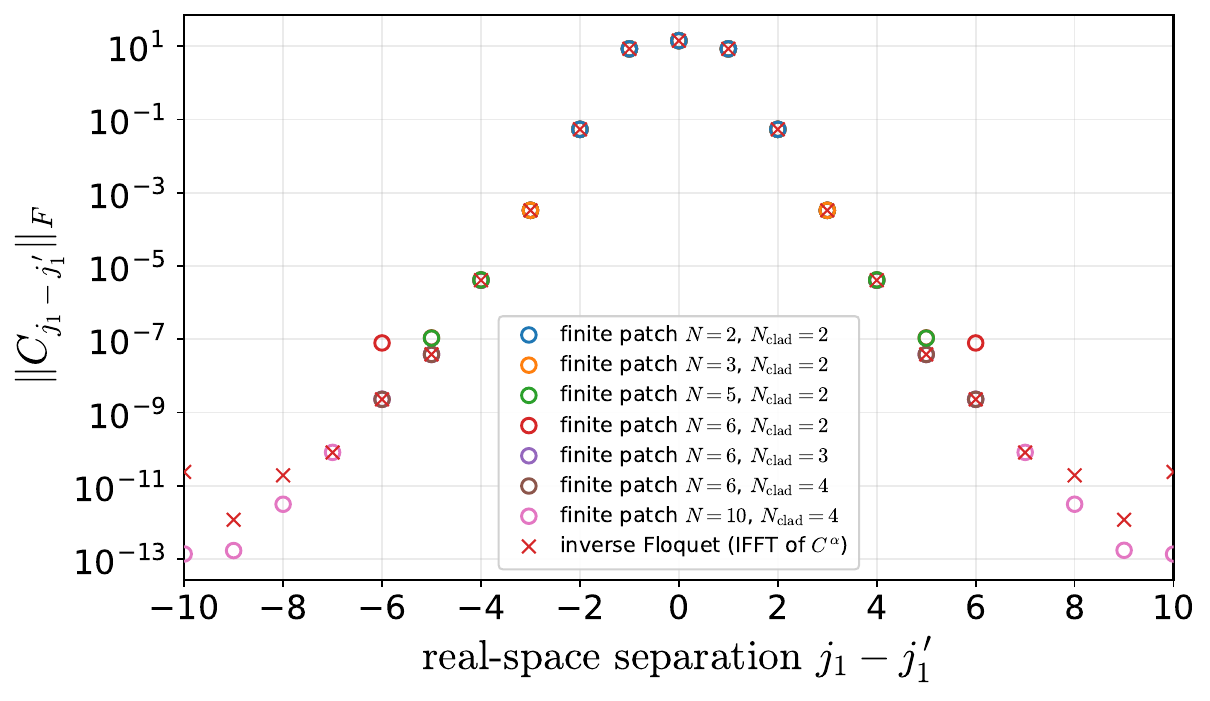}
        \caption{Quadrupole patch vs. inverse Floquet--Bloch transform of $\mathcal C^\alpha$}
        \label{fig:quad_decay}
        \end{subfigure}
    \caption{Capacitance operator {block-norm decay for finite patches with $2N+1$ retained defect resonators} vs. the inverse Floquet--Bloch transform of $\mathcal C^\alpha$. {The IFFT reference uses $N_{\mathrm{clad}}=6$ resonators on each side of the defect; the finite-patch values are shown in the legends. All patches use $N_{\mathrm{fringe}}=1$.}}
    \label{fig:patch_vs_floquet}
\end{figure}

For the common-cluster decay computations in \Cref{fig:patch_bentv,fig:patch_bent}, we take square patches of size $(2N + 1)\times(2N + 1)$ resonators with a free-space outgoing formulation and one fringe layer to avoid possible edge effects in the capacitance calculation, except for $N = 7$, which has a varying number of fringe layers removed to show how they affect the computation. In every patch, we have $N$ defects for the vertical part of the ``L'' bend, $1$ for the center of the patch (corner of the ``L'' and the reference resonator), and $N$ for the horizontal part. The capacitance matrices are then $(2(N - N_\text{fringe}) + 1)\times(2(N - N_\text{fringe}) + 1)$ block matrices.

To index the resonators, we may pick the reference resonator (here $(j_1, j_2) =(0,0)$) and take the positive $x_1$-distances to be the positive part of the index set and the positive $x_2$-axis part to be the negative part of the index set. This is achievable by setting the distance index to $j_1^\prime - j_1 - j_2^\prime + j_2$. In this way, when we interact with the resonators directly above $(0,0)$, we get negative indices, and when we interact with the resonators to the right of $(0,0)$, we get positive indices, which allows for the same type of symmetric decay plot as for the line waveguide. A useful depiction of this can be seen in \Cref{fig:subpatchfig}.

For the source-centered window construction below, $N$ instead denotes the extent of the retained active sites along each arm, as in \Cref{fig:subpatchfig}. Let $\widecheck{\mathcal I}_N(\omega_0)$ denote the active modes at these $2N+1$ retained sites. The local boundary-integral problem includes the fringe and has half-width $W+N_{\mathrm{fringe}}$; the fringe does not reduce the global set of source columns.

For each $(\vect{j},\ell)\in\widecheck{\mathcal I}_N(\omega_0)$, let $\Phi_{\vect{j},\ell}^{W}$ denote the sampled local column field: the outgoing field in the exterior of its source-centered cluster, the interior mode $u_{\vect{j},\ell}$ in the source resonator, and zero in the other resonators of that cluster. Each contribution is set to zero outside its evaluation window. If $\mathbf v^{(n)}$ is the right singular vector associated with the smallest singular value at a selected local minimum of $F(z)$ below, the approximate field is reconstructed as
\begin{equation}
\label{formula:eigenmode}
u_n^{N,W}(x) := \sum_{(\vect{j},\ell)\in\widecheck{\mathcal I}_N(\omega_0)} \mathbf v^{(n)}_{\vect{j},\ell}\Phi_{\vect{j},\ell}^{W}(x).
\end{equation}
Since the interior wave speed is constant on the active resonators in each experiment, the factor $M_N^{-1/2}$, where $M_N$ is the restriction of $M$ to these modes, affects only the overall normalization. The exterior column fields are computed at $k=\omega_0/v$ using the layer-potential representation in Appendix~\ref{appendix:proofs}.
These approximate fields can be seen in \Cref{fig:bent_patch_fieldsv} and \Cref{fig:bent_patch_fields}.

\subsubsection{Patch method using Colbrook's rectangular truncation}
For large patches, direct computations become computationally expensive; additionally, finite-section eigenvalues may suffer from spectral pollution.
The computations can be significantly sped up {by restricting} the capacitance computation for each source resonator to its neighborhood, a subpatch of radius $W$ as done in \cite{ammari2026resolvent_con}. This is depicted in \Cref{fig:subpatchfig}. We can then {assemble a banded capacitance matrix for the retained sites using these source-centered subpatches}.
The field assembly is performed additively per subpatch, and the subpatch fields are weighted according to the right singular vector in \eqref{formula:eigenmode}.
{We use rectangular truncation following} \cite{Colbrook2023}. For the computations below, $W$ block rows are added to each end of the square section, giving a rectangular matrix $\mathcal C_{\mathrm{rect}}$ of size $m[2(N+W)+1]\times m(2N+1)$, where $m=2$ for the dipole and quadrupole modes.
To {identify candidate approximate modes}, we look at the local minima of the function

\begin{equation}
    F(z) = \sigma_\mathrm{min}\left({\mathcal C_{\mathrm{rect}}} - zE\right), \quad E  = \begin{bmatrix} \boldsymbol{0} \\ I \\ \boldsymbol{0}\end{bmatrix},
\end{equation}
where $E$ is the identity matrix padded with zero rows to match the size of ${\mathcal C_{\mathrm{rect}}}$. Here, $z\in\mathbb R$ is the effective capacitance spectral parameter, giving the frequency prediction $\omega_0+\delta z$; it is distinct from the squared-frequency parameter used earlier. At each selected minimum, the corresponding unit right singular vector $\mathbf v$ satisfies $\| (\mathcal C_{\mathrm{rect}}-zE)\mathbf v\|_2=F(z)$. Thus, small values of $F(z)$ provide finite-matrix residuals for candidate approximate modes, rather than certifying eigenvalues of the infinite system.
The fields can be seen in \Cref{fig:bent_patch_fieldsv} and \Cref{fig:bent_patch_fields} along with the values of $z$ and $F(z).$ Smaller patch sizes are shown for visibility; \Cref{fig:cvgtest} shows stabilization of the selected minimizing values of $z$ and decreasing residuals as $N$ grows, with the material-quadrupole residual reaching a plateau for $N\gtrsim40$. At fixed $N=100$, the results are fairly stable as $W$ varies, except for the material-quadrupole case, whose residual decreases substantially before leveling off. The tests at fixed $W$ measure numerical stabilization of the fixed-window approximation. Note that there exist multiple minima for each bent waveguide; here we choose one for each case and plot its {approximate field and residual}.

\subsubsection{Waveguiding by detuning the material parameters.}
Using the same parameters as in \Cref{linemat}, we can now repeat the patch computations for the bent waveguide.

The results show exponential decay and nearly perfect overlap for both the dipole and quadrupole in \Cref{fig:patch_bentv}.

The approximate fields in \Cref{fig:bent_patch_fieldsv} show localization along the waveguide and clear dipole and quadrupole modes. In the dipole case, the localization is more pronounced along the arms, with small field values near the corner. The quadrupole field is more concentrated near the
corner in the displayed reconstruction. %

\begin{figure}[!h]
        \centering
\begin{subfigure}{0.8\textwidth}
        \includegraphics[width=\linewidth]{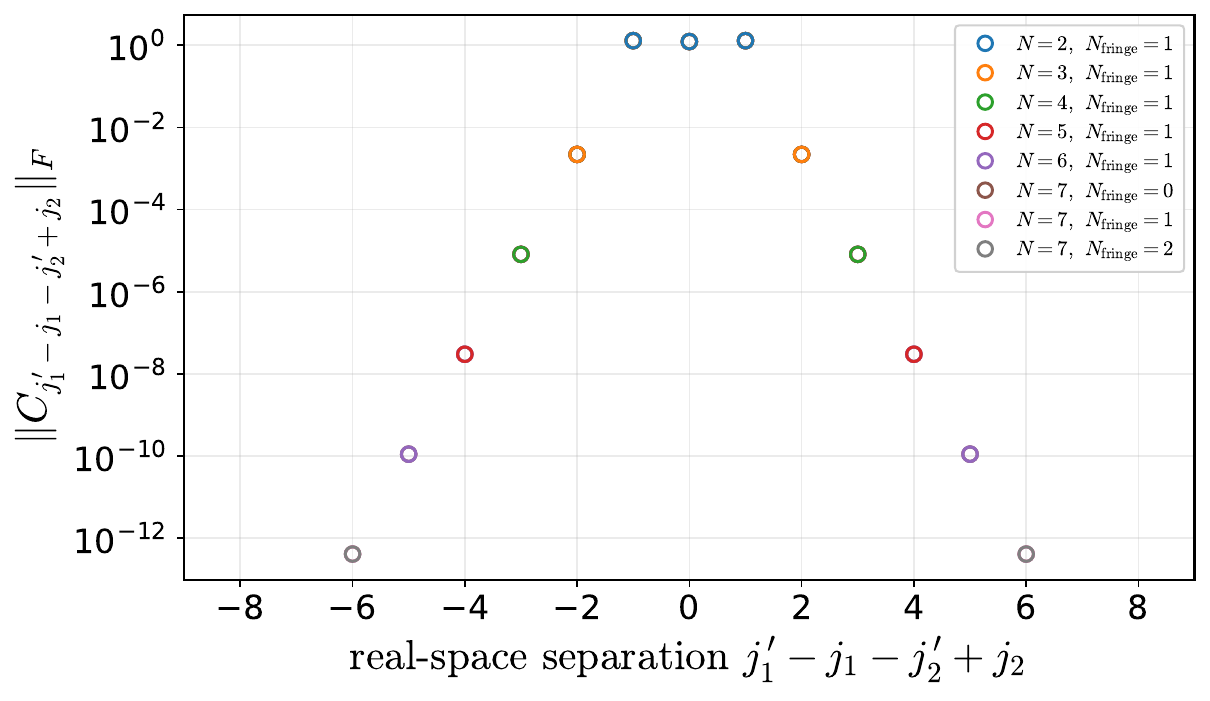}
        \caption{Dipole patch}
        \label{fig:dip_decay5}
        \end{subfigure}
        \begin{subfigure}{0.8\textwidth}
        \includegraphics[width=\linewidth]{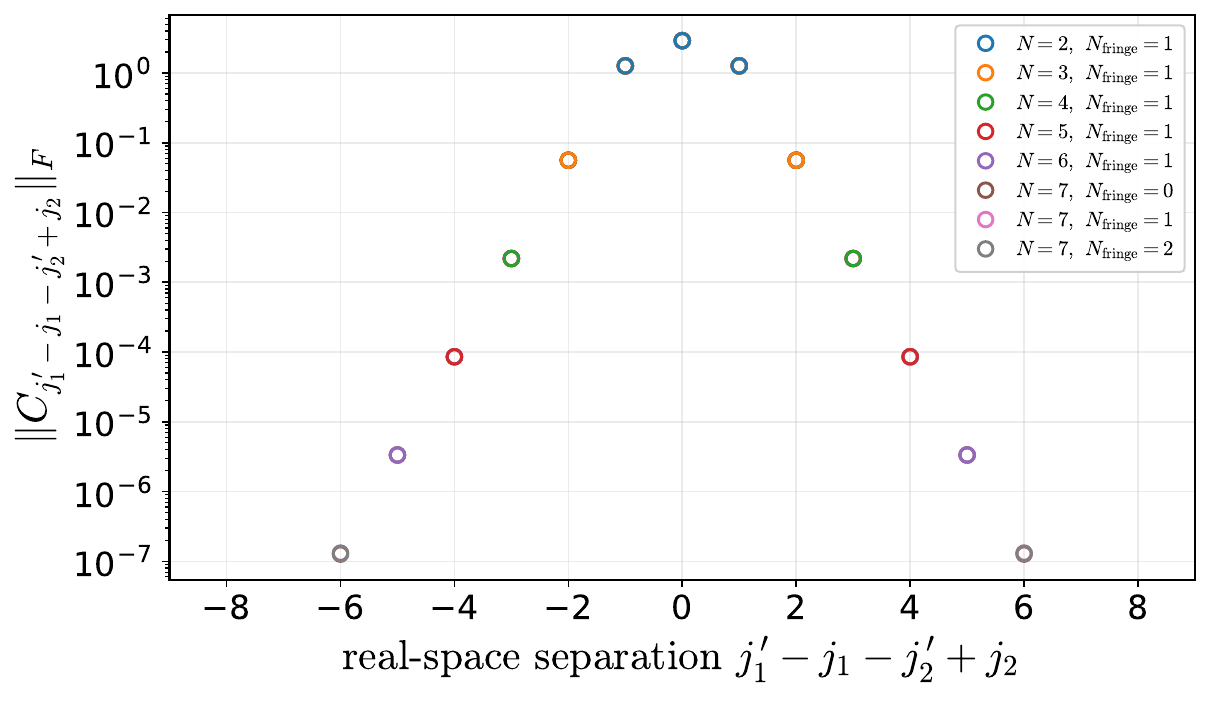}
        \caption{Quadrupole patch}
        \label{fig:quad_decay8}
        \end{subfigure}
    \caption{Capacitance operator entry decay for square patches of bent waveguides of size $(2N + 1)\times(2N + 1)$ with defect resonators of a different material parameter.}
    \label{fig:patch_bentv}
\end{figure}
\begin{figure}[!h]
        \centering
        \begin{subfigure}{0.8\textwidth}
        \includegraphics[width=\linewidth]{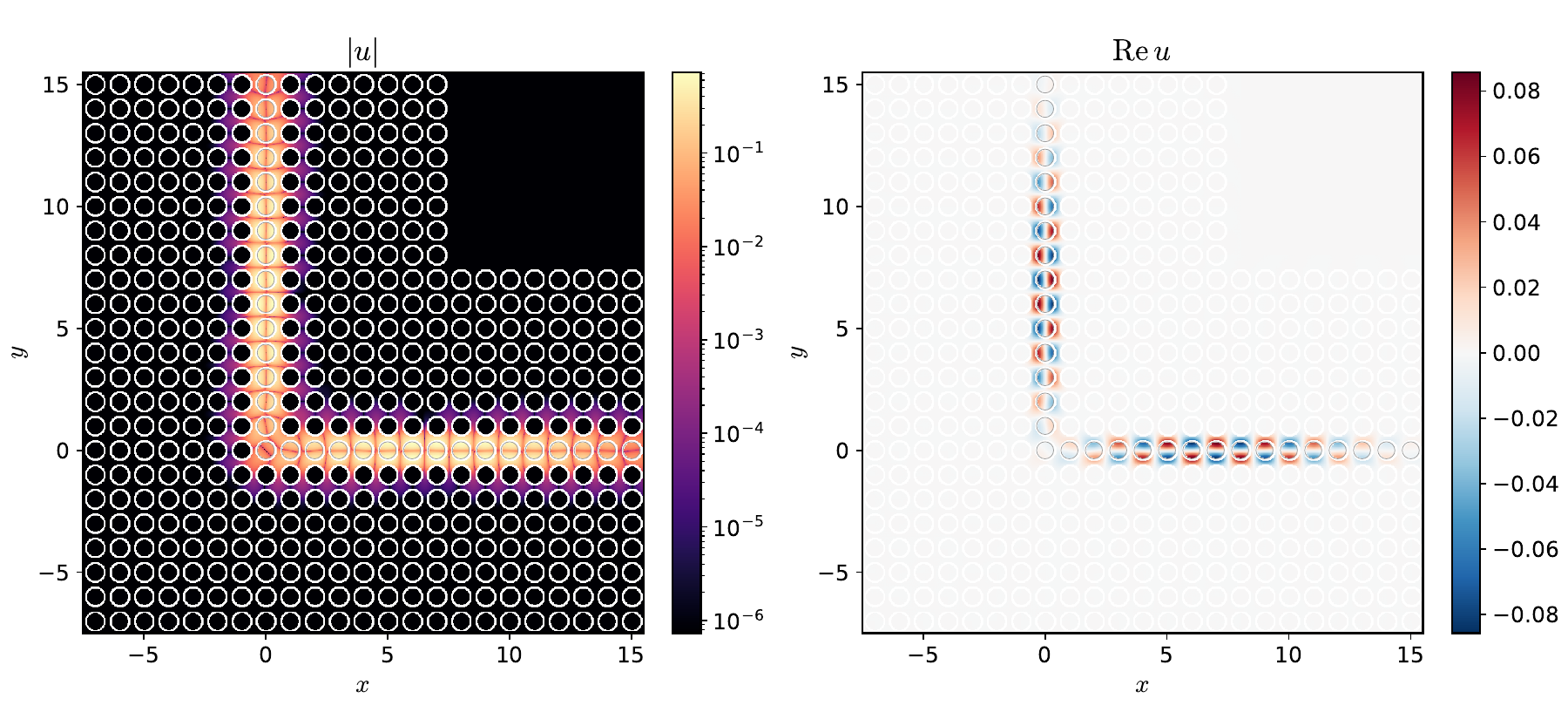}
        \caption{Dipole mode $z=1.7542,~F(z) = 1.8607 \cdot 10^{-2}$} %
        \label{fig:dip_decay6}
        \end{subfigure}
        \begin{subfigure}{0.8\textwidth}
        
        \includegraphics[width=\linewidth]{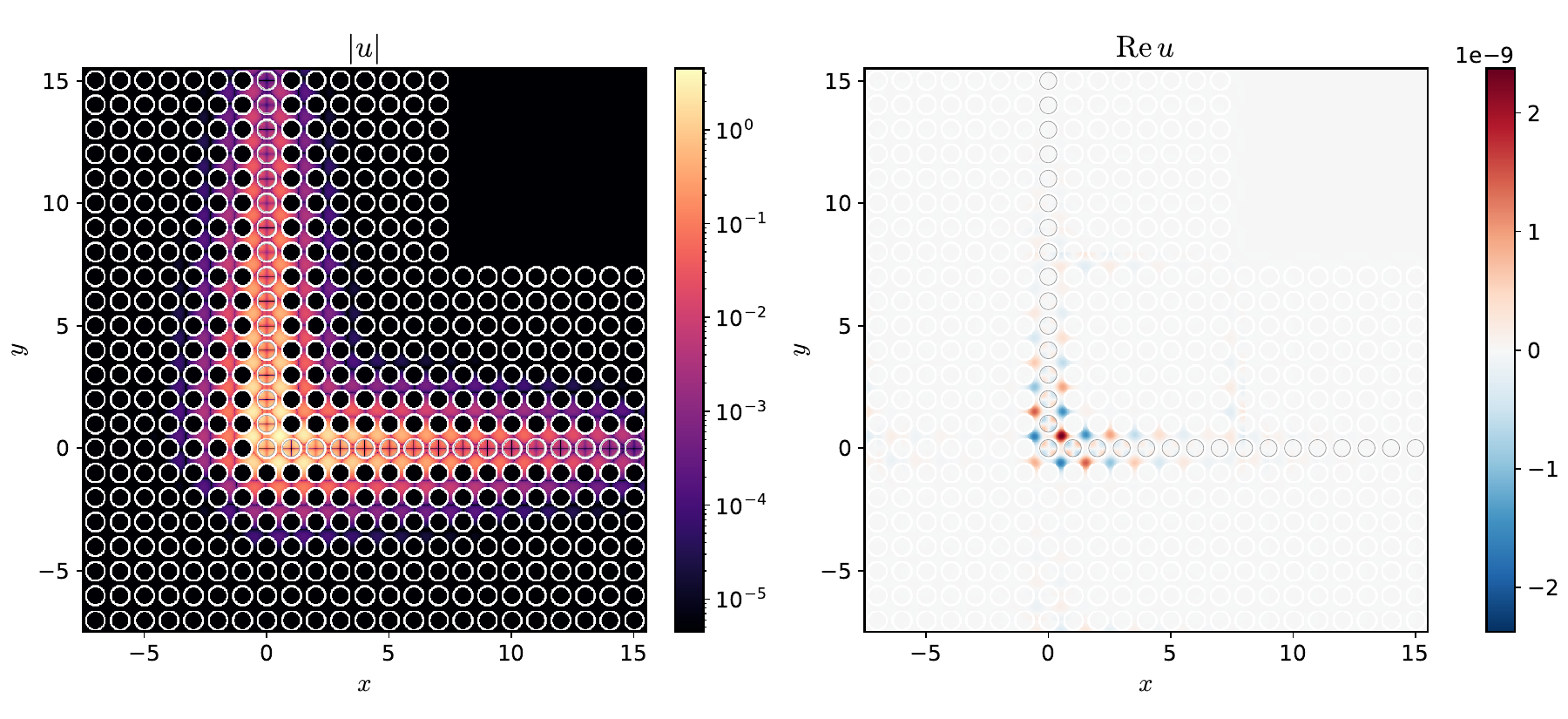}
        \caption{Quadrupole mode $z=5.47128,~F(z)=1.0230 \cdot 10^{-3}$} %
        \label{fig:quad_decay7}
        \end{subfigure}
    \caption{Approximate fields of the dipole and quadrupole modes of the defect resonators with different material parameters. Here, $N = 15,~N_\text{fringe} = 1$, and $W = 6$, with $801 \times 801$ pixels.}
    \label{fig:bent_patch_fieldsv}
\end{figure}

\subsubsection{Waveguiding by detuning the size of resonators.}
Using the same parameters as in \Cref{linerad}, we perform the bent waveguide computations for the defect resonators that are of different radii.

Once again, the results show exponential decay and nearly perfect overlap for both the dipole and the quadrupole in \Cref{fig:patch_bent}. The only exception is \Cref{fig:quad_decay2}, which shows a second decay rate after $10^{-7}$, just as in \Cref{linerad}.

The approximate fields in \Cref{fig:bent_patch_fields} show localization along the waveguide and clear dipole and quadrupole modes, similar to the last subsection. In the dipole case, the localization is more pronounced along the arms, with small field values near the corner, while the quadrupole case also shows corner localization.

\begin{figure}[!h]
\centering
    \begin{subfigure}{0.8\textwidth}
        \includegraphics[width=\linewidth]{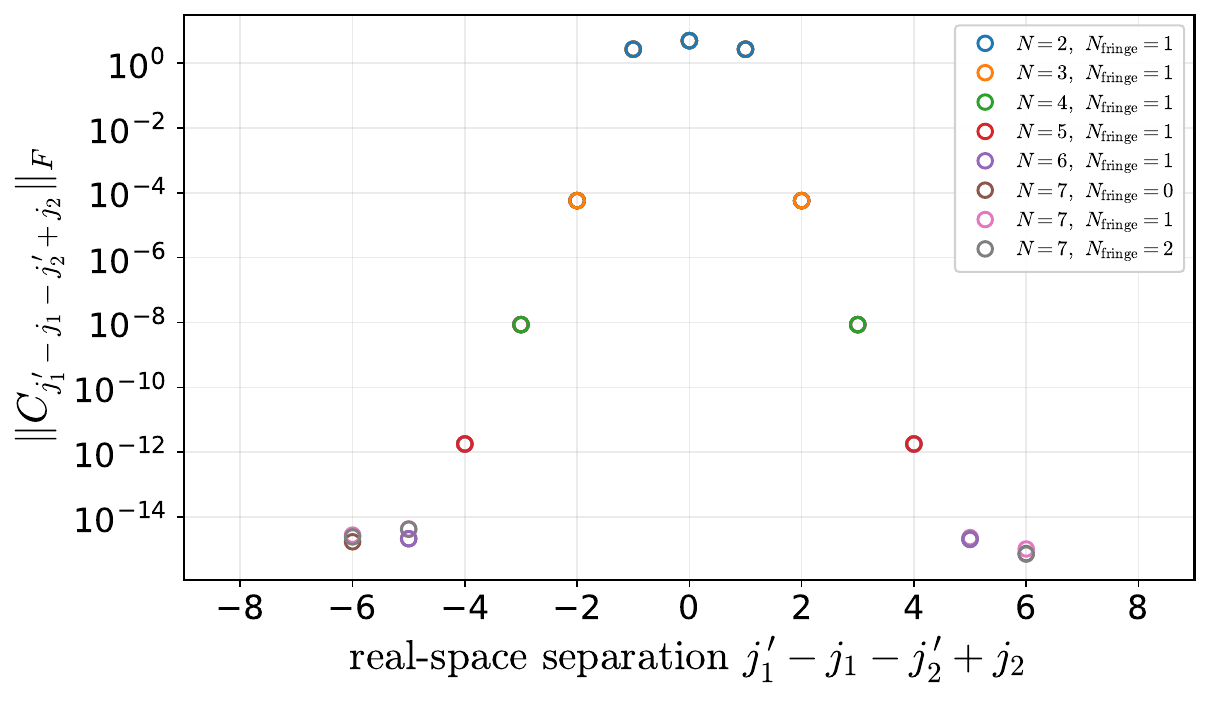}
        \caption{Dipole patch}
        \label{fig:dip_decay2}
        \end{subfigure}
        \begin{subfigure}{0.8\textwidth}
        \includegraphics[width=\linewidth]{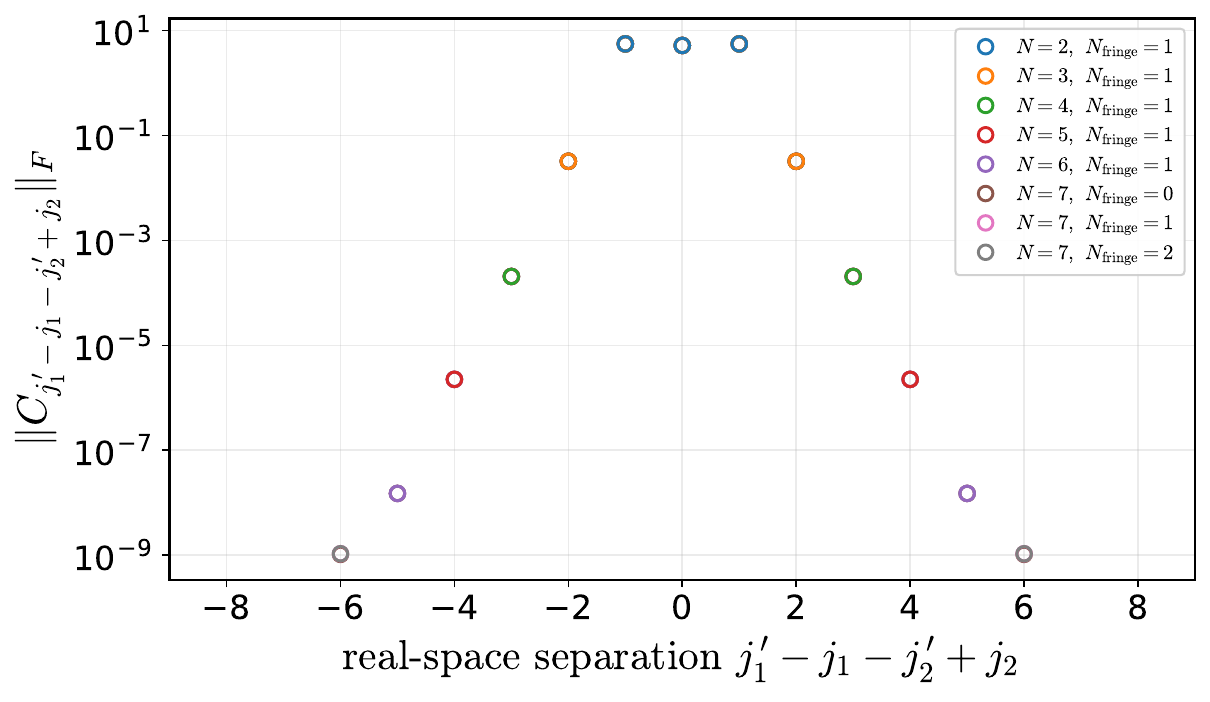}
        \caption{Quadrupole patch}
        \label{fig:quad_decay2}
        \end{subfigure}
    \caption{Capacitance operator entry decay for square patches of bent waveguides of size $(2N + 1)\times(2N + 1)$ with defect resonators of a different radius.}
    \label{fig:patch_bent}
\end{figure}

\begin{figure}[!h]
        \centering
\begin{subfigure}{0.8\textwidth}
        \includegraphics[width=\linewidth]{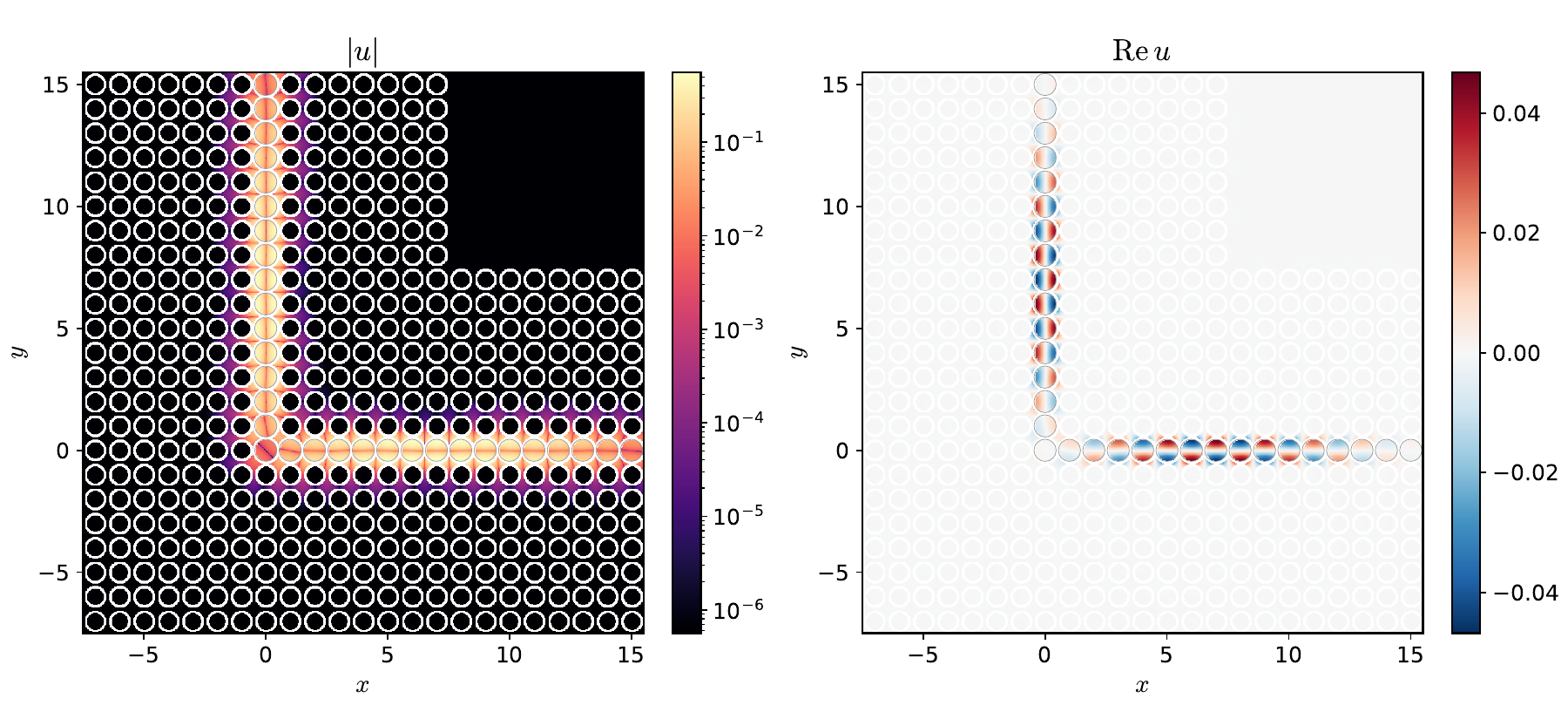}
        \caption{Dipole mode $z = 3.5965,~F(z) = 2.7556 \cdot 10^{-2}$} %
        \label{fig:dip_decay3}
        \end{subfigure}
        \begin{subfigure}{0.8\textwidth}
        
        \includegraphics[width=\linewidth]{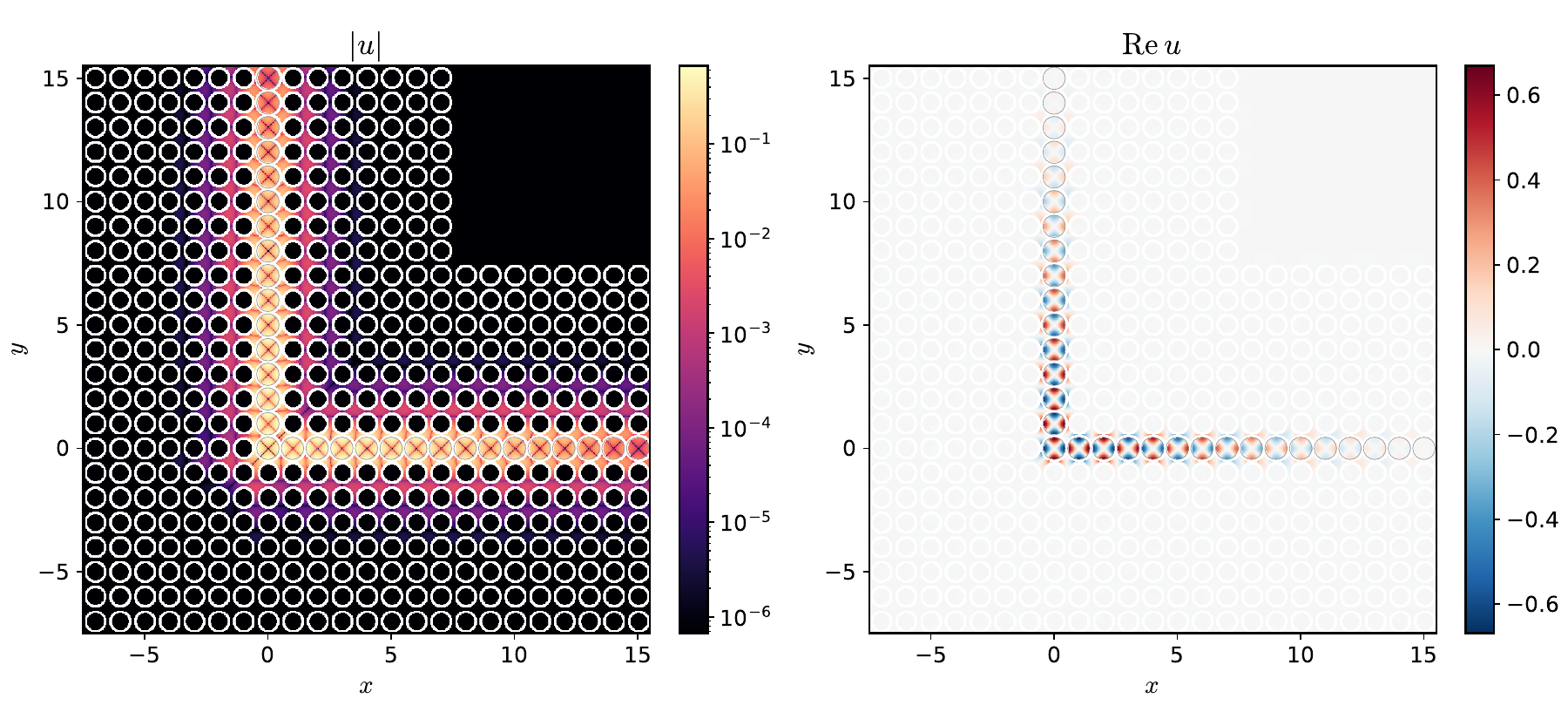}
        \caption{Quadrupole mode $z = 5.6397,~F(z)= 2.6453 \cdot 10^{-2} $} %
        \label{fig:quad_decay3}
        \end{subfigure}
    \caption{Approximate fields of the dipole and quadrupole modes of the defect resonators with different radii. Here, $N = 15,~N_\text{fringe} = 1$, and $W = 6$, with $801 \times 801$ pixels.}
    \label{fig:bent_patch_fields}
\end{figure}

\begin{figure}
    \centering
    \includegraphics[width=\linewidth]{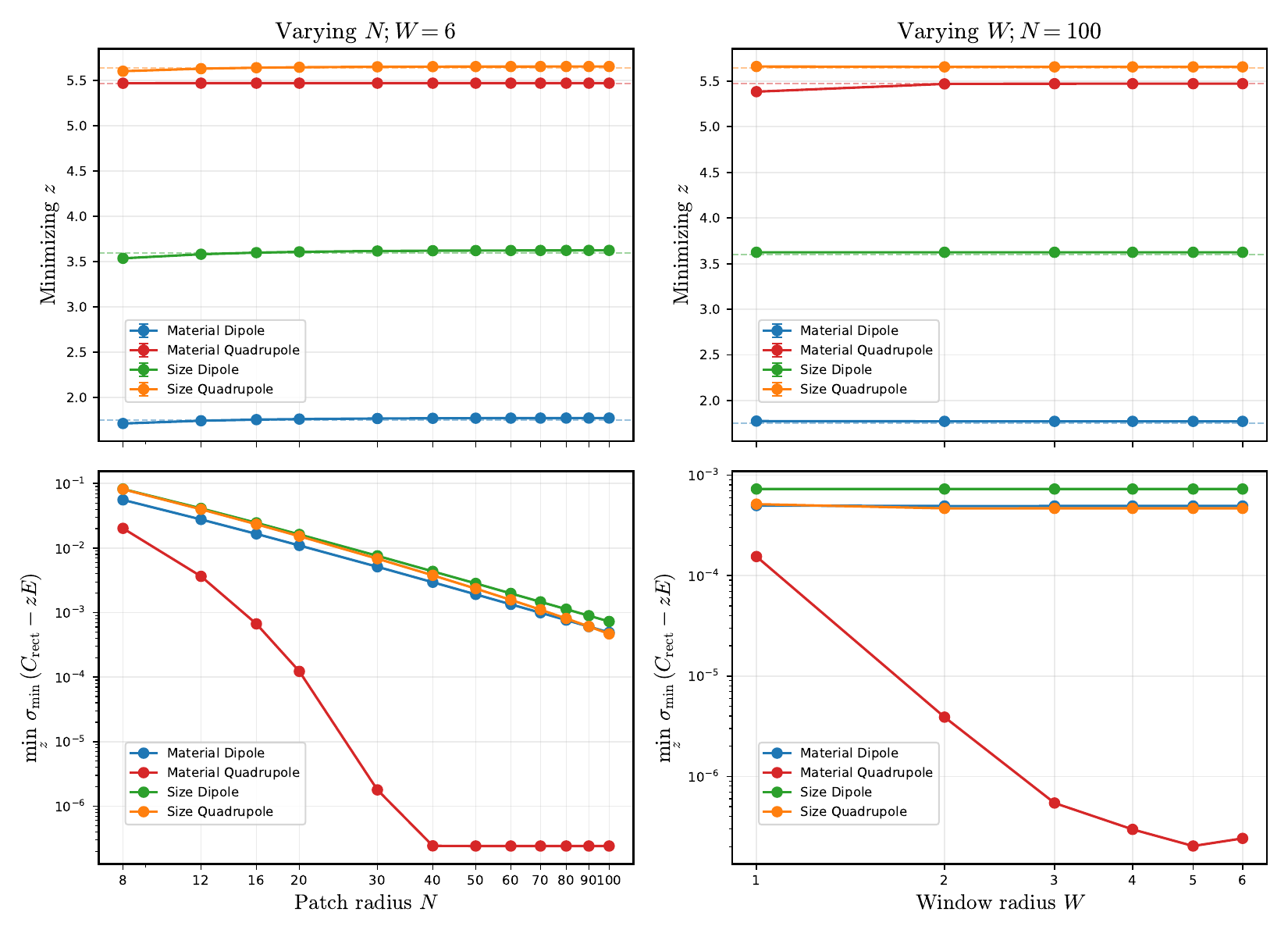}
    \caption{Behavior of the minima of $F(z) = \sigma_\mathrm{min}\left({\mathcal C_{\mathrm{rect}}} - zE\right)$ with varying $N$ with a constant $W = 6$ and varying $W$ with a constant $N = 100.$}
    \label{fig:cvgtest}
\end{figure}

\section{Concluding remarks}
\label{sec:conclusion}
In this work, we have developed a discrete framework for the analysis and computation of guided modes in high-contrast resonator systems near nonzero interior Neumann frequencies. In the regular exterior regime, where the reference frequency is an interior Neumann eigenfrequency but the corresponding exterior wave number squared belongs to the resolvent of minus the exterior Dirichlet Laplacian, the continuous transmission problem reduces, after rescaling around a reference frequency $\omega_0$, to a frequency-dependent capacitance operator acting on the coefficients of the active interior Neumann modes. The norm-resolvent approximation justifies this reduction and shows that compact-defect frequencies and line-defect band functions are determined to first order by the spectrum of the effective operator and its partial Floquet--Bloch transform.

A central feature of the regular regime is the exponential locality of the effective interaction. The exterior spectral gap yields exponential decay of the exterior resolvent and therefore of the off-diagonal frequency-dependent capacitance coefficients. This property allows exponentially accurate interaction truncations and, under a uniform stability assumption for the local Helmholtz problems, an extension of the patch approximation of \cite{ammari2026resolvent_con} beyond the subwavelength regime. The resulting numerical method is particularly well suited to nonperiodic configurations, such as bent waveguides, for which a global Floquet--Bloch reduction does not hold.

The numerical results for dipole and quadrupole resonances illustrate both the accuracy of the first-order frequency-dependent capacitance approximation and the rapid spatial decay of the interaction blocks. They also show that material and geometric detuning can be treated within a common active-mode framework. The finite outgoing computations produce small imaginary parts from outgoing radiation
and, where used, the numerical frequency shifts; these should be distinguished from the real spectrum of the infinite self-adjoint operator in an exterior gap.

Several challenging questions remain open. First, when the reference exterior wavenumber squared belongs to the spectrum of minus the exterior Dirichlet Laplacian, the exterior DtN map is singular, and the leading splitting is expected to be of order $\delta^{1/2}$. It would be interesting to extend the results of \cite{ammariFrequencydependentCapacitanceMatrix2026} in this case to the waveguiding problem with a line defect. Second, another remaining analytical question is to remove or verify, under natural geometric hypotheses, the uniform real-frequency stability assumption \eqref{eq:outgoing-patch-stability} used for the outgoing finite-cluster approximation in free space.  The complex-frequency stabilization developed in this paper provides an unconditional computational alternative, but a direct real-frequency convergence theorem under minimal assumptions would further connect rigorous bounded-patch construction and practical radiation computations. A third direction consists of extending the developed numerical approach to topologically protected interfaces beyond the subwavelength regime. This extension will be the subject of future work.

\begin{code}
The software used to produce the numerical results in this work is openly available at \\ \href{https://github.com/vrabacmath/WaveguideModes}{https://github.com/vrabacmath/WaveguideModes}.
\end{code}

\begin{acknowledgements}
    This work was partially supported by the National Key R\&D Program of China grant number 2024YFA1016000, the City University of Hong Kong start-up fund 7200843, Hong Kong
RGC Early Career Scheme grant 21301326, NSFC grant 12601817, and the Swiss National Science Foundation grant number 200021-236472.
\end{acknowledgements}

\appendix
\section{Existence and characterization of defect and band defect modes}
\label{appendix:proofs}

\subsection{Boundary integral operators}

Following \cite{cbms}, we introduce the following boundary integral formulations of the solutions of \eqref{eq:scattering-problem}. For simplicity, we set $d=2$ and introduce the following layer potentials. For $\omega\in \mathbb{C}$, let
\begin{equation*}
G^{\omega}(x):=-\frac{\mathrm{i}}{4} H^{(1)}_0(\omega|x|), \qquad x\neq 0,
\end{equation*}
be the outgoing Helmholtz Green function in $\mathbb{R}^2$. Here, $H^{(1)}_0$ is the Hankel function of the first kind and order zero.

Next, we introduce the corresponding single-layer potential and Neumann--Poincar\'e operator that are defined by
\begin{equation*}
\mathcal{S}_D^{\omega}[\varphi](x) := \int_{\partial D} G^{\omega}(x-y)\,\varphi(y)\,\mathrm{d}\sigma(y):\ L^2(\partial D) \to H^1_{\rm loc}(\mathbb{R}^2)\,,
\end{equation*}
\begin{equation*}
\mathcal{K}_D^{\omega,*}[\varphi](x) := \int_{\partial D} \frac{\partial}{\partial \nu(x)}G^{\omega}(x-y)\,\varphi(y)\,\mathrm{d}\sigma(y):\ L^2(\partial D) \to L^2(\partial D)\,.
\end{equation*}
Here, $H^1_{\rm loc}(\mathbb{R}^2)$ is the set of functions that together with their weak derivatives are locally square integrable in $\mathbb{R}^2$. With abuse of notation, we use the same notation for the boundary operator $\mathcal{S}_D^{\omega}: L^2(\partial D) \to H^1(\partial D)$, where $H^1(\partial D)$ denotes the space of $L^2$-functions on $\partial D$ with weak $L^2$ tangential derivatives.

We introduce the quasi-periodic Green function $G^{\vect{\alpha},\omega}$, which satisfies
\begin{equation} \label{fGpoisson}
(\Delta + \omega^2) G^{\vect{\alpha},\omega} (x,y) = \sum_{\vect{j} \in \Lambda} \delta_0(x-y - \vect{j}) e^{\i \vect{j} \cdot \vect{\alpha}} \quad \text{in } \mathbb{R}^2,
\end{equation}
where $\delta_0$ denotes the Dirac mass at $0$.

If $\omega^2 \neq |\vect{q} + \vect{\alpha}|^2, \forall\; \vect{q} \in \Lambda^*$, then, using Poisson's summation formula
\begin{equation} \label{poisson}
\frac{1}{|Y|} \sum_{\vect{q} \in \Lambda^*} e^{\i (\vect{q} +\vect{\alpha}) \cdot x} = \sum_{\vect{j} \in \Lambda} \delta_0(x-\vect{j}) e^{\i \vect{j}\cdot \vect{\alpha}},
\end{equation}
where $|Y|$ denotes the volume of $Y$, we can show that the quasi-periodic Green function $G^{\vect{\alpha},\omega}$ can be represented as the sum of augmented plane waves over the reciprocal lattice: \begin{equation} \label{rephelm1} G^{\vect{\alpha},\omega} (x,y) = \frac{1}{|Y|} \sum_{\vect{q}\in \Lambda^*} \frac{e^{\i (\vect{q}+ \vect{\alpha}) \cdot (x-y)}}{\omega^2 - |\vect{q} + \vect{\alpha}|^2};\end{equation} see, for instance, \cite{photonic2018}.

Whenever the quasi-periodic layer-potential representation is used in the following, we assume that $\omega^2 \neq |\vect{q}+ \vect{\alpha}|^2$ for all $\vect{q} \in \Lambda^*$. For $\omega > 0$, we let $\mathcal{S}_D^{\vect{\alpha},\omega}$ and $\mc D_D^{\vect{\alpha},\omega}$ be the quasi-periodic single- and double-layer potentials associated with $G^{\vect{\alpha},\omega}$ on $D$; that is, for a given density $\vp \in L^2(\partial D)$,
\begin{align*}
\mc S_D^{\vect{\alpha}, \omega}[\varphi] (x) & = \int_{\partial D} G^{\vect{\alpha},\omega}(x,y) \varphi (y) \, \mathrm{d} \sigma (y), \quad x \in \mathbb{R}^2, \\
\mc D_D^{\vect{\alpha},\omega} [\varphi] (x) & = \int_{\partial D} \frac{\partial G^{\vect{\alpha},\omega} (x,y)}{\partial \nu(y)} \varphi (y) \, \mathrm{d} \sigma (y), \quad x \in \mathbb{R}^2 \setminus \partial D.
\end{align*}
We also recall the jump relations obeyed by the double-layer potential and by the normal derivative of the single-layer potential on $\partial D$:
\begin{align*}
\frac{\partial (\mc S_D^{\vect{\alpha},\omega} [\varphi])}{\partial \nu} \bigg |_{\pm}(x) & = \bigg (\pm \frac{1}{2} I + \mc K_D^{\vect{\alpha},\omega,*}\bigg )[\varphi] (x), \quad x \in \partial D, \\
(\mc D_D^{\vect{\alpha},\omega}[\varphi]) \bigg|_{\pm} (x) & = \bigg(\mp \frac{1}{2} I + \mc K_D^{\vect{\alpha},\omega} \bigg)[\varphi] (x), \quad x \in \partial D,
\end{align*}
for $\varphi \in L^2(\partial D)$, where $\mc K_D^{\vect{\alpha},\omega}$ is the operator on $L^2(\partial D)$ defined by $$ \mc K_D^{\vect{\alpha},\omega}[\varphi] (x) = \int_{\partial D} \frac{\partial G^{\vect{\alpha},\omega}(x,y)}{\partial \nu(y)} \varphi (y) \, \mathrm{d} \sigma (y) $$ and $\mc K_D^{\vect{\alpha},\omega,*}$ is the corresponding Neumann--Poincar\'e operator, defined directly by $$ \mc K_D^{\vect{\alpha},\omega,*}[\varphi] (x) = \int_{\partial D} \frac{\partial G^{\vect{\alpha},\omega}(x,y)}{\partial \nu(x)} \varphi (y) \, \mathrm{d} \sigma (y); $$ see again \cite{photonic2018}.

\subsection{Characterizations of defect and band defect modes and their associated eigenmodes}

This appendix recalls the two pure detuning cases used in \Cref{sec:numerical}: either the interior wave speed or the disk radius is changed. The simultaneous shape-and-parameter detuning allowed by the abstract framework of \Cref{sec:problem-formulation} is not used in these legacy characteristic-value formulas.

To characterize and approximate the defect and band defect modes and their associated eigenmodes, we introduce the operator $\mathcal{A}^{\vect{\alpha}}(\omega,\delta):L^2(\partial D)\times L^2(\partial D)\to H^1(\partial D)\times L^2(\partial D)$ defined by
\begin{equation}\label{aomegadeltalpha}
\mathcal{A}^{\vect{\alpha}}(\omega,\delta):= \begin{pmatrix} {\mc S}_D^{\omega/v_b} & -\mc S_D^{\vect{\alpha},\omega/v}\\
-\frac{1}{2}I+ {\mc K}_D^{\omega/v_b,*} & - \delta \left(\frac{1}{2}I+\Ka{\vect{\alpha}}{\omega/v}\right)
\end{pmatrix}.
\end{equation}

We start by characterizing defect modes. In exactly the same way as in \cite{cbms,erik2019}, we can show that for a single defect at the origin the defect frequencies that bifurcate from $\omega_0$ are precisely the characteristic values of the operator-valued function
\begin{align} \label{eq:Mdensity}
\omega \mapsto \mathcal{M}(\omega,\delta) :=I+ \frac{|Y|}{(2\pi)^2} \big(\widetilde{\mathcal{A}}(\omega,\delta) - \mathcal{A}(\omega,\delta)\big) \int_{Y^*}\mathcal{A}^{\vect{\alpha}}(\omega,\delta)^{-1} \mathrm{d} \vect{\alpha}
\end{align}
within a band gap containing $\omega_0$. Here, $\widetilde{\mathcal{A}}(\omega,\delta), \mathcal{A}(\omega,\delta)$ are, respectively, defined by
\begin{equation}\label{opAD}
{\mathcal{A}}(\omega,\delta):= \begin{pmatrix} {\mc S}_D^{\omega/v_b} & -\mc S_D^{\omega/v}\\
-\frac{1}{2}I+ {\mc K}_D^{\omega/v_b,*} & - \delta \left(\frac{1}{2}I+\mc{K}_D^{\omega/v,*}\right)
\end{pmatrix}
\end{equation}
and, in the case of a defect in the material parameter,
\begin{equation} \label{deftildea}
\widetilde{\mathcal{A}}(\omega,\delta):= \begin{pmatrix} {\mc S}_D^{\omega/\widetilde{v}_b} & -\mc S_D^{\omega/v}\\
-\frac{1}{2}I+ {\mc K}_D^{\omega/\widetilde{v}_b,*} & - \delta \left(\frac{1}{2}I+\mc{K}_D^{\omega/v,*}\right)
\end{pmatrix}.
\end{equation}
In the case of a disk defect resonator with a detuned radius, instead of \eqref{deftildea}, we let $\widetilde{\mathcal{A}}(\omega,\delta)$ be given by
\begin{equation}\label{eq:ADd}
\widetilde{\mathcal{A}}(\omega,\delta) := (\mathcal{P}_2)^{-1} {\mathcal{A}}(\omega,\delta; \widetilde{D}) \mathcal{P}_1,
\end{equation}
where the operator ${\mathcal{A}}(\omega,\delta; \widetilde{D})$ is defined in \eqref{opAD} with $D$ replaced by $\widetilde{D}$ and 
the radius-transfer operators $\mathcal{P}_1$ and $\mathcal{P}_2$ 
are defined modewise on finite Fourier expansions and act on each Fourier mode as follows: for $a_n,b_n\in\mathbb C$,
\begin{align*}
\mathcal{P}_1\left[\begin{pmatrix} a_n \\ b_n\end{pmatrix}e^{\i n\theta}\right]
&= \frac{r}{\widetilde r}\begin{pmatrix}
\displaystyle \frac{H_n^{(1)}(\omega r/v_b)}{H_n^{(1)}(\omega \widetilde r/v_b)}a_n \\[0.8em]
\displaystyle \frac{J_n(\omega r/v)}{J_n(\omega \widetilde r/v)}b_n
\end{pmatrix}e^{\i n\theta}, \\[1em]
\mathcal{P}_2\left[\begin{pmatrix} a_n \\ b_n\end{pmatrix}e^{\i n\theta}\right]
&= \begin{pmatrix}
\displaystyle \frac{J_n(\omega \widetilde r/v)}{J_n(\omega r/v)}a_n \\[0.8em]
\displaystyle \frac{J_n'(\omega \widetilde r/v)}{J_n^\prime(\omega r/v)}b_n
\end{pmatrix}e^{\i n\theta}.
\end{align*}
Here, $J_n$ and $H_n^{(1)}$ are the Bessel and Hankel functions of order $n$, respectively. The composite in \eqref{eq:ADd} is interpreted through its Fourier
multipliers, cancelling removable factors and using its meromorphic continuation where necessary.

Note that, away from the auxiliary singularities of the layer-potential
representation, $\mathcal{A}^{\vect{\alpha}}$ is invertible for sufficiently small $\delta$ and for $\omega$ within that band gap around $\omega_0$.

In the case of a line defect, the defect band is characterized similarly. From \cite{ammariSubwavelengthGuidedModes2021,cbms}, we can show that for each $\alpha_1$, the $\alpha_1$-quasi-periodic Bloch eigenfrequencies that bifurcate from $\omega_0$ are precisely the characteristic values of the $\alpha_1$-quasi-periodic operator-valued function
\begin{align} \label{eq:MdensityLN}
\omega \mapsto \mathcal{M}(\omega,\delta; \alpha_1) :=I+ \frac{1}{2\pi} \big(\widetilde{\mathcal{A}}(\omega,\delta) - \mathcal{A}(\omega,\delta)\big) \int_{\widetilde{Y}^*}\mathcal{A}^{(\alpha_1,\alpha_2)}(\omega,\delta)^{-1} \mathrm{d} \alpha_2
\end{align}
within the band gap around $\omega_0$.

The eigenmodes associated with the defect and band defect frequencies can be represented in terms of the root functions associated with the characteristic values in \eqref{eq:Mdensity} and \eqref{eq:MdensityLN}, respectively. As in the subwavelength regime, the eigenmodes can be approximated, as $\delta$ tends to zero, in terms of eigenvectors of the frequency-dependent capacitance matrix. The windowed field reconstruction used in the numerical experiments is described in \eqref{formula:eigenmode}.

\end{document}